\documentclass[a4paper]{amsart}
\usepackage[textwidth=15cm,top=3cm, bottom=3cm, hcentering]{geometry}
\usepackage[colorlinks=true,linkcolor=red,citecolor=blue]{hyperref}

\usepackage{amsmath,amscd,amsthm,amssymb,mathtools}
\usepackage{color}
\usepackage[utf8]{inputenc}
\normalfont
\usepackage[T1]{fontenc}
\usepackage{tikz-cd}
\usepackage{tikz}
\usepackage{enumerate}
\usepackage{comment}
\usepackage{mathabx}
 \usepackage[all]{xy}

\usetikzlibrary{decorations.pathmorphing,shapes}
\usetikzlibrary{arrows.meta}
\tikzcdset{arrow style=tikz,
    squigarrow/.style={
        decoration={
        snake, 
        amplitude=.4mm,
        segment length=2mm
        }, 
        rounded corners=.2pt,
        decorate
        }
    }

\newtheorem{theo}{Theorem}[section]
\newtheorem*{theo*}{Theorem}
\newtheorem{lemm}[theo]{Lemma}
\newtheorem{prop}[theo]{Proposition}
\newtheorem{coro}[theo]{Corollary}

\theoremstyle{definition}
\newtheorem{defi}[theo]{Definition}
\newtheorem{rema}[theo]{Remark}

\newtheorem{exam}[theo]{Example}

\theoremstyle{theorem}

\newcommand{\CC}{\mathbb{C}}

\newcommand{\NN}{\mathbb{N}}

\newcommand{\PP}{\mathbb{P}}
\newcommand{\QQ}{\mathbb{Q}}
\newcommand{\RR}{\mathbb{R}}
\renewcommand{\SS}{\mathbb{S}}

\newcommand{\ZZ}{\mathbb{Z}}
\newcommand{\Aa}{\mathcal{A}}

\newcommand{\Cc}{\mathcal{C}}
\newcommand{\Dd}{\mathcal{D}}

\newcommand{\Hh}{\mathcal{H}}

\newcommand{\Ll}{\mathcal{L}}
\newcommand{\Mm}{\mathcal{M}}

\newcommand{\Pp}{\mathcal{P}}

\newcommand{\Ho}{\mathrm{Ho}}
\newcommand{\Ker}{\mathrm{Ker}}
\newcommand{\Coker}{\mathrm{Coker}}

\newcommand{\Img}{\mathrm{Im}}

\newcommand{\ov}{\overline}

\newcommand{\del}{\partial}
\newcommand{\delb}{{\overline\partial}}

\newcommand{\ch}[1]{\mathrm{Ch^*({#1})}}

\newcommand{\antishriek}{\text{\raisebox{\depth}{\textexclamdown}}}
\newcommand{\MHD}{{\mathsf{MHD}}}

\newcommand{\MHS}{{\mathsf{MHS}}}
\newcommand{\MHC}{{\mathsf{MHC}}}

\newcommand{\Hom}{\mathrm{Hom}}

\newcommand{\End}{\mathrm{End}}

\title[Mixed Hodge Formality]{Mixed Hodge Formality}

\author{Pedro Magalhães}

\address[P. Magalhães]{Departament de Matemàtiques i Informàtica, Universitat de Barcelona\\
Gran Via 585}
\email{pedrommagalhaes1998@gmail.com}

\thanks{ This work was supported by FCT - Fundação para a Ciência e Tecnologia, I.P. by project 2021.06151.BD with DOI identifier https://doi.org/10.54499/2021.06151.BD.
Partial financial support from the Spanish State Research Agency through projects PID2020-117971GB-C22, PID2024-155646NB-I00 and EUR2023-143450.
}

\begin{document}

\maketitle 

\begin{abstract}
  We introduce the notion of \textit{mixed Hodge formality}, which refines classical formality and takes into account the mixed Hodge structures present on the cohomology of complex algebraic varieties. We develop an obstruction theory for mixed Hodge formality, witnessing the non-triviality of extensions of mixed Hodge structures. This allows us to understand the non-formality of certain compact Kähler manifolds in the mixed Hodge sense.
\end{abstract}

\tableofcontents

\newpage

\section{Introduction}
An important result at the interface of homotopy theory and complex geometry is the Formality Theorem of compact Kähler manifolds, established by Deligne, Griffiths, Morgan, and Sullivan in
 \cite{DGMS}. It states that the commutative differential graded algebra (cdga for short) of piece-wise linear forms of such manifolds is quasi-isomorphic to its cohomology as a cdga with trivial differential. As a consequence, for a simply connected compact Kähler manifold $X$, the rational homotopy groups $\pi_*(X) \otimes \QQ$ are entirely determined by the cohomology ring $H^*(X;\QQ)$. It is a classical fact that each cohomology group $H^n(X;\QQ)$ has a pure Hodge structure of weight $n$. On the other hand, by work of Morgan \cite{morgan}, the rational homotopy groups of $X$ carry functorial mixed Hodge structures. The category $\MHS$ of mixed Hodge structures was introduced by Deligne \cite{DeHII} in order to extend the presence of Hodge structures in cohomology to the case of complex algebraic varieties. It is an abelian category with non-trivial extensions and also arises naturally when considering algebraic models of  both complex algebraic varieties and of compact Kähler manifolds.
 In \cite{CCM}, Carlson, Clemens and Morgan give examples of smooth complex projective varieties which are diffeomorphic and have the same pure Hodge structure on cohomology but have distinct mixed Hodge structures on $\pi_3 \otimes \QQ$. They do so by showing that, associated to the mixed Hodge structure on $\pi_3(X) \otimes \QQ$ for $X$ a complex smooth projective variety, there is an invariant
\[
  u(\pi_3^*) \in \mathrm{Ext}_{\MHS}^1(\Ker(\mu),H^3(X;\QQ)),
\]
where $\mu: H^2(X) \otimes H^2(X) \to H^4(X)$ is the cup product. They then show that $u(\pi_3^*)$
differs across varieties of the same family.

This hints to the fact that, although all compact Kähler manifolds are formal, not all such manifolds should be formal in some stronger sense that takes into account the Hodge structures present in their rational homotopy type. The aim of this paper is to define and study this stronger notion of formality, which we call \textit{mixed Hodge formality}. This notion is also related to formality considerations in the motivic sense (see for instance \cite{Iwanari}).

Mixed Hodge structures are present functorially on the cohomology of complex algebraic varieties \cite{DeHII}, \cite{DeHIII} and on their rational homotopy type \cite{morgan}, \cite{navarro}, \cite{Hain}.
In particular, by \cite{navarro}, there exists a functor from the category complex algebraic varieties $\mathrm{Var}_\CC$ to the category of \textit{mixed Hodge diagrams}, localized at quasi-isomorphisms:
\[ \Aa : \mathrm{Var}_\CC \to \mathsf{MHD}[\mathrm{Qiso}^{-1}]. \]
Mixed Hodge diagrams are, roughly, enriched versions of commutative differential graded algebras  which carry mixed Hodge structures on cohomology.
The composition of this functor with the forgetful functor  $\mathsf{MHD} \to \mathsf{cdga}$ recovers Sullivan's functor $\Aa_{pl}$ of piecewise linear forms. It therefore computes the mixed Hodge structures on cohomology and on the rational homotopy type of complex algebraic varieties.

We say that $X \in \mathrm{Var}_\CC$ is \textit{mixed Hodge formal} if $\Aa(X) \cong H^*(X)$ in the category $\mathsf{MHD}[\mathrm{Qiso}^{-1}]$. We shall see that the examples in \cite{CCM} mentioned above are not mixed Hodge formal.
We study, more generally, mixed Hodge formality over any algebraic (Koszul) operad. This allows us to use the language of infinity algebras and of deformation theory of differential graded Lie algebras to prove our results. The language of operads also allows us to compare the notions of mixed Hodge formality over an operad and its dual.

A useful approach to study formality is that of cohomological obstructions. In the setting of rational homotopy, Halperin and Stasheff \cite{HalpStash} developed an obstruction theory to decide when an isomorphism of cohomology algebras $f: H^*(X;\QQ) \to H^*(Y;\QQ)$ can be realized by a rational homotopy equivalence. In the same vein, but using the language of infinity algebras, Saleh \cite{Bashar} constructed successively defined obstructions to formality of algebras over one of the operads $\Pp = Ass, Com, Lie$ over a field of characteristic $0$. The obstructions to formality of a $\Pp$-algebra $A$ live in the operadic cohomology groups $\Pp H^*(H^*(A))$.

Denote by $\Pp\textrm{-}\mathsf{MHD}$ the category of mixed Hodge diagrams over an operad $\Pp$.
Since $\mathrm{Ext}^2_{\MHS}$ is trivial,
given any $A \in \Pp\textrm{-}\mathsf{MHD}$, there is always a zig-zag of quasi-isomorphisms between $A$ and $H^*(A)$ at the additive level of cochain complexes of mixed Hodge structures.
Note, however, that the morphisms in the zig-zag may not be compatible with the $\Pp$-algebra structure. The main theorem of this paper, Theorem \ref{obst-theo}, gives successively defined obstructions to mixed Hodge formality of $\Pp$-mixed Hodge diagrams. For that, we introduce a version of operadic cohomology adapted to mixed Hodge objects, which is related to Beilinson's absolute Hodge cohomology \cite{Beilinson}. These cohomology groups are bigraded and we denote them by $\Pp H_{\mathrm{DB}}^{*,*}(H^*(A))$, where the subscript $\mathrm{DB}$ stands for Deligne-Beilinson. We prove:

\newtheorem*{intro1}{\normalfont\bfseries Theorem $\textbf{\ref{obst-theo}}$}
\begin{intro1}
  Let $\Pp = Ass, Com$ or $Lie$. Given $A$ a $\Pp$-mixed Hodge diagram, there exist successively defined classes 
  \begin{equation*}
    \theta_k \in \Pp H_{\mathrm{DB}}^{k,2-k}(H^*(A)), \mathrm{\qquad} k \geq 3
  \end{equation*}
  such that if all classes are trivial and $A$ is mixed Hodge formal.
\end{intro1}
We then prove several implications related to mixed Hodge formality:
\begin{enumerate}
 \item When $X$ is simply connected, we show that there is a well-defined map
    \[ R : \Pp H_{\mathrm{DB}}^{3,-1}(H^*(X)) \to \mathrm{Ext}_{\MHS}^1(\Ker(\mu),H^3(X)), \]
    which maps the first obstruction $\theta_3$ to Carlson, Clemens and Morgan's invariant $u(\pi_3^*)$ (Corollary \ref{ccmobs-coro}).
    This identifies the lack of mixed Hodge formality in all the examples of \cite{CCM}.
 \item Given a simply connected complex algebraic variety $X$, Quillen's $Lie$-model of $X$ admits a refinement to $Lie$-mixed Hodge diagrams, and we define mixed Hodge coformality of $X$ as formality of such a refinement. We show that mixed Hodge formality is equivalent to mixed Hodge coformality when a complex algebraic variety has a cohomology algebra satisfying the following properties. It is Koszul and generated by elements in some degree $\geq 2$ and has pure Hodge structures (Theorem \ref{form-coform-theo2}). This applies, for instance, to the compactification $\overline{\Mm}_{0,n}$ of the moduli spaces of genus $0$ curves with $n$ marked points, to the configuration spaces $F_k(\CC^n)$ of $k$ points in $\CC^n$ and to simply connected complex surfaces. 
\item We provide examples of complex algebraic varieties that are mixed Hodge formal. For instance, this is the case for the configuration spaces $F_k(\CC^n)$ of $k$ points in $\CC^n$ for $k < 2n$ (Proposition \ref{confspaceform-prop}) and homogeneous compact Kähler manifolds (Proposition \ref{homogKahler-prop}).
\item We relate the obstructions to mixed Hodge formality to ABC-Massey products, defined for complex manifolds \cite{DaniToma} and, more generally, for any bidifferential bigraded algebra, and give a further example of a compact Kähler manifold that is not mixed Hodge formal (Theorem \ref{obst3=abc-theo}). 
\item We give an example illustrating the fact that mixed Hodge formality does not satisfy descent with respect to field extensions. This contrasts sharply with the behaviour of formality over fields of characteristic zero (Example \ref{exam-fieldext}).
\end{enumerate}

\medskip

\noindent
\textbf{Contents and organization of the paper.} We begin with a preliminary section in which we review results from the classical theory of algebraic operads adapted to the setting of filtered and bifiltered complexes.

In section \ref{mhs-sec} we study the homotopy theory of operadic mixed Hodge diagrams. Specifically, we define operadic mixed Hodge diagrams over any algebraic (co)operad and in Theorem \ref{mhd-barcobar}, give a bar-cobar adjunction
  \[ \Omega_\kappa : \Pp^\antishriek\textrm{-}\mathsf{MHD} \rightleftharpoons \Pp\textrm{-}\mathsf{MHD} : \mathrm{B}_\kappa, \]
where $\Pp\textrm{-}\MHD$ denotes the category of mixed Hodge diagrams over an operad $\Pp$ and likewise for its Koszul dual cooperad $\Pp^\antishriek$. We then introduce an infinity version of operadic mixed Hodge diagrams $\Pp_\infty\textrm{-}\MHD$, whose homotopy category models the homotopy category of $\Pp\textrm{-}\MHD$. In Theorem \ref{htt-mhd}, we prove a Homotopy Transfer Theorem in the mixed Hodge setting and use it show that there exists a zig-zag of quasi-isomorphisms between two $\Pp$-mixed Hodge diagrams
\[ A \xleftarrow{\sim} \cdots \xrightarrow{\sim} A' \]
if and only if there exists what we call a ho-$\infty$-quasi-isomorphism from $A$ to $A'$ (Proposition \ref{zig-zag-equiv}). 

In section \ref{sec-main}, we define mixed Hodge formality, compare formality of an operad and its Koszul dual, relate formality with splitting of certain mixed Hodge structures and prove our main theorem, Theorem \ref{obst-theo}, constructing a successively defined sequence of cohomological obstructions to formality of mixed Hodge diagrams. We give a sketch of the proof. 

Let $(A,\mu) \in \Pp\textrm{-}\MHD$, where $\mu$ denotes its $\Pp$-algebra structure and $\Pp$ is either $Ass$, $Com$ or $Lie$. As a consequence of the homotopy transfer theorem (Theorem \ref{htt-mhd}), there is a model 
\[(H^*(A),m) \in \Pp_\infty\textrm{-}\MHD \]
of $A$. The components of $(H^*(A),m)$ consist of 
\begin{enumerate}
  \item A filtered $\Pp_\infty$-algebra with $\Bbbk$ coefficients $(H^*(A_\Bbbk),W,m_\Bbbk)$, 
  \item A bifiltered $\Pp_\infty$-algebra with $\CC$ coefficients $(H^*(A_\CC),W,F,m_\CC)$, 
  \item A filtered $\infty$-isomorphism 
  \[ \varphi : (H^*(A_\Bbbk),W,m_\Bbbk) \otimes \CC \to (H^*(A_\CC),W,m_\CC). \]
\end{enumerate}
The $\Pp_\infty$-structures $m_\Bbbk$ and $m_\CC$ are encoded by maps 
\[ (m_\Bbbk)_n : H^*(A_\Bbbk)^{\otimes n} \to H^*(A_\Bbbk), \quad (m_\CC)_n : H^*(A_\CC)^{\otimes n} \to H^*(A_\CC) \]
of degree $2-n$ for $n \geq 2$, satisfying certain relations. In particular, $(m_\Bbbk)_2$ and $(m_\CC)_2$ are the induced product $\mu^*$. The $\infty$-morphism $\varphi$ consists of maps 
\[ \varphi_n : (H^*(A_\Bbbk) \otimes \CC)^{\otimes n} \to H^*(A_\CC) \]
of degree $1-n$ for $n \geq 1$, satisfying certain relations. As a consequence of the results of the previous chapter, $A$ is formal if and only if there is a ho-$\infty$-isotopy
\[ (H^*(A),m) \to (H^*(A),\mu^*). \] 
A ho-$\infty$-isotopy is composed of level-wise $\infty$-isotopies $f_\Bbbk$ and $f_\CC$ and a homotopy $h$ making the following diagram commute up to homotopy:
\[ \begin{tikzcd}
    (H^*(A_\Bbbk),m_\Bbbk) \otimes \CC \arrow[r,"f_\Bbbk \otimes \CC"] \arrow[d,"\varphi"] \arrow[dr, "h",Rightarrow] & (H^*(A_\Bbbk), \mu^*) \otimes \CC \arrow[d,"\varphi_1"] \\
    (H^*(A_\CC),m_\CC) \arrow[r,"f_\CC"] & (H^*(A_\CC), \mu^*).
\end{tikzcd}\]

Let $n \geq 3$ be the first value for which the three morphisms
\[ ((m_\Bbbk)_n,(m_\CC)_n,\varphi_{n-1}) \]
are non-trivial. Then this triple yields a class in $\Pp H_{\mathrm{DB}}^{n,2-n}(H^*(A))$, the obstruction $\theta_n$. We show that if this obstruction vanishes, then there is a ho-$\infty$-isotopy from $(H^*(A),m)$ to another $\Pp_\infty$-mixed Hodge diagram for which 
\[ ((m_\Bbbk)_n,(m_\CC)_n,\varphi_{n-1}) = 0. \]
The vanishing of all obstructions implies that there is a ho-$\infty$-isotopy from the model $(H^*(A),m)$ to $(H^*(A),\mu^*)$ and proves the theorem.

In the same section we give a description for part of the first obstruction (Proposition \ref{first-obs}) and study the case where the cohomology carries pure Hodge structures. In this setting, we prove that formality of $A \in \Pp\textrm{-}\MHD$ is equivalent to $A$ being quasi-isomorphic to a $\Pp$-algebra in mixed Hodge structures which is split as a graded mixed Hodge structure (Proposition \ref{inter-obs}) and show that under further conditions on the cohomology (being Koszul and generated in a fixed degree), $A$ is formal if and only if its Koszul dual $\mathrm{B}_\kappa(A)$ is formal (Proposition \ref{form-coform-prop}).

The last section is devoted to geometric applications.

\medskip

\noindent
\textbf{Acknowledgments.} The author would like to thank their supervisor Joana Cirici for suggesting this project and for many useful discussions. The author would also like to thank Geoffroy Horel for many insights and, in particular, for the statement of Proposition \ref{koszul-alpha-prop}.

\section{Filtered operadic algebras}
\label{filalg-sec}
We extend results from the classical theory of algebraic operads to the filtered setting. We start by defining filtered operadic algebras and promote the bar and cobar functors to filtered algebras. We define the notion of filtered infinity algebras and morphisms and state a homotopy transfer theorem for $d$-strict filtered algebras. We end with analogous results in the bifiltered setting.

\subsection{Filtered bar-cobar adjunction}
\label{subsec-filalg}
Throughout, we will fix a field $\Bbbk$ of characteristic $0$ and assume that every operad $\Pp$ is connected, by which we mean that $\Pp(0) = 0$ and $\Pp(1) = \Bbbk$ . We use the cohomological convention, so all differentials increase degree by $1$ and the suspension and desuspension of a cochain complex $A$ are
\[ (sA)^n = A^{n+1} \quad \text{ and }\quad (s^{-1}A)^n = A^{n-1}. \]
We will also work with filtered complexes $(A,W)$ over $\Bbbk$ with exhaustive and Hausdorff filtrations. That is,
\begin{equation}
  \label{exhauHausfil-eq}
   \bigcup_{k} W_k A = A, \qquad \bigcap_k W_k A = 0, \qquad \bigcap_k W_k H^*(A) = 0.
\end{equation}
Here, the filtration $W$ of $H^*(A)$ is the  filtration induced on cohomology by:
\[ W_k H^*A:= \Img( H^*(W_k A) \to H^*(A)). \]
Denote by $\mathsf{F}\ch{\Bbbk}$ the category of filtered cochain complexes. It is closed symmetric monoidal with tensor product and inner Hom defined by
\begin{align}
  \label{filttens-eq}
  &W_k(A \otimes B) := \sum_{i+j = k} W_i A \otimes W_j B, \\ 
  &W_k \underline{Hom}_\Bbbk(A,B) := \{ f \in  \underline{Hom}_\Bbbk(A,B) \, | \, f(W_n A) \subset W_{n+k} B, \, \text{for all } n \in \ZZ \}. \nonumber
\end{align} 
Consider, for each $k \in \ZZ$, the \textit{$k$-th associated graded} functor
\begin{align*}
  Gr_k : \mathsf{F}\ch{\Bbbk} &\to \ch{\Bbbk} \\
  (A,W) &\mapsto Gr^W_k(A) := \frac{W_k A}{W_{k-1}A}
\end{align*} 
The \textit{associated graded} functor obtained by summing $Gr_k^W$ for all $k$
\[ Gr^W(A) = \bigoplus_k Gr^W_k(A) \]
is symmetric monoidal. We will make extensive use of this fact in the present section.

Given a filtered complex $(A,W)$, consider the graded endomorphism operad $\mathrm{End}_{(A,W)}$ defined by
\begin{align*}
  \mathrm{End}_{(A,W)}(n) = \underline{Hom}_\Bbbk((A,W)^{\otimes n},(A,W)),
\end{align*}
together with the usual composition and unit maps (see, for instance, section 5.2 of \cite{LV}). 
\begin{defi}
  \label{fil-alg-defi}
  For $\Pp$ an operad, a \textit{filtered $\Pp$-algebra} is a filtered complex $(A,W)$ together with a morphism of operads $\Pp \to \mathrm{End}_{(A,W)}$. 
\end{defi}
Given a filtered cochain complex $(A,W)$, define a filtration on the free $\Pp$-algebra $\Pp(A)$ by extending the filtration on $A$ to each tensor product
 \begin{equation}
    \label{freealgfilt-eq}
    W_k (\Pp(n) \otimes_{\SS^n} A^{\otimes n}) = \bigoplus_{i_1 + ... + i_n = k} \Pp(n) \otimes_{\SS^n} W_{i_1}A \otimes \cdots \otimes W_{i_n}A,
 \end{equation} 
 where $\Pp(n)$ is given the trivial filtration concentrated in weight $0$.
\begin{rema}
  Equivalently, a filtered $\Pp$-algebra is given by a filtered complex $(A,W)$ together with a map
  \[ \gamma : (\Pp(A),W) = \bigoplus_{n} \Pp(n) \otimes_{\SS_n} (A^{\otimes n},W) \to (A,W) \]
  that satisfies the conditions of a $\Pp$-algebra and preserves the filtrations. This definition of filtered algebra coincides with the notions of filtered algebra in \cite{mc-meth-twist} and \cite{ChatCi} where $\Pp$ is seen as a filtered operad with trivial filtration concentrated in weight $0$.
\end{rema}
We denote by $\Pp$-$\mathsf{Falg}$ the category of filtered $\Pp$-algebras. Dually, given a filtered complex $(C,W)$, extend the filtration to
\[ \hat{\Cc}(C) = \prod_n (\Cc(n) \otimes (C,W)^{\otimes n})^{\SS_n} \]
in a similar way to (\ref{freealgfilt-eq}).
\begin{defi}
  For $\Cc$ a cooperad, a \textit{filtered $\Cc$-coalgebra} is a filtered complex $(C,W)$ together with a map
  \[ \Delta : (C,W) \to (\hat{\Cc}(C),W), \]
  satisfying the conditions of $\Cc$-coalgebra (see section $5.7$ of \cite{LV}) and preserving filtrations.
\end{defi}
\begin{rema}
  In order not to deal with infinite products, we assume in the rest of this work that all cooperads and coalgebras are conilpotent, as defined, for instance, in section $5.7$ of \cite{LV}. In this case, the map
  $\Delta : C \to \hat{\Cc}(C)$
  factors through $\Cc(C) = \bigoplus_n (\Cc(n) \otimes (C,W)^{\otimes n})^{\SS_n}$.
\end{rema}
We assume also that all cooperads $\Cc$ are connected $(\Cc(1) = \Bbbk)$ and reduced $(\Cc(0) = 0)$. In this case, there is a canonical filtration $R$ on a $\Cc$-coalgebra, called the \textit{coradical} filtration. It is given by
\[ R_0C = 0, \quad R_n(C) = \{ x \in C \, | \Delta_C(x)_k = 0, \text{ for } k > n \}, \] 
for $n \geq 1$ and where $\Delta_C(x)_k$ is the $k$-th component of 
\[ \Delta_C(x) \in \prod_{n \geq 1} (\Cc(n) \otimes C^{\otimes n})^{\SS_n}. \]
The condition of being conilpotent is equivalent to $R$ being exhaustive:
\[ C = \mathrm{colim}_n \, R_n(C). \]
Denote by $\Cc$-$\mathsf{Fcoalg}$ the category of filtered conilpotent $\Cc$-coalgebras. We next review a filtered version of the twisted bar-cobar adjunction.

Given a operad $\Pp$ and a cooperad $\Cc$, denote the composition maps of $\Pp$ by $\gamma : \Pp \circ \Pp \to \Pp$ and the decomposition maps of $\Cc$ by $\Delta : \Cc \to \Cc \circ \Cc$. The set $\Hom_\SS(\Cc,\Pp)$ of homomorphisms of symmetric sequences is a dg-pre-Lie algebra with differential
\[ \partial(f) = d_\Pp \circ f - (-1)^{|f|} f \circ d_\Cc \] 
and pre-Lie product given by
\[ f \star g = \Cc \xrightarrow{\Delta_{(1)}} \Cc \circ_{(1)} \Cc \xrightarrow{f \circ_{(1)} g} \Pp \circ_{(1)} \Pp \xrightarrow{\gamma_{(1)}} \Pp. \]
Here, the $(1)$ subscript stands for composing $g$ at only one slot in $f$, see section 6.4 of \cite{LV} for details. Recall also that an \textit{operadic twisting morphism} $\alpha : \Cc \to \Pp$ is a morphism of symmetric sequences of degree $1$ satisfying 
\[ \partial(\alpha) + \alpha \star \alpha = 0. \]
Moreover, given a twisting morphism $\alpha : \Cc \to \Pp$, there is a \textit{twisted bar-cobar adjunction}
\[ \Omega_\alpha : \{ \mathrm{conil. dg \hspace{1ex}} \Cc\textrm{-coalgebras} \} \rightleftharpoons \{\mathrm{dg \hspace{1ex}} \Pp\textrm{-algebras} \} : \mathrm{B}_\alpha. \]
Given a $\Pp$-algebra $A$, the coalgebra $\mathrm{B}_\alpha(A)$ is the free $\Cc$-coalgebra $\Cc(A)$. Its differential is the sum of three differentials. One is given by
\[\Cc(A) \xrightarrow{d_\CC \circ Id} \Cc(A), \]
where $d_\Cc$ denotes the differential of $\Cc$. One is the unique coderivation extending 
\[ \Cc(A) \twoheadrightarrow A \xrightarrow{d_A} A, \] 
where $d_A$ denotes the differential of $A$. The other is the unique coderivation extending the morphism  
\[ \Cc(A) \xrightarrow{\alpha \circ \mathrm{Id}_A} \Pp(A) \xrightarrow{\gamma_A} A, \]
where $\gamma_A$ denotes the structure map of $A$. Given a conilpotent $\Cc$-coalgebra $C$, the underlying $\Pp$-algebra of $\Omega_\alpha(C)$ is $\Pp(C)$. The definition of the differentials is analogous (see, for instance, section 11.3 of \cite{LV}). Adding the filtrations induced on $\Cc(A)$ and $\Pp(C)$ as defined in (\ref{freealgfilt-eq}), we get an adjunction
\begin{align}
  \label{fil-barcobar}
  &\Omega_\alpha : \Cc\textrm{-}\mathsf{Fcoalg} \rightleftharpoons \Pp\textrm{-}\mathsf{Falg} : \mathrm{B}_\alpha.
\end{align} 
\begin{defi}
  A map $f : (A,W) \to (A',W)$ of filtered complexes is said to be a \textit{filtered quasi-isomorphism} if for every $k \in \ZZ$, the induced map
  \[ Gr^W_k(f)^* : H^*(Gr^W_k(A)) \to H^*(Gr^W_k(A')) \]
  is an isomorphism. A map of filtered $\Pp$-algebras is a \textit{filtered quasi-isomorphism} if the underlying map of filtered complexes is a filtered quasi-isomorphism.
\end{defi}
In the literature, the usual definition of a filtered quasi-isomorphism is a morphism $f : (A,W) \to (A',W)$ such that for every $k \in \ZZ$,
\[ W_k f : W_kA \to W_kA'\]
is a quasi-isomorphism. These two definitions coincide when $A$ and $A'$ are \textit{regular}. A filtered complex $A$ is regular if for every degree $n\in \ZZ$, there exists $k \in \ZZ$ such that
\[ W_k A^n = 0.\]
For conilpotent filtered $\Cc$-coalgebras, the desired notion of weak-equivalence is a subclass of filtered quasi-isomorphisms.
\begin{defi}
  A morphism of filtered conilpotent $\Cc$-coalgebras $f : (C,W) \to (C',W)$ is said to be a \textit{weak-equivalence} if $\Omega_\alpha(f)$ is a filtered quasi-isomorphism of $\Pp$-algebras.
\end{defi}
\begin{prop}
  If $f : (C,W) \to (C',W)$ is a weak-equivalence, then the underlying morphism of filtered complexes is a filtered quasi-isomorphism.
\end{prop}
\begin{proof}
  The associated graded $Gr^W$ is a symmetric monoidal functor so $Gr^W(C)$ and $Gr^W(C')$ are $\Cc$-coalgebras and $Gr^W(f)$ is a morphism of $\Cc$-coalgebras. We have the following equality:
  \[ Gr^W(\Omega_\alpha(f)) = \Omega_\alpha(Gr^W(f)). \]
  By the corresponding statement in the unfiltered setting (see section $2.3$ of \cite{Vall-homalg}), the morphism $Gr^W(f)$ is a quasi-isomorphism which implies that $f$ is a filtered quasi-isomorphism.
\end{proof}
Suppose further that $\alpha : \Cc \to \Pp$ is a Koszul twisting morphism (see section $6.6$ of \cite{LV}). Then, we have that
\begin{prop}
  \label{filcounit-prop}
  The counit $\epsilon_A : \Omega_\alpha \mathrm{B}_\alpha(A) \to A$ is a filtered quasi-isomorphism for all $A \in \Pp\textrm{-}\mathsf{Falg}$ and the unit $\nu_C : C \to \mathrm{B}_\alpha \Omega_\alpha(C)$ is a weak-equivalence for all $C \in \Cc\textrm{-}\mathsf{Fcoalg}$. 
\end{prop}
\begin{proof}
  We prove that $\epsilon_A$ is a filtered quasi-isomorphism. The associated graded $Gr^W$ is a symmetric monoidal functor and the composition
  \[ \Omega_\alpha \mathrm{B}_\alpha (Gr^W(A)) \cong Gr^W(\Omega_\alpha \mathrm{B}_\alpha(A)) \xrightarrow{Gr^W(\epsilon_A)} Gr^W(A) \] 
  is the counit of the classical bar-cobar adjunction applied to the $\Pp$-algebra $Gr^W(A)$. The result then follows from the fact that the classical counit is a quasi-isomorphism (see section $11.3$ of \cite{LV}). The proof for $\nu_C$ is similar.
\end{proof}
\begin{prop}
  \label{barcobar-qisoweq-prop}
  The functor $\mathrm{B}_\alpha$ sends filtered quasi-isomorphisms to weak-equivalences and the functor $\Omega_\alpha$ sends weak-equivalences to filtered quasi-isomorphisms.
\end{prop}
\begin{proof}
  This follows again from the fact that $Gr^W$ is symmetric monoidal and the corresponding statements for unfiltered (co)algebras (see section $2.3$ \cite{Vall-homalg} and section $2.1$ of \cite{BergKoszul}).
\end{proof}

\subsection{Filtered $\infty$-algebras}
\label{filinfalg-sec}
Fix $\Pp$ a Koszul operad. For us, this means a homogeneous quadratic operad with trivial differential and such that
\[ \Pp_\infty = \Omega \Pp^\antishriek \xrightarrow{\simeq} \Pp\] 
is a resolution of $\Pp$, where $\Pp^\antishriek$ is the Koszul dual cooperad of $\Pp$ (see section 7.4 of \cite{LV}) and $\Omega$ denotes the operadic cobar construction (see section 6.5 of \cite{LV}). The \textit{canonical twisting morphism} $\kappa : \Pp^\antishriek \to \Pp$ is given by projecting onto the generators of $\Pp$ and desuspending and the \textit{universal twisting morphism} $\iota: \Pp^\antishriek \to \Pp_\infty = \Omega \Pp^\antishriek$ is the inclusion.  

\begin{defi}
  A \textit{filtered $\Pp_\infty$-algebra} is a filtered algebra over $\Pp_\infty = \Omega \Pp^\antishriek$. Given two filtered $\Pp_\infty$-algebras $A$ and $A'$, a \textit{filtered $\infty$-morphism} between them is a morphism of filtered $\Pp^\antishriek$-coalgebras
  \[ \mathrm{B}_\iota A \to \mathrm{B}_\iota A'. \]
\end{defi}
\begin{rema}
  A filtered $\Pp_\infty$-algebra $A$ is equivalently defined as a filtered complex $(A,W)$ together with a filtration preserving degree $1$ map
  \[ m \in W_0 \Hom \Big(\overline{\Pp^\antishriek}(A),A \Big) = \Hom_\SS \Big(\overline{\Pp^\antishriek},\End_{(A,W)} \Big) \]
  satisfying the equation $\partial_A(m) + m \star m  = 0$ (see section $10.1$ of \cite{LV}). Here, $\overline{\Pp^\antishriek}$ denotes the cokernel of the coaugmentation $\Bbbk \to \Pp^\antishriek$ (see section 6.3 of \cite{LV}).

  Given complexes $A$ and $A'$, denote the operad $\End_{A'}^A$ given arity-wise by
  \[\End_{A'}^A(n) := \Hom(A^{\otimes n},A').\]
  Given $f \in \Hom(\Pp^{\antishriek},\End^A_{A'})$ and $g \in \Hom(\Pp^{\antishriek},\End_{A''}^{A'})$, define the following operation
  \[ g \circledcirc f : \Pp^\antishriek \xrightarrow{\Delta} \Pp^\antishriek \circ \Pp^\antishriek \xrightarrow{g \circ f} \End_{A'}^A \circ \End_{A''}^{A'} \xrightarrow{\gamma} \End_{A''}^A, \]
  where $\gamma$ is the usual composition map. Given two filtered $\Pp_\infty$-algebras $(A,W,m)$ and $(A',W,m')$, a filtered $\infty$-morphism is equivalently defined as a filtration preserving degree $0$ morphism 
  \[ f \in W_0 \Hom(\Pp^\antishriek(A),A') \]
  satisfying the equation 
  \[ f \star m - m' \circledcirc f = \del(f) \]
  (see section 10.2 of \cite{LV}). 
\end{rema}
\begin{defi}
    \label{filinfqiso-defi}
  A filtered $\infty$-morphism $f : A \to B$ is said to be a \textit{filtered $\infty$-quasi-isomorphism} if its first component $f_1 : A \to B$ is a filtered quasi-isomorphism. Likewise, $f$ is said to be a filtered \textit{$\infty$-isomorphism} (or \textit{$\infty$-isotopy}) if $f_1$ is an isomorphism of complexes such that $f_1(W_k A) = W_k B$ (or $f_1 = id$, respectively).
\end{defi}
Denote by $\infty\textrm{-}\Pp_\infty\textrm{-}\mathsf{Falg}$ the category whose objects are filtered $\Pp_\infty$-algebras and whose morphisms are filtered $\infty$-morphisms. The bar construction for the universal twisting morphism $\iota$ gives a fully faithful embedding 
\[ \mathrm{B}_\iota : \infty\textrm{-}\Pp_\infty\textrm{-}\mathsf{Falg} \to \Pp^\antishriek\textrm{-}\mathsf{Fcoalg}. \] 
Applying the cobar $\Omega_\kappa$ for the canonical twisting morphism $\kappa$ yields an adjunction
\begin{align}
  \label{inc-barcobar}
  \Omega_\kappa \mathrm{B}_\iota : \infty\textrm{-}\Pp_\infty\textrm{-}\mathsf{Falg} \rightleftharpoons \Pp\textrm{-}\mathsf{Falg} : i, 
\end{align} 
where $i$ denotes the inclusion of filtered $\Pp$-algebras into filtered $\Pp_\infty$-algebras. The following proposition is proved like Proposition \ref{filcounit-prop}, by comparing with the unfiltered setting.
\begin{prop}
  \label{filunit-prop}
  The unit $\nu_A : A \to \Omega_\kappa \mathrm{B}_\iota(A)$ of the adjunction is an $\infty$-quasi-isomorphism for every $A \in \Pp\textrm{-}\mathsf{Falg}$ and $\Omega_\kappa \mathrm{B}_\iota$ sends filtered $\infty$-quasi-isomorphisms to filtered quasi-isomorphisms.
\end{prop}
\subsection{Filtered infinity morphisms as MC elements}
We next describe filtered $\infty$-morphisms as Maurer-Cartan elements of a complete suspended $L_\infty$-algebra (or $SL_\infty$-algebra) and
collect properties of the gauge equivalences that we will need in later sections.

The equations in the definitions of Maurer-Cartan elements and gauge equivalences involve infinite series. To handle problems of convergence of such series, we consider $SL_\infty$-algebras which are \textit{complete}. We follow the definition of completeness of \cite{Berg-mapspace} (see also \cite{taleDotNorb}).
\begin{defi}
  An $SL_\infty$-algebra $(L,\{l_k\}_{k \geq 2})$ is said to be \textit{complete} if it is equipped with a decreasing filtration
  \[ L = F^1 L \supset F^2 L \supset \cdots \]
  such that 
  \begin{enumerate}
    \item For all $r \in \NN$ and $k \geq 2$, we have 
      \[ l_k(F_r L,L,...,L) \subset F^{r+1}L. \]
    \item For each $r$, there exists some $N$ such that for all $k > N$, we have
      \[ l_k(L,...,L) \subset F^r L. \]
    \item The canonical map 
    \[ L \cong \lim_{k \to \infty} L / F_r L \]
    is an isomorphism.
  \end{enumerate}
\end{defi}
\begin{defi}
    A \textit{Maurer-Cartan element} of an $SL_\infty$-algebra $L$ is a degree $1$ element $x \in L$ satisfying the Maurer-Cartan equation:
    \[ dx + \sum_{n \geq 2} \frac{1}{n!} l_n(x,...,x) = 0. \]
    We denote the set of Maurer-Cartan elements of $L$ by $\mathrm{MC}(L)$.
\end{defi}
The following two propositions are standard facts which can be proved by direct computation. See also chapter $4$ of \cite{mc-meth-twist} for a conceptual proof.

\begin{prop}
    Given a Maurer-Cartan element $x$ of a complete $SL_\infty$-algebra $L$, the following operations
    \[ d^x = \sum_{n \geq 0} \frac{1}{k!} l_{k+1}(x^{k}, -), \qquad l^x_n = \sum_{k \geq 0} \frac{1}{k!} l_{k+n}(x^{k}, -, ..., -) \quad \text{for } n \geq 2\] 
    define a complete $SL_\infty$-structure on $L$ with the same filtration.
\end{prop}

We denote the resulting $SL_\infty$-algebra by $L^x$ and call it the \textit{$SL_\infty$-algebra twisted by $x$}.

\begin{prop}
    \label{Linf-twist-prop}
    Let $x \in L$ be a Maurer-Cartan element of a complete $L_\infty$-algebra $L$. Then, an element $y \in L$ of degree $1$ is a Maurer-Cartan element of $L^x$ if and only if $x+y$ is a Maurer-Cartan element of $L$.
\end{prop}

\begin{defi}
  \label{gauge-defi}
    Given a Maurer-Cartan element $x \in L$ and a degree $0$ element $\lambda \in L$ in a complete $L_\infty$-algebra $L$, consider the differential equation
    \begin{equation}
        \label{gauge-dif-eq}
        \frac{d}{dt} \gamma(t) = d^{\gamma(t)}(\lambda),
    \end{equation}
    with initial condition $\gamma(0) = x$. The \textit{gauge action} of $\lambda$ on $x$ is 
    \[ \lambda \cdot x = \gamma(1). \]
    Two Maurer-Cartan elements $x,y$ are \textit{gauge equivalent} if there exists a \textit{gauge} from $x$ to $y$, that is, an element $\lambda \in L$ such that $\lambda \cdot x = y$.
\end{defi}

Given two $\Pp_\infty$-algebras $A$ and $A'$ (without filtration), the $\infty$-morphisms between $A$ and $A'$ are the Maurer-Cartan elements of the $SL_\infty$-algebra $(\Hom_k( \mathrm{B}_\iota A, A'),\{l_n\})$, with differential given by
\[ d(f) = d_{A'} \circ f - (-1)^{|f|} f \circ d_{\mathrm{B}_\iota A} \]
and structure maps given by
\begin{equation}
  \label{Linftymaps}
  l_n(f_1,...,f_n) = \sum_{\sigma \in \SS^n} (-1)^{\mathrm{sgn}(\sigma,f_1,\dots,f_n)} \gamma_{A'} \circ (\iota \otimes f_{\sigma(1)} \otimes \cdots \otimes f_{\sigma(n)}) \circ \Delta_n^A,
\end{equation}
for $n \geq 2$. Here, $\gamma_{A'}$ is the composition map of $A'$,
\[ \Delta_n^A : \Pp^\antishriek(A) \xrightarrow{\Delta} \Pp^\antishriek(\Pp^\antishriek(A)) \twoheadrightarrow \Pp^\antishriek(n) \otimes_{\SS_n} (\Pp^\antishriek(A))^{\otimes n} \] 
is the restriction of the cooperation $\Delta$ of $\Pp^\antishriek$ and $\mathrm{sgn}(\sigma,f_1,...,f_n)$ is the sign obtained by switching the $f_i$ according to the permutation $\sigma$ and the Koszul sign rule.

Now, given two filtered $\Pp_\infty$-algebras $A,A'$, consider the $SL_\infty$-algebra given by filtration preserving morphisms $W_0\Hom( \mathrm{B}_\iota A, A')$ together with the structure maps in (\ref{Linftymaps}). Exactly the same proof as in Theorem $7.1$ of \cite{wierstra} but replacing the $SL_\infty$-algebra by the one of filtration-preserving morphisms, yields the following
\begin{prop}
  \label{Pinfmorph=MC-prop}
  Given $A,A'$ filtered $\Pp_\infty$-algebras, there is a bijection
  \begin{align*}
    \Hom_{\Pp^{\antishriek}\textrm{-}\mathsf{Fcoalg}}(\mathrm{B}_\iota A, \mathrm{B}_\iota A') \cong \mathrm{MC}(W_0\Hom(\mathrm{B}_\iota A, A')).
  \end{align*}
\end{prop}

In \cite{defLiealg}, it is proved that composition with $\infty$-morphisms gives a well defined map of Maurer-Cartan sets. We sketch their proof, adapted to filtration preserving $\infty$-morphisms and prove a related statement regarding gauge equivalences. 

\begin{prop}
    \label{comp-gauge-prop}
    Let $A, A', A''$ be filtered $\Pp_\infty$-algebras and $f : A' \to A''$ a filtered $\infty$-morphism. Then, post-composition with $f$ induces a morphism of Maurer-Cartan sets
    \[ \mathrm{MC}(W_0\Hom(\mathrm{B}_\iota A,A'))  \xrightarrow{f_*} \mathrm{MC}(W_0\Hom(\mathrm{B}_\iota A,A'')). \]
    Moreover, if $x,y \in \mathrm{MC}(W_0\Hom(\mathrm{B}_\iota A,A')$ are gauge equivalent by a gauge $h$ (so $h \cdot x = y$), then 
    \[ (f \circledcirc h) \cdot (f \circledcirc x) = f \circledcirc y.\]
    Likewise, if $g : A'' \to A$ is a filtered $\infty$-morphism, then precomposition by $\mathrm{B}_\iota(g)$ induces a morphism between the Maurer-Cartan sets
    \[ \mathrm{MC}(W_0\Hom(\mathrm{B}_\iota A,A')) \xrightarrow{g^*} \mathrm{MC}(W_0\Hom(\mathrm{B}_\iota A'',A')) \]
    and if $x,y \in \mathrm{MC}(W_0\Hom(\mathrm{B}_\iota A,A')$ are gauge equivalent by a gauge $h$ (so $h \cdot x = y$), then 
    \[ (h \circledcirc g) \cdot (x \circledcirc g) = y \circledcirc g.\]
\end{prop}
\begin{proof}
    The construction of this $SL_\infty$-algebra is natural with respect to $\infty$-morphisms in both variables and thus yields a functor 
    \begin{align*}
    \infty\textrm{-}\Pp_\infty\textrm{-}\mathsf{Falg} \times \infty\textrm{-}\Pp_\infty\textrm{-}\mathsf{Falg} &\xrightarrow[]{\Hom(-,-)} \infty\textrm{-}SL_\infty\textrm{-}\mathsf{alg} \\
    (A,A') &\mapsto W_0 \Hom(\mathrm{B}_\iota(A),A'),
    \end{align*}
    where $\infty\textrm{-}SL_\infty\textrm{-}\mathsf{alg}$ denotes the category of $SL_\infty$-algebras together with $\infty$-morphisms. See \cite{NicoudDefhomalg} for this functor in the unfiltered setting, the same construction works for filtered algebras and morphisms. Given a filtered $\infty$-morphism $f : A' \to A''$, there is thus an $\infty$-morphism of $SL_\infty$-algebras
    \[ \Hom(1,f) : W_0 \Hom(\mathrm{B}_\iota(A),A') \to W_0 \Hom(\mathrm{B}_\iota(A),A''). \]
    An $\infty$-morphism $\theta : \mathfrak{g} \to \mathfrak{h}$ between $SL_\infty$-algebras induces a morphism on the Maurer-Cartan sets by the formula
    \[ \mathrm{MC}(\theta)(x) = \sum_{n \geq 1} \frac{1}{n!} \theta_n(x,...,x) \in \mathrm{MC}(\mathfrak{h}), \qquad \mathrm{for \hspace{1ex}} x \in \mathrm{MC}(\mathfrak{g}). \]
     Applied to the $\infty$-morphism $\Hom(1,f)$, this formula yields
     \[ \mathrm{MC}(\Hom(1,f))(x) = f \circledcirc x, \]
     as shown in \cite{defLiealg}. This proves the first claim. We now show the claim regarding $f$ and gauge equivalences. Given a gauge $h$ between
     \[ x,y \in \mathrm{MC}(W_0\Hom(\mathrm{B}_\iota(A),A')),\]
     let $\alpha(t)$ be a solution to the differential equation associated to the gauge $h$ (recall Definition \ref{gauge-defi}):
     \[ \alpha'(t) = - d^{\alpha(t)}(h), \]
     such that $\alpha(0) = x$ and $\alpha(1) = y$. Note that
     \begin{align*}
         - d^{f \circledcirc \alpha(t)}(f \circledcirc h) &= - \sum_n \frac{1}{n!} l_{1+n}(f \circledcirc \alpha(t),\dots,f \circledcirc \alpha(t), f \circledcirc h) \\  &= f \circledcirc \Big( - \sum_n \frac{1}{n!} l_{1+n}(\alpha(t),\dots,\alpha(t),h) \Big) \\ &= f \circledcirc \alpha'(t),
    \end{align*}
    where the second equality follows from the definition of the structure maps $l_n$. Together with the fact that
    \[ f \circledcirc \alpha(0) = f \circledcirc x \quad \text{and} \quad  f \circledcirc \alpha(1) = f \circledcirc y,\]
    this shows that $f \circledcirc h$ is a gauge between $f \circledcirc x$ and $f \circledcirc y$. The proof for the statements regarding precomposition with $g : A'' \to A$ follow similarly.
\end{proof}

Consider the free commutative algebra $I = \Lambda(t,dt)$, with $|t| = 0$, seen as a filtered algebra concentrated in weight $0$. Given a filtered $\Pp$-algebra $(A,W)$ with structure maps $m : \Pp \to \End_{(A,W)}$, the complex $A \otimes I$ has a natural structure of filtered $\Pp$-algebra, given by
\[ \Pp \cong \Pp \otimes \mathrm{Com} \xrightarrow{m \otimes \mu} \End_{(A,W)} \otimes \End_I \cong \End_{(A \otimes I,W)}, \]
where $\mu$ is the commutative product structure of $I$.
\begin{defi}
  \label{hom-filmorph}
  Two morphisms of filtered $\Pp$-algebras, $f,g : A \to A'$ are said to be \textit{homotopic} if there exists a morphism $h : A \to A' \otimes I$ of filtered $\Pp$-algebras making the following diagram commute:
  \begin{equation*}
    \begin{tikzcd}
      A \arrow[d, "h"] \arrow[rrd, "f \times g"]& & \\
      A' \otimes I \arrow[rr,"\delta_0 \times \delta_1"] & & A' \times A'.
    \end{tikzcd}
  \end{equation*}
\end{defi}

\begin{lemm}
  \label{barcobar-hom}
  If two Maurer-Cartan elements $f,g \in W_0\Hom(\mathrm{B}_\iota A, A')$ are gauge equivalent, then the morphisms 
  \[ \Omega_\kappa B_\iota(f) \quad \text{and} \quad \Omega_\kappa B_\iota(g) \] 
  of filtered $\Pp$-algebras are homotopic. 
\end{lemm}
\begin{proof}
  Denote by $\mathfrak{g}$ the $SL_\infty$-algebra $\mathfrak{g} = W_0\Hom(\mathrm{B}_\iota A, A')$. The elements of $\mathrm{MC}(\mathfrak{g})$ are the $0$-simplices of the Deligne-Hinich groupoid $\mathrm{MC}_\bullet(\mathfrak{g}) = \mathrm{MC}(\mathfrak{g} \otimes I_\bullet)$ and two of those are gauge equivalent if and only if they are in the same path component of $\mathrm{MC}_\bullet(\mathfrak{g})$. In section 4.1 of \cite{defLiealg}, the authors prove that there is a natural homotopy equivalence of simplicial sets given by the canonical inclusion:
  \begin{align*}
    \mathrm{MC}_\bullet(\Hom(\mathrm{B}_\iota A, A')) &\xrightarrow{\simeq} \mathrm{MC}(\Hom(\mathrm{B}_\iota A, A' \otimes I_\bullet)). \\
    f \otimes \omega &\mapsto ( x \mapsto f(x) \otimes \omega).
  \end{align*}
  The same proof shows the filtered version of this result. Let us sketch it:
  \begin{enumerate}
    \item First one considers the Dupont contraction \cite{dupont-contrac}:
    \begin{equation*}
      \begin{tikzcd}
        I_\bullet \arrow[loop left, "h_\bullet"] \arrow[r,shift left,"p_\bullet"] \arrow[r,shift right, leftarrow, "s_\bullet"'] & C_\bullet,
      \end{tikzcd}
    \end{equation*}
    seen as a contraction of simplicial filtered complexes concentrated in weight $0$. The simplicial complex $C_\bullet$ is a certain finite dimensional subcomplex of $I_\bullet$.
    \item Using this contraction, one transfers the $SL_\infty$-structure of 
    \[W_0\Hom( \mathrm{B}_\iota A, A') \otimes I_\bullet \] to the complex $W_0\Hom( \mathrm{B}_\iota A, A') \otimes C_\bullet$ and prove that the simplicial sets obtained from these $SL_\infty$ algebras by taking Maurer-Cartan sets are homotopy equivalent. 
    \item Again using the contraction, one can do the same for the $SL_\infty$-algebra 
    \[ W_0\Hom(\mathrm{B}_\iota A, A' \otimes I_\bullet) \] 
    and the complex $W_0\Hom(\mathrm{B}_\iota A, A' \otimes C_\bullet)$. Note that here is relevant the fact that the Dupont contraction is a contraction of filtered complexes.
    \item One shows that the natural inclusion
    \begin{align*}
      W_0\Hom( \mathrm{B}_\iota A, A') \otimes C_\bullet \to W_0\Hom( \mathrm{B}_\iota A, A' \otimes C_\bullet)
    \end{align*}
    is an isomorphism of simplicial $SL_\infty$-algebras for the transferred structures.
    \item Finally, one shows that the composition of all these equivalences is the map
    \begin{align*}
      \mathrm{MC}_\bullet(W_0\Hom(\mathrm{B}_\iota A, A')) &\xrightarrow{\simeq} \mathrm{MC}(W_0\Hom( \mathrm{B}_\iota A, A' \otimes I_\bullet)). \\
      f \otimes \omega &\mapsto ( x \mapsto f(x) \otimes \omega).
    \end{align*}
  \end{enumerate}
  Now, if $f,g \in \mathrm{MC}(W_0\Hom( \mathrm{B}_\iota A, A'))$ are gauge equivalent, then, by the previous result, they lie in the same path component of \[ \mathrm{MC}(W_0\Hom( \mathrm{B}_\iota A, A' \otimes I_\bullet)). \] This means that there exists a morphism of filtered $\Pp^\antishriek$-coalgebras
  \[ h : \mathrm{B}_\iota A \to \mathrm{B}_\iota (A' \otimes I) \]
  such that $\delta_0 h = B(f)$ and $\delta_1 h = B(g)$. Applying $\Omega_\kappa$ to $h$, one gets a diagram of filtered $\Pp$-algebras
  \begin{equation*}
    \begin{tikzcd}
      \Omega_\kappa \mathrm{B}_\iota A \arrow[d, "\Omega_\kappa(h)"] \arrow[rrd, "\Omega_\kappa(B(f) \times B(g))"]& & \\
      \Omega_\kappa \mathrm{B}_\iota (A' \otimes I) \arrow[rr, "\Omega_\kappa(\delta_0 \times \delta_1)"]& & \Omega_\kappa (\mathrm{B}_\iota (A') \times \mathrm{B}_\iota(A')).
    \end{tikzcd}
  \end{equation*}
  Consider the following square of filtered $\Pp$-algebras:
  \begin{equation*}
    \begin{tikzcd}
      \Omega_\kappa \mathrm{B}_\iota (A' \otimes I) \arrow[d, "\gamma"] \arrow[rr, "\Omega_\kappa(\delta_0 \times \delta_1)"]& &\Omega_\kappa (\mathrm{B}_\iota (A') \times \mathrm{B}_\iota(A')) \arrow[d,"\Omega_\kappa \pi_1 \times \Omega_\kappa \pi_2"]\\
      \Omega_\kappa \mathrm{B}_\iota (A') \otimes I\arrow[rr, "\delta_0 \times \delta_1"] & & \Omega_\kappa \mathrm{B}_\iota (A') \times \Omega_\kappa \mathrm{B}_\iota (A').
    \end{tikzcd}
  \end{equation*}
  Here, $\gamma$ is the map of algebras induced by
  \[ (a_1 \otimes \omega_1) \otimes \cdots \otimes (a_n \otimes \omega_n) \mapsto (a_1 \otimes \cdots \otimes a_n) \otimes (\omega_1 \cdots \omega_n), \]
  for $a_i \in A', \omega_i \in I$ and the maps $\pi_1$ and $\pi_2$ are the projections \[ \mathrm{B}_\iota (A') \times \mathrm{B}_\iota(A') \to \mathrm{B}_\iota(A') \] onto the first and second factors, respectively. One can check that $\gamma$ is compatible with the differentials, the square commutes and 
  \[ (\Omega_\kappa \pi_1 \times \Omega_\kappa \pi_2) \circ (\Omega_\kappa(B(f) \times B(g))) = \Omega_\kappa B(f) \times \Omega_\kappa B(g). \] It then follows that $\gamma \circ \Omega_\kappa(h)$ is an homotopy between $\Omega_\kappa B(f)$ and $\Omega_\kappa B(g)$.
\end{proof} 

\begin{rema}
  There are model structures on the category of filtered $\Pp$-algebras for an operad $\Pp$ (see \cite{Brotherston}, \cite{filtderived}) and one could probably define a model structure on filtered $\Pp^\antishriek$-coalgebras like in \cite{Vall-homalg} such that bar-cobar is a Quillen equivalence between filtered $\Pp$-algebras and filtered $\Pp^\antishriek$-coalgebras. From this, the previous lemma and its converse would follow. For the purposes of this paper, the previous lemma is sufficient.
\end{rema}

\subsection{Homotopy transfer for $d$-strict complexes}
In this section we give a filtered version of the
Homotopy Transfer Theorem for operadic algebras. To do so, we introduce a category of filtered complexes satisfying certain regularity conditions, the $\NN$-regular filtered complexes. On one hand, for such complexes satisfying an extra strictness condition (Definition \ref{inddstrict-defi}), we show that there exists a filtered contraction (Lemma \ref{fil-contrac-lemma}), an essential ingredient in the homotopy transfer theorem. On the other hand, the bar and cobar functors admit natural extensions to this category. The category of $\NN$-regular filtered complexes is thus restrictive enough to be able to apply the homotopy transfer theorem but flexible enough to admit cofibrant resolutions of its objects. Moreover, all examples arising from geometry yield $\NN$-regular filtered complexes. 
\label{subsec-dstric}

The following definition is based on the definition of regular filtered complex. Recall that a filtered complex $(A,W)$ is regular if for all $n \in \ZZ$, there exists $k \in \ZZ$ such that $W_k A^n = 0$.

\begin{defi}
  \label{indreg-defi}
  An \textit{$\NN$-regular filtered complex} is a complex with two filtrations $(A,W,L)$ such that $(A,W)$ is a filtered complex whose cohomology with the associated filtration $(H^*(A),W)$ is regular and $L$ satisfies the following conditions:
  \begin{enumerate}
    \item $L$ is ascending, positive and exhaustive, meaning
    \[ L_0 A = 0 \subset L_1 A \subset \dots \subset A, \qquad \bigcup_k L_k A = A. \]
    \item For every $k \geq 1$, the filtered complex $(L_k A, W)$ is regular.
  \end{enumerate}
\end{defi}
Note that a regular filtered complex $(A,W)$ is $\NN$-regular for the trivial filtration 
\[ L_0 A = 0 \subset L_1 = A. \]
The filtered complexes coming from geometry that we will consider are all regular so, in particular, they are $\NN$-regular. 

\begin{defi}
  A morphism $f : (A,W,L) \to (B,W,L)$ of $\NN$-regular filtered complexes is a morphism of complexes which is compatible with both filtrations. Moreover,
  \begin{enumerate}
      \item A morphism of $\NN$-regular filtered complexes is a \textit{quasi-isomorphism} if for all $n \in \ZZ$ and $k \geq 1$,
  \[ Gr^W_n Gr^L_k(f) :  Gr^W_n Gr^L_k(A) \to  Gr^W_n Gr^L_k(B) \]
  is a quasi-isomorphism of complexes.
    \item A morphism of $\NN$-regular filtered complexes is an \textit{isomorphism} if it is an isomorphism of complexes that strictly preserves both filtrations.
  \end{enumerate}
  
\end{defi}
We similarly define $\Pp$-algebras in $\NN$-regular filtered complexes as $\Pp$-algebras whose product preserves both filtrations and likewise for the morphisms. Denote by $\Pp\textrm{-}\mathsf{Falg}^{\NN}$ the category of $\NN$-regular filtered $\Pp$-algebras. Let $(A,W,L) \in \Pp\textrm{-}\mathsf{Falg}^{\NN}$ and denote by $\tilde{W}$ and $\tilde{L}$ the free extensions of the filtrations $W$ and $L$ (according to formula (\ref{freealgfilt-eq})). Given a twisting morphism $\alpha : \Cc \to \Pp$, the object
\[ (\mathrm{B}_\alpha(A),\tilde{W},\tilde{L}) \]
is an $\NN$-regular filtered $\Pp$-algebra. This is because for each $k \in \NN$,
\[ \tilde{L}_k \mathrm{B}_\alpha(A) \]
is composed of a finite sum of tensor products of filtered complexes of the form $L_i A$, which are all regular. Hence, $L_k \mathrm{B}_\alpha(A)$ is itself regular. The bar-cobar adjunction thus naturally extends to this setting:
\begin{align}
  \label{Nfil-barcobar}
  &\Omega_\alpha : \Cc\textrm{-}\mathsf{Fcoalg}^\NN \rightleftharpoons \Pp\textrm{-}\mathsf{Falg}^\NN : \mathrm{B}_\alpha.
\end{align} 
\begin{rema}
    Note that this does not happen for regular filtered complexes. Given a degree $n \in \ZZ$, the space $\mathrm{B}_\alpha(A)^n$ is, in general, composed of an infinite sum tensor powers of $A$ and so even if $A$ is regular, $\mathrm{B}_\alpha(A)$ can be non-regular.
\end{rema}
A morphism $f : (C,W,L) \to (C,W,L)$ between $\NN$-regular filtered $\Cc$-coalgebras is a \textit{weak equivalence} if $\Omega_\alpha(f)$ is a quasi-isomorphism of $\NN$-regular filtered complexes. The same proofs as in Proposition \ref{filcounit-prop} and Proposition \ref{barcobar-qisoweq-prop} show
\begin{prop}
\label{Nfilcounit-prop}
    Let $A \in \Pp\textrm{-}\mathsf{Falg}$ and $C \in \Cc\textrm{-}\mathsf{Fcoalg}$.
    \begin{enumerate}
        \item The counit $\epsilon_A : \Omega_\alpha \mathrm{B}_\alpha(A) \to A$ is a filtered quasi-isomorphism.
        \item The unit $\nu_C : C \to \mathrm{B}_\alpha \Omega_\alpha(C)$ is a weak-equivalence.
        \item The functor $\mathrm{B}_\alpha$ sends filtered quasi-isomorphisms to weak-equivalences and the functor $\Omega_\alpha$ sends weak-equivalences to filtered quasi-isomorphisms.
    \end{enumerate}
\end{prop}
Given a Koszul operad $\Pp$, the definition of $\infty$-morphism in this setting is completely analogous and the $SL_\infty$-algebra 
\[ W_0 L_0 \Hom(\mathrm{B}_\iota A,A') \]
with the structure maps in (\ref{Linftymaps}) has $\infty$-morphisms as its Maurer-Cartan elements. 

A filtered complex $(A,W)$ is said to be \textit{$d$-strict} if for all $k \in \ZZ$, it satisfies
\[ d(W_k A) = W_k A \cap d(A). \]
If $(A,W)$ is regular, then it is $d$-strict if and only if its associated spectral sequence degenerates at the first page. For $\NN$-regular filtered complexes, we consider the following analogous definition:
\begin{defi}
  \label{inddstrict-defi}
  An $\NN$-regular filtered complex $(A,W,L)$ is said to be \textit{$\NN$-$d$-strict} if 
  \begin{enumerate}
    \item $(A,L)$ is $d$-strict,
    \item For every $k \geq 1$, the filtered complex $(L_k A,W)$ is $d$-strict.
  \end{enumerate}
\end{defi}
Note that when $L$ is trivial, then we recover the definition of $d$-strict filtered complex. For us, the important fact about $\NN$-d-strict filtered complexes is that they admit filtered contractions.

\begin{defi}
  Let $(A,W,L)$ be an $\NN$-regular filtered complex. A \textit{filtered contraction} is a diagram
  \begin{equation*}
    \begin{tikzcd}
      (A,W,L) \arrow[loop left, "h"] \arrow[r,shift left,"p"] \arrow[r,shift right, leftarrow, "s"'] & (H^*(A),W,L),
    \end{tikzcd}
  \end{equation*}
  where $s$ and $p$ are morphisms of $\NN$-regular filtered complexes such that
  \[ ps = id_{H^*(A)} \]
  and $h$ is a homotopy between $sp$ and $id_A$, compatible with filtrations:
  \[ dh + hd = sp - 1, \qquad h(W_p A) \subset W_p A, \qquad h(L_k A) \subset L_k A. \]
\end{defi}

\begin{lemm}
  \label{fil-contrac-lemma}
  Every $\NN$-$d$-strict filtered complex $(A,W,L)$ admits a filtered contraction.
\end{lemm}
\begin{proof}
  A sequence of bifiltered vector spaces of the form
  \[ 0 \to (A,W,L) \xrightarrow{\iota} (B,W,L) \xrightarrow{\pi} (C,W) \to 0 \]
  is said to be \textit{exact} if for every $n,k \in \ZZ$, the sequence
  \begin{equation}
    \label{d-strict-ses}
    0 \to W_n L_k A \xrightarrow{\iota} W_n L_k B \xrightarrow{\pi} W_n L_k C \to 0 
  \end{equation} 
  is a short exact sequence. It is a classical fact that any exact sequence of bifiltered vector spaces with bounded below filtration $W$ and $L$ is \textit{split}. That is, there exist bifiltered morphisms $r : C \to B$ and $l : B \to A$ such that $\pi r = id_C$ and $l \iota = id_A$. 
  
  Denote by $Z^n(A)$ and $B^n(A)$ the cycles and boundaries, respectively, of degree $n$. The fact that $(A,W,L)$ is $\NN$-$d$-strict implies that for every $n \in \ZZ$ and $k \geq 1$, we have that 
  \begin{enumerate}
    \item \label{complreglemm1} $d(W_n L_k A) = \Img(d) \cap W_n L_k A$,
    \item \label{complreglemm2} $\Img(H^*(W_n L_k A) \to H^*(A)) = W_n L_k H^*(A)$.
  \end{enumerate}
  These properties, in turn, imply that the bifiltered sequences
  \begin{align*}
    0 \to (B^n(A),W,L) \to (Z^n(A),W,L) \to (H^n(A),W,L) \to 0 \\ 
    0 \to (Z^n(A),W,L) \to (A^n,W,L) \to (B^{n+1}(A),W,L) \to 0
  \end{align*} 
  are exact and hence, split. The splittings of both sequences yield an isomorphism of bifiltered vector spaces
  \[ A^n \cong B^n(A) \oplus H^n(A) \oplus B^{n+1}(A). \]
  Endowing the filtered graded vector space on the right with the differential $d(b,a,b') = (b',0,0)$, one gets an isomorphism of bifiltered complexes 
  \[ A \cong B^*(A) \oplus H^*(A) \oplus B^{*+1}(A). \]
  Now, the morphisms $s(a) = (0,a,0)$, $p(b,a,b') = a$, $h(b,a,b') = (0,0,b)$ give a filtered contraction 
\end{proof}
Fix $\Pp$ a Koszul operad. We now give a filtered version of the homotopy transfer theorem.

\begin{prop}[Filtered Homotopy Transfer]
  \label{filt-htt}
  Let $A$ be an $\NN$-regular filtered $\Pp$-algebra with a filtered contraction
  \begin{equation*}
    \begin{tikzcd}
      (A,W,L) \arrow[loop left, "h"] \arrow[r,shift left,"p"] \arrow[r,shift right, leftarrow, "s"'] & (H^*(A),W,L).
    \end{tikzcd}
  \end{equation*}
  There exists an $\NN$-regular filtered $\Pp_\infty$-structure on $H^*(A)$ together with filtered $\infty$-quasi-isomorphisms 
  \[ S : A \to H^*(A) \quad \text{and} \quad P : A \to H^*(A) \] 
  extending $s$ and $p$, respectively. 
  Moreover, $h$ extends to a gauge between $SP$ and $id_A$.
\end{prop}
\begin{proof}
  Denote the $\{x,y\}$-coloured endomorphism operad from $A$ to $H^*(A)$ by $\End_{A,H^*(A)}$. That is,
      \begin{align*}
        &\End_{A,H^*(A)}(\underbrace{x,...,x}_{n \mathrm{ \hspace{0.5ex} times}};x) = \Hom_\Bbbk(A^{\otimes n}, A) \\
        &\End_{A,H^*(A)}(\underbrace{y,...,y}_{n \mathrm{ \hspace{0.5ex} times}};y) = \Hom_\Bbbk(H^*(A)^{\otimes n}, H^*(A)) \\
        &\End_{A,H^*(A)}(\underbrace{x,...,x}_{n \mathrm{ \hspace{0.5ex} times}};y) = \Hom_\Bbbk(A^{\otimes n}, H^*(A)).
    \end{align*}
  In section $4$ of \cite{Markl-trans}, the author proves the unfiltered version of this theorem by constructing a map of $\{x,y\}$-colored operads $\mathfrak{m}_{S} \to \mathfrak{m}_{\underline{S}}$ where
  \begin{enumerate}
    \item morphisms $\mathfrak{m}_{S} \to \End_{A,H^*(A)}$ consist of $\Pp_\infty$-algebra structures on $A$ and $H^*(A)$, $\Pp_\infty$-morphisms $P : A \to H^*(A)$ and $S : H^*(A) \to A$ and a gauge $H$ between $SP$ and $id_A$,
    \item morphisms $\mathfrak{m}_{\underline{S}} \to \End_{A,H^*(A)}$ consist of $\Pp_\infty$-algebra structures on $A$ and $H^*(A)$, chain maps $p : A \to H^*(A)$ and $s : H^*(A) \to A$ and a homotopy $h$ between $sp$ and $id_A$.
  \end{enumerate}
  The fact that there is a morphism  of operads $\mathfrak{m}_{S} \to \mathfrak{m}_{\underline{S}}$ implies the unfiltered homotopy transfer theorem. The filtered version is obtained by replacing $\End_{A,H^*(A)}$ by 
  \[ \End_{(A,W,L),(H^*(A),W,L)}, \] 
  the endomorphism operad of maps preserving both filtrations.
\end{proof}
\begin{prop}
  \label{inv-infiso}
  Let $f : A \to B$ be a filtered $\infty$-isomorphism between $\NN$-regular $\Pp_\infty$-algebras $A$ and $B$. Then, there exists a filtered $\infty$-morphism $g : B \to A$ such that $fg = id_B$ and $gf = id_A$.
\end{prop}
\begin{proof}
  This follows from the same proof in the unfiltered setting (see section 10.4 of \cite{LV}) noting that the condition $f_1(W_n A) = W_n B$, $f_1(L_k A) = L_k B$, (recall Definition \ref{filinfqiso-defi}) implies that $g$ preserves both filtrations.
\end{proof}
When $A$ and $B$ are $\NN$-$d$-strict, then a filtered $\infty$-quasi-isomorphism $f : A \to B$ admits an inverse up to gauge equivalence. We require $A$ and $B$ to be $\NN$-$d$-strict to be able to use the filtered homotopy transfer theorem.
\begin{prop}
  \label{inv-infqiso}
  Let $f : A \to B$ be a filtered $\infty$-quasi-isomorphism between $\NN$-$d$-strict filtered $\Pp$-algebras. Then, there exists a filtered $\infty$-quasi-isomorphism $g : B \to A$ such that $fg \sim id_B$ and $gf \sim id_A$.
\end{prop}
\begin{proof}
  Again, this follows from the same arguments as in the unfiltered setting (see section 10.4 \cite{LV}). Namely, by Lemma \ref{fil-contrac-lemma}, since $A$ and $B$ are $\NN$-$d$-strict, there are filtered contractions
  \begin{equation*}
    \begin{tikzcd}
      (A,W) \arrow[loop left, "h_A"] \arrow[r,shift left,"p_A"] \arrow[r,shift right, leftarrow, "s_A"'] & (H^*(A),W), & & & (B,W) \arrow[loop left, "h_B"] \arrow[r,shift left,"p_B"] \arrow[r,shift right, leftarrow, "s_B"'] & (H^*(B),W).
    \end{tikzcd}
  \end{equation*}
  By the filtered homotopy transfer theorem there are filtered $\Pp_\infty$-structures on $H^*(A)$ and $H^*(B)$ and morphisms $P_A$, $S_A$, $P_B$ and $S_B$ extending $p_A$, $s_A$, $p_B$ and $s_B$, respectively and such that $S_A P_A \sim id_A$ and $S_B P_B \sim id_B$. Then, the composition $P_B f S_A : H^*(A) \to H^*(B)$ is a filtered $\infty$-isomorphism. By Proposition \ref{inv-infiso}, there is an inverse $\tilde{g}$ to $P_B f S_A$. It is then easy to check that $g := S_A \tilde{g} P_B$ satisfies $fg \sim id_B$ and $gf \sim id_A$.
\end{proof}

\begin{prop}
  \label{zig-zag-fil}
  Given two $\NN$-$d$-strict filtered $\Pp$-algebras $A,A'$, there exists a zig-zag of filtered quasi-isomorphisms between $A$ and $A'$ if and only if there exists a filtered $\infty$-quasi-isomorphism between $A$ and $A'$. 
\end{prop}
\begin{proof}
  If there exists a zig-zag of filtered quasi-isomorphisms 
  \[ A \xleftarrow{\sim} \cdots \xrightarrow{\sim} A', \]
  then for every morphism $f$ going from right to left, we apply Proposition \ref{inv-infqiso} to get an $\infty$-quasi-isomorphism going from left to right. Then, composing all resulting $\infty$-morphisms (with the original morphisms going from left to right), we get a filtered $\infty$-quasi-isomorphism from $A$ to $A'$.

  If there is a filtered $\infty$-quasi-isomorphism $f : A \to A'$, then the bar-cobar functor from (\ref{inc-barcobar}) applied to $f$ gives a map
  \[ \Omega_\kappa \mathrm{B}_\iota A \xrightarrow{\Omega_\kappa \mathrm{B}_\iota f} \Omega_\kappa \mathrm{B}_\iota A'. \]
  The map $\Omega_\kappa \mathrm{B}_\iota f$ is a filtered quasi-isomorphism since 
  \[ Gr_n^W Gr_k^L (\Omega_\kappa \mathrm{B}_\iota f) = \Omega_\kappa \mathrm{B}_\iota Gr_n^W Gr_k^L (f) \]
  and the map $Gr_n^W Gr_k^L (f)$ is an $\infty$-quasi-isomorphism, which implies that $\Omega_\kappa \mathrm{B}_\iota Gr^W_n Gr_k^L (f)$ is a quasi-isomorphism (see section $11.4$ of \cite{LV}).
  Note that since $A$ is a $\Pp$-algebra, we have $\mathrm{B}_\iota(A) \cong \mathrm{B}_\kappa(A)$ and likewise for $A'$. Hence, there is a zig-zag 
  \[ A \xleftarrow{\sim} \Omega_\kappa \mathrm{B}_\kappa A \xrightarrow{\Omega_\kappa \mathrm{B}_\iota f} \Omega_\kappa \mathrm{B}_\kappa A' \xrightarrow{\sim} A'\]
  and by Proposition \ref{Nfilcounit-prop}, the maps on the extremes are also filtered quasi-isomorphisms. Since $A$ is $\NN$-$d$-strict, $\Omega_\kappa \mathrm{B}_\kappa A$ is $\NN$-regular and $\Omega_\kappa \mathrm{B}_\kappa A$ is filtered quasi-isomorphic to $A$, it follows that $\Omega_\kappa \mathrm{B}_\kappa A$ is also $\NN$-$d$-strict. The same follows for $\Omega_\kappa \mathrm{B}_\kappa A'$.
\end{proof}

\subsection{Bifiltered operadic algebras}
\label{subsec-bifil}
In the remainder of this work, by a bifiltered complex, we mean a cochain complex $A$ together with one ascending filtration $W$ and one descending filtration $F$. Both filtrations are exhaustive and Hausdorff (see (\ref{exhauHausfil-eq})).

We can similarly define a bifiltered $\Pp$-algebra for a dg operad $\Pp$ as a bifiltered complex $(A,W,F)$ together with a morphism $\Pp \to \mathrm{End}_{(A,W,F)}$. The morphisms in $\mathrm{End}_{(A,W,F)}$ now preserve both filtrations. And analogously for bifiltered coalgebras for a dg cooperad $\Cc$.

All the constructions and results of the previous section for filtered algebras extend to the setting of bifiltered algebras, by just replacing "filtered" by "bifiltered". In particular, the bar-cobar construction extends directly to bifiltered algebras. Furthermore, bifiltered $\infty$-morphisms between two bifiltered $\Pp_\infty$-algebras $A$ and $A'$ are the Maurer-Cartan elements of the $L_\infty$-algebra
\[ W_0 F^0 \Hom_\Bbbk(\mathrm{B}_\iota A,A') \]
with structure maps given by equation (\ref{Linftymaps}). 
\begin{defi}
  \label{bifilqiso-defi}
  A morphism of bifiltered complexes $f : (A,W,R) \to (B,W,R)$ is said to be a \textit{bifiltered quasi-isomorphism} if the induced morphism
  \[ (Gr^W_k Gr_R^p f)^* : H^*(Gr^W_k Gr_R^p A) \to H^*(Gr^W_k Gr_R^p B) \]
  is an isomorphism of complexes for every $k,p \in \ZZ$. 
\end{defi} 
\begin{defi}
  A bifiltered $\infty$-morphism $f : A \to B$ is said to be
  \begin{enumerate}
      \item \textit{bifiltered $\infty$-quasi-isomorphism} if its first component $f_1$ is a bifiltered quasi-isomorphism,
      \item  a bifiltered \textit{$\infty$-isomorphism} if $f_1$ is an isomorphism such that $f_1(W_k A) = W_k B$ and $f_1(F^p A) = F^p B$,
      \item  an \textit{$\infty$-isotopy} if $f_1 = id$.
  \end{enumerate}
\end{defi}
A bifiltered complex $(A,W,F)$ is \textit{regular} if for every $n \in \ZZ$, there exist $k,p \in \ZZ$ such that
\[ W_k A^n = 0, \quad F^p A^n = 0.\]
To extend the results of section \ref{subsec-dstric} one must also consider \textit{$\NN$-regular bifiltered complexes}. These are complexes with three filtrations $(A,W,F,L)$ such that $L$ is positive, ascending and exhaustive and
\begin{enumerate}
    \item $(A,L)$ is regular,
    \item for every $k \geq 1$, $(L_kA,W,F)$ is regular.
\end{enumerate}
Likewise, a quasi-isomorphism of $\NN$-regular bifiltered complexes 
\[ f : (A,W,F,L) \to (B,W,F,L)\]
is a morphism preserving all filtrations and such that for every $k \geq 1$, the morphism $Gr_k^L(f)$ is a bifiltered quasi-isomorphism. The analogous notions of strictness are the following:

\begin{defi}
  \label{d-bistrict-defi}
  A bifiltered cochain complex $(A,W,F)$ is said to be \textit{$d$-bistrict} if for all $k,p \in \ZZ$,
  \begin{enumerate}
    \item $d(W_k F^p A) = W_k F^p A \cap d(A)$,
    \item The filtered complexes $(Gr^W_k A,F)$ and $(Gr_F^p A,W)$ are $d$-strict.
  \end{enumerate}
\end{defi}
A regular bifiltered complex is $d$-bistrict if and only if the following four spectral sequences degenerate at the $E_1$ page:
\[
  \begin{tikzcd}[sep=small]
    E_1(Gr^W(A),F) \arrow[dd,equal] \arrow[r,Rightarrow] & E_1(A,W) \arrow[rd,Rightarrow] & \\
    & & H^*(A) \\
    E_1(Gr_F(A),W) \arrow[r,Rightarrow] & E_1(A,F) \arrow[ru,Rightarrow]
  \end{tikzcd} 
\]
Analogously to filtered complexes, we define 
\begin{defi}
  \label{inddbistrict-defi}
  An $\NN$-regular bifiltered complex $(A,W,F,L)$ is said to be \textit{$\NN$-$d$-bistrict} if 
  \begin{enumerate}
    \item $(A,L)$ is $d$-strict and
    \item $(L_k A, W,F)$ is $d$-bistrict for all $k \geq 1$.
  \end{enumerate}
\end{defi}
\begin{defi}
  Let $(A,W,F,L)$ be an $\NN$-regular bifiltered complex. A \textit{bifiltered contraction} is diagram,
  \begin{equation*}
    \begin{tikzcd}
      (A,W,F,L) \arrow[loop left, "h"] \arrow[r,shift left,"p"] \arrow[r,shift right, leftarrow, "s"'] & (H^*(A),W,F,L),
    \end{tikzcd}
  \end{equation*}
  where $s$ and $p$ are morphisms of bifiltered complexes such that
  \[ ps = id_{H^*(A)} \] 
  and $h$ is a homotopy between $sp$ and $id_A$, compatible with filtrations:
  \[ dh + hd = sp - 1, \qquad h(W_k A) \subset W_k A, \qquad h(F^p A) \subset F^p A. \]
\end{defi}
The proof of the following lemma is similar to the one of Lemma \ref{fil-contrac-lemma}, where now we introduce the extra $F$ filtration.
\begin{lemm}
  \label{bifil-contrac-lemma}
  Every $\NN$-$d$-bistrict complex admits a bifiltered contraction. 
\end{lemm}
The same proofs as in Propositions \ref{filt-htt}, \ref{inv-infiso}, \ref{inv-infqiso} and \ref{zig-zag-fil} yield the following analogous propositions:

\begin{prop}[Bifiltered Homotopy Transfer]
  \label{bifil-htt}
  Let $A$ be an $\NN$-regular bifiltered $\Pp$-algebra with a bifiltered contraction
  \begin{equation*}
    \begin{tikzcd}
      (A,W,F,L) \arrow[loop left, "h"] \arrow[r,shift left,"p"] \arrow[r,shift right, leftarrow, "s"'] & (H^*(A),W,F,L).
    \end{tikzcd}
  \end{equation*}
  There exists a bifiltered $\Pp_\infty$-structure on $H^*(A)$ together with bifiltered $\infty$-quasi-isomorphisms $S : A \to H^*(A)$ and $P : A \to H^*(A)$ extending $s$ and $p$, respectively. Moreover, $h$ extends to a gauge between $SP$ and $id_A$.
\end{prop}

\begin{prop}
  \label{inv-infiso2}
  Let $f : A \to B$ be a bifiltered $\infty$-isomorphism. Then, there exists a bifiltered $\infty$-morphism $g : B \to A$ such that $fg = id_B$ and $gf = id_A$.
\end{prop}

\begin{prop}
  \label{inv-infqiso2}
  Let $f : A \to B$ be a bifiltered $\infty$-quasi-isomorphism between regular $d$-bistrict bifiltered $\Pp$-algebras. Then, there exists a bifiltered $\infty$-quasi-isomorphism $g : B \to A$ such that $fg \sim id_B$ and $gf \sim id_A$.
\end{prop}

\begin{prop}
  \label{zig-zag-fil2}
  Given two $\NN$-$d$-bistrict bifiltered $\Pp$-algebras $A$ and $A'$, there exists a zig-zag of bifiltered quasi-isomorphisms between $A$ and $A'$ if and only if there exists a bifiltered $\infty$-quasi-isomorphism between $A$ and $A'$.
\end{prop}

\section{Mixed Hodge structures in homotopy theory}
\label{mhs-sec}
In this section we introduce \textit{$\Pp$-mixed Hodge diagrams}. These
combine the notion of mixed Hodge complex introduced by Deligne with $\Pp$-algebra structures,
for $\Pp$ an operad.
We construct a bar-cobar adjunction for such objects
and develop a homotopy transfer theory in this setting. 

\subsection[]{Mixed Hodge structures}
We review some properties of the abelian category of mixed Hodge structures.  We refer to \cite{PS-mhs} for a detailed exposition. Let $\Bbbk$ be a subfield of $\RR$.
\begin{defi}
  A $\Bbbk$-\textit{pure Hodge structure} of weight $n$ on a finite dimensional $\Bbbk$-module $V$ is
  given by a direct sum decomposition of $V \otimes \CC$:
  \begin{equation*}
    V \otimes \CC = \bigoplus_{p+q = n} V^{p,q}, \text{\qquad where } \ov{V^{p,q}} = V^{q,p}. 
  \end{equation*}
\end{defi}

\begin{rema}
  A $\Bbbk$-pure Hodge structure of weight $n$ on a vector space $V$ is equivalently defined as a decreasing filtration
  $F$ on $V \otimes \CC$, called the \textit{Hodge filtration}, which satisfies the following condition:
  \begin{equation*}
    F^p V \cap \ov{F^q} = 0, \text{\qquad for } p + q = n + 1. 
  \end{equation*}
\end{rema}

Pure Hodge structures arise primarily in the cohomology of a compact Kähler manifold $X$, where $H^n(X)$ has a pure Hodge structure of weight $n$.
In his study of the cohomology of complex algebraic varieties \cite{DeHII}, Deligne introduced mixed Hodge structures, generalizing the above pure situation.

\begin{defi}
  A $\Bbbk$-\textit{mixed Hodge structure} on a finite dimensional $\Bbbk$-module $V$ is
  given by a pair $(W,F)$, where $W$ is an increasing filtration on $V$, called the
  \textit{weight filtration} and $F$ is a decreasing filtration on $V_\CC := V \otimes \CC$, called
  the \textit{Hodge filtration}, such that on
  \begin{equation*}
    Gr_i^W(V) = W_i V / W_{i-1} V,
  \end{equation*}
  the filtration $F$ induces a pure Hodge structure of weight $i$, for any integer $i$.
\end{defi}

In the following, may drop the $\Bbbk$- from the notation when the field is irrelevant for the discussion. 
\begin{rema}
  The data of a pure Hodge structure of weight $n$ on $V$ is the same as the data of a mixed Hodge 
structure on $V$ with weight filtration equal to the trivial filtration of weight $n$. That is, the filtration of length $0$ concentrated in weight $n$. So pure Hodge structures are a particular case of mixed Hodge structures.
\end{rema}

Denote by $\MHS_\Bbbk$ the category whose objects are $\Bbbk$-mixed Hodge structures
and morphisms are morphisms of $\Bbbk$-modules preserving both filtrations.
The category $\MHS_\Bbbk$ is abelian and closed symmetric monoidal with tensor and inner Hom given by the following. The underlying $\Bbbk$-modules are $A \otimes_\Bbbk B$ and $\mathrm{Hom}_\Bbbk(A,B)$, respectively, for $A,B \in \MHS_\Bbbk$. The weight filtrations are defined by
\begin{align}
  &W_n(A \otimes_\Bbbk B) = \bigoplus_{i+j = n} W_i A \otimes_\Bbbk W_j B \label{filt-tensor} \\
  &W_n \Hom_\Bbbk(A,B) = \{ f \in \Hom_\Bbbk(A,B) | f(W_i A) \subset W_{i+n}B \} \nonumber
\end{align}
and the Hodge filtrations are defined analogously.

The category $\MHS_\Bbbk$ has non-trivial extensions. They are described as follows:
\begin{prop}[\cite{mhsext-carlson}, see also \cite{PS-mhs}]
  \label{mh-ext}
  Given $A,B \in \MHS_\Bbbk$, 
  \begin{equation*}
    \mathrm{Ext}_{\MHS_\Bbbk}^1(A,B) \cong\frac{W_0\Hom_{\CC}(A \otimes \CC,B \otimes \CC)}{W_0 \Hom_{\Bbbk}^W(A,B) + W_0F^0\Hom_\CC(A \otimes \CC,B \otimes \CC)}.
  \end{equation*}
  Furthermore, $\mathrm{Ext}_{\MHS_\Bbbk}^n(A,B) = 0$ for $n \geq 2$.
\end{prop}

Given $(V,W,F) \in \MHS_\Bbbk$, consider the the $\Bbbk$-module $Gr^W(V) \coloneq \bigoplus_i Gr^W_i(V)$ together with the filtrations $(W_{col})_mV \coloneq \bigoplus_{i \leq m} Gr_i^W(V)$ and $F_{split}$ the Hodge filtration induced by $F$ on the graded pieces. Then $(Gr^W(V),W_{col},F_{split})$ is also a $\Bbbk$-mixed Hodge structure.
\begin{defi}
  \label{defi-split}
  If $(V,W,F) \in \MHS_\Bbbk$ is isomorphic to $(Gr^W(V),W_{col},F_{split})$ as $\Bbbk$-mixed Hodge structures, then $V$ is said to be \textit{split} over $\Bbbk$.
\end{defi}

\begin{rema}
  Equivalently, a $\Bbbk$-mixed Hodge structure is split over $\Bbbk$ if it is isomorphic to the direct sum of $\Bbbk$-pure Hodge structures of possibly different weights.
\end{rema}

Although not every $\Bbbk$-mixed Hodge structure is split, the complexified vector space $V_\CC = V \otimes_\Bbbk \CC$ of any $(V,W,F) \in \MHS_\Bbbk$ admits a functorial splitting, the \textit{Deligne splitting} (see, for instance, section 3.1 of \cite{PS-mhs}). Namely, there is a bigrading $V_\CC = \bigoplus_{p,q} I^{p,q}$ given by
\begin{equation*}
  I^{p,q} \coloneq F^p \cap W_{p+q} \big( \overline{F^q} \cap W_{p+q} + \sum_{j \geq 2} \overline{F^{q-j+1}} \cap W_{p+q-j} \big),
\end{equation*}
which satisfies 
\begin{equation*}
  W_m V \otimes \CC = \bigoplus_{p+q \leq m} I^{p,q}, \text{\qquad} F^p V_\CC = \bigoplus_{i \geq p} I^{i,j}.
\end{equation*}
\begin{rema}
  \label{realmhs-split-rema}
  In the case where $I^{p,q} = \overline{I^{q,p}}$, the spaces $V^m \coloneq \bigoplus_{p+q = m} I^{p,q}$ give real pure Hodge structures, so $V \otimes_\Bbbk \RR$ is split over $\RR$. In general, one only has
  \begin{equation*}
    I^{p,q} \equiv \overline{I^{q,p}} \mathrm{ \quad mod \hspace{1ex}} W_{p+q-2}. 
  \end{equation*}  
  Define the \textit{length} of the weight filtration of $V$ to be $b - a$, where $[a,b]$ is the smallest interval for which $W_k / W_{k-1} = 0$ for all $k \notin [a,b]$. In particular, when the length of $W$ is less than or equal to $1$, then $V \otimes_\Bbbk \RR$ is split over $\RR$. 
\end{rema}

\subsection[]{Mixed Hodge complexes}
\label{sec-mhc}
We recall the notion of mixed Hodge complex, introduced by Deligne \cite{DeHIII}. This category includes complexes of mixed Hodge structures and is the target of Deligne's functor from complex algebraic varieties.

\begin{defi}
  \label{mhc-defi}
  A $\Bbbk$-\textit{mixed Hodge complex of length $s$} is given by a filtered cochain complex $(A_\Bbbk,W)$ over $\Bbbk$, a bifiltered cochain complex $(A_\CC,W,F)$ over $\CC$, filtered complexes $(A_i,W)$ over $\CC$ for $1 \leq i \leq s-1$ and filtered quasi-isomorphisms $\varphi_{u}$ for each arrow $u : i \to j$ as in the following diagram:
  \[ \begin{tikzcd}
    & (A_1,W) & & \cdots & \\
    (A_0,W) = (A_\Bbbk,W) \otimes \CC \arrow[ru,"\varphi_{01}"] & & (A_2,W) \arrow[lu,"\varphi_{21}"'] \arrow[ru,"\varphi_{23}"] & & \arrow[lu,"\varphi_{s s-1}"'] (A_s,W) = (A_\CC,W)  
  \end{tikzcd} \]
  In addition, the following axioms are satisfied:
  \begin{enumerate}
    \item The weight filtrations $W$ are increasing, regular and exhaustive. The Hodge filtration $F$
    is decreasing and biregular. The cohomology $H^*(A_\Bbbk)$ has finite type.
    \item \label{mhc-axiom2} The bifiltered complex $(A_\CC,W,F)$ is $d$-bistrict.
    \item \label{mhc-axiom3} For all $n \in \NN$ and $p \in \ZZ$, the filtration induced by $F$ on $H^n(Gr_p^W A_\CC)$ and the isomorphisms $\varphi_{u}^*$ induce a pure Hodge structure of weight $p$ on $H^n(Gr_p^W A_\Bbbk)$.
  \end{enumerate}
\end{defi}

Note that axioms \ref{mhc-axiom2} and \ref{mhc-axiom3} imply that for $n \in \NN$, the triple $(H^*(A_\Bbbk),W,F)$ is a $\Bbbk$-mixed Hodge structure. Again, in the following, we may suppress $\Bbbk$ from the notation, for simplicity. The morphisms $\varphi_{u}$ are called the \textit{comparison morphisms}. 

\begin{defi}
  \label{mhc-morph}
  Given mixed Hodge complexes $A$ and $B$, a \textit{morphism of mixed Hodge complexes} from $A$ to $B$ consists of 
  \begin{enumerate}
    \item a morphism of filtered complexes $f_\Bbbk : (A_\Bbbk,W) \to (B_\Bbbk,W)$ over $\Bbbk$,
    \item a morphism of bifiltered complexes $f_\CC : (A_\CC,W,F) \to (B_\CC,W,F)$ over $\CC$ and
    \item morphisms of filtered complexes $f_i : (A_i,W) \to (B_i,W)$ over $\CC$ 
  \end{enumerate}  
  such that $f_0 = f_\Bbbk \otimes \CC$, $f_s = f_\CC$ and the following squares commute for all arrows $u : i \to j$:
    \begin{equation*}
      \begin{tikzcd}
        A_i \arrow[r,"f_i"] \arrow[d, "\varphi^A_{ij}"] & B_i \arrow[d, "\varphi^B_{ij}"] \\
        A_j \arrow[r, "f_\CC"] & B_j
      \end{tikzcd}
    \end{equation*}
\end{defi}
Denote by $\mathsf{MHC}_s$ be the category of mixed Hodge complexes of length $s$. It is a symmetric monoidal category with tensor $A \otimes B$ defined component-wise by $A_i \otimes B_i$ and with comparison morphisms $\varphi^A_{u} \otimes \varphi^B_{u}$.

Given a mixed Hodge complex $A$, the \textit{cohomology of $A$} is the mixed Hodge complex $H^*(A)$, defined by
\begin{equation}
  \label{mhc-coho}
  H^*(A) = ((H^*(A_\Bbbk),W) \xleftrightarrow{\varphi^*} (H^*(A_\CC),W,F)),
\end{equation}
with the induced filtrations and comparison morphisms on cohomology. Given a morphism $f: A \to B$ of mixed Hodge complexes, the morphism $f^* : H^*(A) \to H^*(B)$, given level-wise by the induced morphisms of $f$ on cohomology, is a morphism of mixed Hodge complexes.
A morphism $f: A \to B$ is said to be a \textit{quasi-isomorphism} if $f^* : H^*(A) \to H^*(B)$ is an isomorphism. 

There is an inclusion of categories $\ch{\MHS} \hookrightarrow \mathsf{MHC}_s$, where a cochain complex of mixed Hodge structures is seen as a mixed Hodge complex with identity comparison morphisms. This inclusion induces an equivalence between the derived category $\mathsf{D}(\MHS)$ and the category $\Ho(\mathsf{MHC}_s)$ of mixed Hodge complexes localized at quasi-isomorphisms (see \cite{Beilinson}, see also \cite{CG2}). In particular, the categories $\Ho(\MHC_s)$ are all equivalent for different lengths $s$. For simplicity, we shall drop the $s$ from the notation unless needed.
\begin{rema}
  \label{del-vs-bei}
  Definition \ref{mhc-defi} does not exactly coincide with Deligne's definition of a mixed Hodge complex (cf. \cite{DeHIII}). In Deligne's definition, axioms (\ref{mhc-axiom2}) and (\ref{mhc-axiom3}) are replaced with
  \begin{itemize}
    \item[(2')] For each $k \in \ZZ$, the filtered complex $(Gr_k^W(A_\CC),F)$ is $d$-strict,
    \item[(3')] For all $n \in \NN$ and $p \in \ZZ$, the filtration induced by $F$ on $H^{n}(Gr_p^W A_\CC)$ and the maps $\varphi_u^*$ induce a pure Hodge structure of weight $p+n$ on $H^n(Gr_p^W A_\Bbbk)$.
  \end{itemize}
  Instead, our definition coincides with Beilinson's notion of absolute Hodge complex \cite{Beilinson}. There is, however, a functor from one category to the other. Denote by $\MHC'$ the category of mixed Hodge complexes as defined by Deligne. There exists a functor
  \[ \mathrm{Dec} : \mathsf{MHC}' \to \MHC, \]
  called the \textit{decalage functor}, defined by changing the weight filtration of each $A_i$ in the following way
  \[ \mathrm{Dec}(W)_kA^n = W_{k-n} A^n \cap d^{-1}(W_{k-n-1}A^{n+1}). \]
  This functor becomes an equivalence after localizing at quasi-isomorphisms (see section $4$ of \cite{CG2}).
\end{rema}

\subsection{Operadic mixed Hodge diagrams}
\label{mhd-sec}
We now consider a multiplicative version of mixed Hodge complexes and introduce the notion of ho-morphism by allowing homotopy commutativity of diagrams.
\begin{defi}
  \label{mhd-defi}
  A $\Pp$-\textit{mixed Hodge diagram} is a mixed Hodge complex 
  \[ A = ((A_\Bbbk,W) \xleftrightarrow{\varphi} (A_\CC,W,F)) \]
  together with the data of filtered $\Pp$-algebra structures on $(A_\Bbbk,W)$ and $(A_i,W)$ and a bifiltered $\Pp$-algebra structure on $(A_\CC,W,F)$ such that the comparison morphisms $\varphi_{u}$ are maps of $\Pp$-algebras.
\end{defi}
Morphisms of $\Pp$-mixed Hodge diagrams are morphisms of mixed Hodge complexes which are component-wise morphisms of $\Pp$-algebras.
\begin{rema}
  \label{PopMhc-rema}
  Consider the inclusion $\ch{\Bbbk} \hookrightarrow \ch{\MHS} \hookrightarrow \MHC$ given by endowing a cochain complex with a trivial weight and Hodge filtrations in weight $0$. Given a dg operad $\Pp$ over $\Bbbk$, we see $\Pp$ as an operad in $\MHC$ under this inclusion. Then, given $A \in \MHC$,
  \[ \Pp(A) = \bigoplus_n \Pp(n) \otimes_{\SS_n} A^{\otimes n} \]
  is itself a mixed Hodge complex. A $\Pp$-mixed Hodge diagram is equivalently defined by a mixed Hodge complex $A$ together with a map
  \[\gamma : \Pp(A) \to A \]
  such that the usual $\Pp$-algebra axioms are satisfied. 
\end{rema}
\begin{defi}
  \label{Pmhalg-defi}
  A \textit{$\Pp$-mixed Hodge algebra} is a $\Pp$-algebra in complexes of mixed Hodge structures.
\end{defi}
Denote by $\Pp\textrm{-}\mathsf{MHD}$ the category of $\Pp$-mixed Hodge diagrams and by $\Pp\textrm{-}\mathsf{MHalg}$ the subcategory of $\Pp$-mixed Hodge algebras.
\begin{rema}
  In the case $\Pp = Com$, mixed Hodge diagrams play a fundamental role for endowing rational homotopy types of complex algebraic varieties with mixed Hodge structures (see \cite{morgan}, \cite{Hain}, \cite{navarro}, \cite{CG1}). 
\end{rema}
Let $\Cc$ be a cooperad over $\Bbbk$. 
\begin{defi}
  \label{def-mhd}
  A $\Cc$-\textit{mixed Hodge diagram} is a mixed Hodge complex 
  \[ C = ((C_\Bbbk,W) \xleftrightarrow{\varphi} (C_\CC,W,F)) \] 
  together with the structures of filtered $\Cc$-coalgebras on $(C_\Bbbk,W)$ and $(C_i,W)$ and of a bifiltered $\Cc$-coalgebra on $(C_\CC,W,F)$ such that the comparison maps $\varphi_{ij}$ are maps of $\Cc$-coalgebras.
\end{defi}
To define a cobar functor $\Omega_\kappa$ for mixed Hodge diagrams, we need to impose extra conilpotency conditions. Recall the coradical filtration of a $\Cc$-coalgebra $C$:
\[ R_0C = 0, \quad R_n(C) = \{ x \in C \, | \Delta_C(x)_k = 0, \text{ for } k > n \}, \] 
for $n \geq 1$ and where $\Delta_C(x)_k$ is the $k$-th component of 
\[ \Delta_C(x) \in \prod_{n \geq 1} (\Cc(n) \otimes C^{\otimes n})^{\SS_n}. \]
Recall also that $C$ is conilpotent if $R$ is exhaustive.
\begin{defi}
  A $\Cc$-mixed Hodge diagram $C$ is \textit{conilpotent} if the underlying $\Cc$-coalgebras $C_i$ of its components are conilpotent and for each arrow $u : i \to j$, the comparison morphism is a bifiltered quasi-isomorphism
  \[ \varphi_u : (C_i,W,R) \to (C_j,W,R), \]
  with respect to the weight and coradical filtrations.
\end{defi}
We assume henceforth all $\Cc$-mixed Hodge diagrams are conilpotent and denote the category of such by $\Cc\textrm{-}\mathsf{MHD}$. 
\begin{rema}
  The conilpotency condition on the comparison morphisms is a technical detail arising from the following facts. Given a dg operad $\Pp$ the bar functor $\mathrm{B}_\kappa$ sends quasi-isomorphisms of $\Pp$-algebras to filtered quasi-isomorphisms of $\Pp^\antishriek$-coalgebras with respect to the coradical filtration. Moreover, the cobar functor $\Omega_\kappa$ sends filtered quasi-isomorphisms to quasi-isomorphisms of $\Pp$-algebras, thus such filtered quasi-isomorphisms form a subclass of the weak equivalences of $\Pp^\antishriek$-coalgebras (see section $2.3$ of \cite{Vall-homalg}). The proof of this last fact uses a certain filtration which we can use to prove that the bar-cobar adjunction of mixed Hodge diagrams, introduced in the following section, is well-defined.
\end{rema}
Given a $\Pp$-mixed Hodge diagram $A$, its cohomology $H^*(A)$ is again a $\Pp$-mixed Hodge diagram. Likewise, the cohomology of a $\Cc$-mixed Hodge diagram is as well a $\Cc$-mixed Hodge diagram.
\begin{defi}
  A morphism $f: A \to A'$ between $\Pp$-mixed Hodge diagrams is said to be a \textit{quasi-isomorphism} if $f$ is a quasi-isomorphism of the underlying mixed Hodge complexes. A morphism $f : C \to C'$ of $\Cc$-mixed Hodge diagrams is a \textit{weak-equivalence} if all components $f_i$ are weak-equivalences of filtered $\Cc$-coalgebras.
\end{defi}
We now define a notion of morphism up to homotopy that we shall later need (see \cite{CG1}).
\begin{defi}
  \label{ho-morph-defi} 
  Given $A,B \in \Pp\textrm{-}\mathsf{MHD}$, a \textit{ho-morphism} $f: A \to B$ is composed by
  \begin{enumerate}
    \item \label{ho-morph1} a morphism of filtered $\Pp$-algebras $f_\Bbbk : (A_\Bbbk,W) \to (B_\Bbbk,W)$ over $\Bbbk$,
    \item \label{ho-morph2} a morphism of bifiltered $\Pp$-algebras $f_\CC : (A_\CC,W,F) \to (B_\CC,W,F)$ over $\CC$,
    \item for each $0 \leq i \leq s$, a morphism of filtered $\Pp$-algebras $f_i : (A_i,W) \to (B_i,W)$ over $\CC$ such that $f_0 = f_\Bbbk \otimes \CC$ and $f_s = f_\CC$,
    \item for each $u : i \to j$, a homotopy of filtered $\Pp$-algebras $h_{u}: (A_i,W) \otimes \CC \to (B_j,W) \otimes I$ (recall Definition \ref{hom-filmorph}) making the following square commute up to homotopy:
      \begin{equation*}
        \begin{tikzcd}
          A_i \arrow[r,"f_i"] \arrow[d, "\varphi^A_{u}"] & B_i \arrow[d, "\varphi^B_{u}"] \\
          A_j \arrow[r, "f_j"] & B_j.
        \end{tikzcd}
      \end{equation*}
  \end{enumerate}
\end{defi}
A ho-morphism $f: A \to B$ is said to be a \textit{quasi-isomorphism} if all components $f_i$ are filtered quasi-isomorphisms.
\begin{prop}
  \label{homorph-fact}
  Given a ho-morphism $f : A \to B$ of $\Pp$-mixed Hodge diagrams, there is a $\Pp$-mixed Hodge diagram $\mathrm{Path}(f)$ and a commuting diagram
  \[ 
    \begin{tikzcd}
      A & \arrow[l,"p_f"'] \mathrm{Path}(f) \arrow[r,"q_f"] & B \\
      & \arrow[ul,"id"'] \arrow[u,"\iota_f"] A \arrow[ur,"f"] & 
    \end{tikzcd}
  \]
  where $\iota_f$ and $f$ are ho-morphisms and the rest are strict morphisms of $\Pp$-mixed Hodge diagrams. In addition, 
  \begin{enumerate}
    \item The morphisms $p_f$ and $\iota_f$ are quasi-isomorphisms,
    \item If $f$ is a quasi-isomorphism, then $q_f$ is also a quasi-isomorphism.
  \end{enumerate}
\end{prop}
\begin{proof}
  We adapt the proof of Proposition 4.16 of \cite{Cir-cofmodel} for commutative algebras to general $\Pp$-algebras. We define the objects in the statement and refer the reader to loc. cit. for the proof that these are well-defined diagrams and morphisms. Denote by $f_i$ the component-wise morphisms of $f$ and by $h_{u} : A_i \to B_j \otimes I$ its homotopies. The $\Pp$-mixed Hodge diagram $\mathrm{Path}(f)$ is given by
  \[ \mathrm{Path}(f) = (\mathrm{Path}(f_\Bbbk) \xleftrightarrow{\varphi} \mathrm{Path}(f_\CC)), \]
  where $\mathrm{Path}(f_i)$ is the fiber product $\mathrm{Path}(f_i) = A_i \times_{B_i} (B_i \otimes I)$ and the morphism $\varphi_{ij}$ is the morphism filling the dotted arrow in
  \[ 
    \begin{tikzcd}
      \mathrm{Path}(f_i) \arrow[r, "\pi_1"] & A_i \arrow[rd,dotted] \arrow[rrd,"h_{u}",bend left] \arrow[rdd,"\varphi^A"',bend right] & & \\
      & & \mathrm{Path}(f_j) \arrow[r,"\pi_2"] \arrow[d,"\pi_1"] & B_j \otimes I \arrow[d,"\delta_0^{B_j}"] \\
      & & A_j \arrow[r,"f_j"] & B_j.
    \end{tikzcd}
  \]
  Here, the morphisms $\pi_i$ denote the corresponding projections and $\delta_0^{B_j}$ is evaluation at $t = 0$. The morphism $p_f$ is given level-wise by $(p_f)_i = \pi_1 : \mathrm{Path}(f_i) \to A_i$, so it is a quasi-isomorphism. The morphism $q_f$ is given level-wise by $(q_f)_i = \delta_0 \pi_2$. Finally, $\iota_f$ is defined level-wise by $(\iota_f)_i = (id_{A_i}, \iota_{B_i} f_i) : A_i \to \mathrm{Path}(f_i)$, where $\iota_{B_i} : B_i \to B_i \otimes I$ is the inclusion into constant polynomials. See Proposition 4.16 \cite{Cir-cofmodel} for the definition of the homotopies of $\iota_f$.
\end{proof}

\subsection{A bar-cobar adjunction}
\label{mhd-barcobar-sec}
We promote the bar-cobar adjunction to $\Pp$-mixed Hodge diagrams. Let $\Pp$ be a Koszul operad and recall the canonical twisting morphism $\kappa : \Pp^\antishriek \to \Pp$. Given $A \in \Pp\textrm{-}\mathsf{MHD}$ and $C \in \Pp^\antishriek \textrm{-}\mathsf{MHD}$, define
\begin{align}
  &\mathrm{B}_\kappa A := (\mathrm{B}_\kappa (A_\Bbbk,W),\mathrm{B}_\kappa (A_\CC,W,F),\mathrm{B}_\kappa \varphi), \label{bar-eq}\\
  &\Omega_\kappa C := (\Omega_\kappa (C_\Bbbk,W),\Omega_\kappa (C_\CC,W,F),\Omega_\kappa \varphi). \nonumber
\end{align}
Here, we use the filtered and bifiltered versions of the bar construction (see (\ref{fil-barcobar})). In Proposition \ref{mhd-barcobar} below, we prove that $\mathrm{B}_\kappa(A)$ is a $\Pp^\antishriek$-mixed Hodge diagram and $\Omega_\kappa(C)$ is a $\Pp$-mixed Hodge diagram and show that the resulting functors give a bar-cobar adjunction.

\begin{rema}
  In \cite{Hain}, Hain already gives a bar construction for associative mixed Hodge diagrams, but using Deligne's notion of mixed Hodge complex $\MHC'$ (see Remark \ref{del-vs-bei}). Given $A$ an associative monoid in $\MHC'$, he defines 
  \[ W_k \mathrm{B}'(A_i)^m = \bigoplus_{n \geq 1} W_{k-n} (A_i^{\otimes n})^{m+n} \]
  for all $0 \leq i \leq s$. The Hodge filtration of $A_\CC$ is just the free extension to tensor powers, like in our setting. The only difference between Hain's definition and ours is in the weight filtration. It is equivalently given by considering the spaces $Ass^\antishriek(n)$ to have weight equal to homological degree (so they are elements of $\MHC'$) and then freely extending the weight filtration.
\end{rema}
Note that in the definition of a mixed Hodge complex $A$, the weight and Hodge filtrations are required to be biregular. However, the filtrations of $\mathrm{B}_\kappa(A)$ are not, in general, biregular. The bar and cobar functors thus land in $\mathrm{Ind}(\MHC)$, the ind completion of mixed Hodge complexes. This category is obtained by formally adjoining filtered colimits of mixed Hodge complexes. We prove, in Theorem \ref{mhd-barcobar}, that the functors $\Omega_\kappa$ and $\mathrm{B}_\kappa$ land in a subcategory of $\mathrm{Ind}(\MHC)$ whose objects are colimits of mixed Hodge complexes indexed by the poset $\NN$.
\begin{defi}
  An \textit{$\NN$-filtered mixed Hodge complex} of length $s$ is given by the data $A = ((A_\Bbbk,W,L),(A_i,W,L),(A_\CC,W,F,L),\varphi_{ij})$ of a bifiltered cochain complex $(A_\Bbbk,W,L)$ over $\Bbbk$, a trifiltered cochain complex $(A_\CC,W,F,L)$ over $\CC$, bifiltered complexes $(A_i,W,L)$ over $\CC$ for $1 \leq i \leq s-1$ and bifiltered morphisms $\varphi_{u}$ for each arrow $u : i \to j$ as in the following diagram:
  \[  \begin{tikzcd}[column sep = tiny]
    & (A_1,W,L) & & \cdots & \\
    (A_0,W,L) \arrow[ru,"\varphi_{01}"] & & (A_2,W,L) \arrow[lu,"\varphi_{21}"'] \arrow[ru,"\varphi_{23}"] & & \arrow[lu,"\varphi_{s s-1}"'] (A_s,W,L),
  \end{tikzcd} \]
  where $(A_0,W,L) = (A_\Bbbk,W,L) \otimes \CC$ and $(A_s,W,L) = (A_\CC,W,L)$.
  In addition, the filtrations $L$ are positive, ascending and exhaustive and, for each $k \geq 1$, the restriction \[ L_k A = ((L_k A_\Bbbk,W),(L_k A_i,W),(L_k A_\CC,W,F),\varphi_{ij}) \] is a mixed Hodge complex.
\end{defi}
\begin{defi}
  Given $\NN$-filtered mixed Hodge complexes $A$ and $B$, a \textit{morphism of $\NN$-filtered mixed Hodge complexes} $f : A \to B$ consists of 
  \begin{enumerate}
    \item a morphism $f_\Bbbk : (A_\Bbbk,W,L) \to (B_\Bbbk,W,L)$ over $\Bbbk$, compatible with $W$ and $L$,
    \item a morphism $f_\CC : (A_\CC,W,F) \to (B_\CC,W,F)$ over $\CC$, compatible with $W$ and $L$,
    \item morphisms $f_i : (A_i,W,F,L) \to (B_i,W,F,L)$ over $\CC$, compatible with $W$, $F$ and $L$,
  \end{enumerate}  
  such that $f_0 = f_\Bbbk \otimes \CC$, $f_s = f_\CC$ and the following squares commute for all arrows $u : i \to j$:
    \begin{equation*}
      \begin{tikzcd}
        A_i \arrow[r,"f_i"] \arrow[d, "\varphi^A_{ij}"] & B_i \arrow[d, "\varphi^B_{ij}"] \\
        A_j \arrow[r, "f_\CC"] & B_j
      \end{tikzcd}
    \end{equation*}
\end{defi}
A morphism $f : A \to B$ of $\NN$-filtered mixed Hodge complexes is a \textit{quasi-isomorphism} if for each $k \geq 1$, the restriction $L_k f$ is a quasi-isomorphism of mixed Hodge complexes. Denote the category of $\NN$-filtered mixed Hodge complexes by $\MHC^{\NN\textrm{-}\mathrm{fil}}$. There are inclusions of categories
\[ \MHC \hookrightarrow \MHC^{\NN\textrm{-}\mathrm{fil}} \hookrightarrow \mathrm{Ind}(\MHC) \]
and all of them are symmetric monoidal. Denote by $\Pp\textrm{-}\MHD^{\NN\textrm{-}\mathrm{fil}}$ the category of $\Pp$-algebras in $\NN$-filtered mixed Hodge complexes. Moreover, an $\NN$-filtered $\Pp^\antishriek$-mixed Hodge diagram $C$ is said to be \textit{conilpotent} if, for every $k \geq 1$, $L_k C$ is a conilpotent $\Pp^\antishriek$-mixed Hodge diagram. Likewise, a morphism $f : C \to D$ of conilpotent $\NN$-filtered $\Pp^\antishriek$-mixed Hodge diagrams is a \textit{weak equivalence} if, for every $k \geq 1$, $L_k f$ is a weak equivalence of conilpotent $\Pp^\antishriek$-mixed Hodge diagrams.

\begin{theo}
  \label{mhd-barcobar}
  The functors $\mathrm{B}_\kappa$ and $\Omega_\kappa$ in (\ref{bar-eq}) are well-defined and form an adjoint pair:
  \[ \Omega_\kappa : \Pp^\antishriek\textrm{-}\mathsf{MHD}^{\NN\textrm{-}\mathrm{fil}} \rightleftharpoons \Pp\textrm{-}\mathsf{MHD}^{\NN\textrm{-}\mathrm{fil}} : \mathrm{B}_\kappa. \]
  Moreover:
  \begin{enumerate}
   \item \label{mhd-barcobar1} For every $A \in \Pp\textrm{-}\mathsf{MHD}^{\NN\textrm{-}\mathrm{fil}}$, the counit $\Omega_\kappa \mathrm{B}_\kappa(A) \to A$ is a quasi-isomorphism of $\NN$-filtered $\Pp$-mixed Hodge diagrams.
   \item \label{mhd-barcobar2} For every $C \in \Pp^\antishriek\textrm{-}\mathsf{MHD}^{\NN\textrm{-}\mathrm{fil}}$, the unit $C \to \mathrm{B}_\kappa \Omega_\kappa(C)$ is a weak-equivalence of $\NN$-filtered $\Pp^\antishriek$-mixed Hodge diagrams.
   \item \label{mhd-barcobar3} The functor $\mathrm{B}_\kappa$ sends quasi-isomorphisms to weak-equivalences and the functor $\Omega_\kappa$ sends weak-equivalences to quasi-isomorphisms.
  \end{enumerate}
  \end{theo}
\begin{proof}
  We first prove that $\mathrm{B}_\kappa$ is a well-defined functor. To do so, we prove that $\mathrm{B}_\kappa(A)$ is a conilpotent $\NN$-filtered $\Pp^\antishriek$-mixed Hodge diagram for $A \in \Pp\textrm{-}\MHD^{\NN\textrm{-}\mathrm{fil}}$. We must show that for each $0 \leq i \leq s$, the filtration $\overline{L}$ of $\mathrm{B}_\kappa(A_i)$ freely extended by $L$ satisfies the following conditions for each $k \geq 1$:
  \begin{enumerate}[(i)]
    \item \label{mhd-barcobar-proof-1} the bifiltered complex $(\overline{L}_k\mathrm{B}_\kappa(A_i),W,F)$ are biregular,
    \item \label{mhd-barcobar-proof-2} for every arrow $u : i \to j$, the map $\overline{L}_k \mathrm{B}_\kappa(\varphi_u)$ is a bifiltered quasi-isomorphism with respect to the weight and coradical filtrations,
    \item \label{mhd-barcobar-proof-3} the bifiltered complex $(\overline{L}_k \mathrm{B}_\kappa(A_\CC),W,F)$ is $d$-bistrict,
    \item \label{mhd-barcobar-proof-4} the cohomology $(H^*(\overline{L}_k \mathrm{B}_\kappa(A_\Bbbk)),W)$ with the Hodge filtration $F$ induced by the one of $H^*(\overline{L}_k \mathrm{B}_\kappa(A_\CC))$ and the maps $(\overline{L}_k \mathrm{B}_\kappa(\varphi_u))^*$ is a mixed Hodge structure.
  \end{enumerate}
  Since $L$ is positive, ascending and exhaustive, the same follows for $\overline{L}$. Moreover, since $L$ is positive, $\overline{L}_k \mathrm{B}_\kappa(A_i)$ is a finite sum of tensor powers of complexes of the form $L_m A_i$ for $m \leq k$ and so point (\ref{mhd-barcobar-proof-1}) follows. We give a spectral sequence argument that shows the other three points. The cooperad $\Pp^\antishriek$ has a grading $\Pp^\antishriek = \bigoplus_{r} \Pp^\antishriek_{(r)}$ induced by the cofree cooperad, where the elements in $\Pp^\antishriek_{(r)}$ are compositions of $r$ generators. For every $0 \leq i \leq s$, the coradical filtration on $\Pp^\antishriek(A_i)$ is given by
  \[ R_n\mathrm{B}_\kappa(A_i) = \bigoplus_{k \leq n} \Pp^\antishriek_{(k)}(A_i). \]
  Recall that the codifferential of $\mathrm{B}_\kappa(A_i)$ has two components $d_A + d_\mu$, the former induced by the differential of $A$, which preserves the grading of $\Pp^\antishriek$; the latter given by the coderivation extending
  \[ \Pp^\antishriek(A) \xrightarrow{\kappa \circ 1} \Pp(A) \xrightarrow{\mu} A,\]
  where $\mu : \Pp(A) \to A$ is the product structure of $A$, which lowers the grading of $\Pp^\antishriek$. Consider now the spectral sequence associated to $\overline{L}_k \mathrm{B}_\kappa(A_i)$ together with $R$ restricted to this complex. The $E_0$ page is given by
  \[ E_0^{-l,m}(\overline{L}_k \mathrm{B}_\kappa(A_i);R) = Gr^R_l \overline{L_k} \mathrm{B}_\kappa(A_i)^{m-l} = \overline{L}_k \Pp^\antishriek_{(l)}(A_i)^{m-l} \]
  and $d_0 = d_A$. We will now show that the higher pages $(E_r,d_r)$ are complexes of mixed Hodge structures for $r \geq 1$. Note that, by definition of $\overline{L}_k$, the complexes $\overline{L}_k \Pp^\antishriek_{(l)}(A)$ are finite direct sums of tensor powers of complexes of the form $L_m A$ with $m \leq k$ and each of these is a mixed Hodge complex. For each arrow $u : i \to j$, the comparison morphism at $E_0$ is given by
  \[ \overline{L}_k \Pp^\antishriek_{(l)}(\varphi_u): E_0^{-l,\bullet}(\overline{L}_k \mathrm{B}_\kappa(A_i);R) \to E_0^{-l,\bullet}(\overline{L}_k \mathrm{B}_\kappa(A_j);R). \]
  This morphism is composed of a finite sum of comparison morphisms of mixed Hodge complexes. Hence, $ \overline{L}_k \Pp^\antishriek_{(l)}(\varphi_u)$ is a filtered quasi-isomorphism with respect to the induced weight filtration. This shows (\ref{mhd-barcobar-proof-2}). The filtered complexes $(E_0(\overline{L}_k \mathrm{B}_\kappa(A_i);R),W)$ are $d$-strict, hence
  \[ Gr^W_n E_1( \overline{L}_k \mathrm{B}_\kappa(A_i);R) \cong E_1(Gr_n^W \overline{L}_k \mathrm{B}_\kappa(A_i);R). \] 
  Moreover, the Hodge filtration of $E_0^{-l,\bullet}(\overline{L}_k \mathrm{B}_\kappa(A_\CC);R)$ coincides with the one of 
  \[ \overline{L}_k \Pp^\antishriek_{(l)} (A_\CC). \] 
  Hence, the bifiltered complex $(E_0(\overline{L}_k \mathrm{B}_\kappa(A_\CC);R),W,F)$ is $d$-bistrict. In particular, we have that
  \[ E_1(Gr_F^p Gr_n^W \overline{L}_k \mathrm{B}_\kappa(A_\CC);R) \cong Gr^p_F Gr_n^W E_1(\overline{L}_k \mathrm{B}_\kappa(A_\CC);R). \]
  The comparison morphism induces an isomorphism of filtered complexes 
  \[ (E_1(\overline{L}_k \mathrm{B}_\kappa(A_i);R),W) \cong (E_1(\overline{L}_k \mathrm{B}_\kappa(A_j);R),W) \]
  and so the complex $E_1(\overline{L}_k \mathrm{B}_\kappa(A_\Bbbk);R)$ has mixed Hodge structures induced by the Hodge filtration of $E_1(\overline{L}_k \mathrm{B}_\kappa(A_\CC);R)$ and the comparison morphisms. This also implies that the differential $d_1$ of the $E_1$ page is a map of mixed Hodge structures and so it is strict with respect to $W$ and $F$. The vector spaces of the $E_2$ page are subquotients of mixed Hodge structures, so they are themselves mixed Hodge structures. The comparison morphisms induce isomorphisms on the higher pages, so the differentials $d_2$ are also morphisms of mixed Hodge structures. The same follows for all the higher pages and so all pages $(E_r,d_r)$ are complexes of MHS's.
  In particular, we have for $r \geq 0$,
  \[ E_r(Gr_n^W \overline{L}_k \mathrm{B}_\kappa(A_i);R) \cong Gr_n^W E_r(\overline{L}_k \mathrm{B}_\kappa(A_i);R), \]
  \[ E_r(Gr_F^p Gr_n^W \overline{L}_k \mathrm{B}_\kappa(A_\CC);R) \cong Gr^p_F Gr_n^W E_r(\overline{L}_k \mathrm{B}_\kappa(A_\CC);R). \]
  Note that Deligne defines several weight and Hodge filtrations on the $E_r$ page. The ones we consider correspond to Deligne's inductive filtrations (see \cite{DeHII}). These isomorphisms at the $E_\infty$ page translate to
  \[ H^*(Gr_n^W \overline{L}_k \mathrm{B}_\kappa(A_i)) \cong Gr_n^W H^*(\overline{L}_k \mathrm{B}_\kappa(A_i)), \]
  \[ H^*(Gr_F^pGr_n^W \overline{L}_k \mathrm{B}_\kappa(A_\CC)) \cong Gr^p_FGr_n^W H^*(\overline{L}_k \mathrm{B}_\kappa(A_\CC)). \]
  This shows $d$-bistrictness of $(\overline{L}_k \mathrm{B}_\kappa(A_\CC),W,F)$, that is, point (\ref{mhd-barcobar-proof-3}). Since the comparison morphisms induce filtered isomorphisms on $E_r$ for every $r \geq 1$, they induce filtered isomorphisms on cohomology. Finally, since all pages $E_r(\overline{L}_k \mathrm{B}_\kappa(A_\Bbbk);R)$ carry mixed Hodge structures, the same follows for the cohomology $H^*(\overline{L}_k \mathrm{B}_\kappa(A_\Bbbk))$ and point (\ref{mhd-barcobar-proof-4}) follows.  This shows that $\mathrm{B}_\kappa(A)$ is a conilpotent $\NN$-filtered $\Pp^\antishriek$-mixed Hodge diagram. A similar proof shows that $\Omega_\kappa(C)$ is an $\NN$-filtered $\Pp$-mixed Hodge diagram, but using now the following filtration, in place of the coradical filtration:
  \[ T_n \Omega_\kappa(C) = \sum_{\stackrel{k \geq 1}{n_1 + \cdots + n_k \leq n}} \Pp^\antishriek(k) \otimes_{\SS_k} R_{n_1} C \otimes \cdots \otimes R_{n_k} C,\]
  where $R$ is the coradical filtration of $C$ (see section $2.3$ of \cite{Vall-homalg}). We now prove that $\mathrm{B}_\kappa$ and $\Omega_\kappa$ are an adjoint pair. The functors $\mathrm{B}_\kappa$ and $\Omega_\kappa$ are given level-wise by their filtered (and bifiltered) counterparts (see (\ref{fil-barcobar})). Hence, for every $0 \leq i \leq s$, $C \in \Pp^\antishriek\textrm{-}\mathsf{MHD}^{\NN\textrm{-}\mathrm{fil}}$ and $A \in \Pp\textrm{-}\mathsf{MHD}^{\NN\textrm{-}\mathrm{fil}}$, there are natural isomorphisms
  \[ \Hom_{\Pp\textrm{-}\mathsf{Falg}}(\Omega_\kappa C_i,A_i) \cong \Hom_{\Pp^\antishriek\textrm{-}\mathsf{Fcoalg}}(C_i,\mathrm{B}_\kappa(A_i)). \]
  By naturality, these isomorphisms are compatible with the comparison morphisms and so there are natural isomorphisms 
  \[ \Hom_{\Pp\textrm{-}\mathsf{MHD}^{\NN\textrm{-}\mathrm{fil}}}(\Omega_\kappa C,A) \cong \Hom_{\Pp^\antishriek\textrm{-}\mathsf{MHD}^{\NN\textrm{-}\mathrm{fil}}}(C,\mathrm{B}_\kappa(A)), \]
  which proves the adjunction. Finally, point (\ref{mhd-barcobar1}) of the claim follows from the fact that the counit $\Omega_\kappa \mathrm{B}_\kappa(A) \to A$ is given level-wise by the counits of the filtered adjunctions, which are filtered quasi-isomorphisms by Proposition \ref{filcounit-prop}. A similar argument shows that the unit is a weak equivalence and Proposition \ref{barcobar-qisoweq-prop} proves that $\mathrm{B}_\kappa$ sends quasi-isomorphisms to weak-equivalences and $\Omega_\kappa$ sends weak-equivalences to quasi-isomorphisms.
\end{proof}

\subsection{Homotopy models for mixed Hodge diagrams}
In this section, we use the framework of infinity algebras to give homotopy models for mixed Hodge diagrams. In the following, let us fix a Koszul operad $\Pp$. Recall its cofibrant resolution $\Pp_\infty = \Omega \Pp^\antishriek$.

\begin{defi}
  \label{Pinftymodel-defi}
  A \textit{$\Pp_\infty$-mixed Hodge diagram} is a mixed Hodge complex 
  \[ A = ((A_\Bbbk,W) \xleftrightarrow{\varphi} (A_\CC,W,F)) \] 
  together with  
  \begin{itemize}
    \item a filtered $\Pp_\infty$-algebra structure $m_\Bbbk$ on $(A_\Bbbk,W)$,
    \item a bifiltered $\Pp_\infty$-algebra structure $m_\CC$ on $(A_\CC,W,F)$,
    \item for each $0 \leq i \leq s$, a filtered $\Pp_\infty$-algebra structure $m_i$ on $(A_i,W)$ such that $m_0 = m_\Bbbk \otimes \CC$ and $m_s = m_\CC$,
    \item for each $u : i \to j$, a filtered $\infty$-morphism $\hat{\varphi}_u : (A_i,W) \otimes \CC \to (A_j,W)$ extending $\varphi_{u}$.
  \end{itemize}
\end{defi}

The morphisms $\hat{\varphi}_u$ in the previous definition are called the \textit{comparison morphisms}. We will occasionally denote a $\Pp_\infty$-mixed Hodge diagram by $(A,m)$ to ease notation.

\begin{rema}
  In the case of $\Pp = Lie$, our definition of $L_\infty$-mixed Hodge diagrams generalises the notion mixed Hodge diagrams of $L_\infty$ algebras, defined in \cite{LinfMHD}. In the latter, comparison morphisms are strict morphisms of $L_\infty$-algebras instead of $\infty$-morphisms. We need the extra flexibility of $\infty$-morphisms to have a homotopy transfer theorem in this setting (Theorem \ref{htt-mhd}).
\end{rema}

\begin{defi}
  Given $\Pp_\infty$-mixed Hodge diagrams $(A,m^A)$ and $(B,m^B)$, a \textit{morphism of $\Pp_\infty$-mixed Hodge diagrams} $f : A \to B$ is
  \begin{itemize}
    \item a filtered $\infty$-morphism $f_\Bbbk : (A_\Bbbk,W,m^A_\Bbbk) \to (B_\Bbbk,W,m^B_\Bbbk)$,
    \item a bifiltered $\infty$-morphism $f_\CC : (A_\CC,W,F,m^A_\CC) \to (B_\CC,W,F,m^B_\CC)$,
    \item for each $0 \leq i \leq s$, filtered $\infty$-morphisms $f_i : (A_i,W,m_i^A) \to (B_i,W,m_i^B)$ such that $f_0 = f_\Bbbk \otimes \CC$ and $f_s = f_\CC$,
  \end{itemize}
  making the following square of $\infty$-morphisms commute
  \begin{equation*}
    \begin{tikzcd}
      A_i  \arrow[r,"f_i"] \arrow[d, "\hat{\varphi}^A_u"] & B_i \arrow[d, "\hat{\varphi}^B_u"] \\
      A_j \arrow[r, "f_j"] & B_j.
    \end{tikzcd}
  \end{equation*}
\end{defi}
\begin{defi}
  A morphism $f : A \to B$ of $\Pp_\infty$-mixed Hodge diagrams is an \textit{$\infty$-quasi-isomorphism} if for each $0 \leq i \leq s$, the $\infty$-morphism $f_i$ is a filtered $\infty$-quasi-isomorphism.
\end{defi}
Denote the category of $\Pp_\infty$-mixed Hodge diagrams by $\Pp_\infty$-$\mathsf{MHD}$. The same definitions apply for $\NN$-filtered mixed Hodge complexes. Note that there is an inclusion 
\[ i : \Pp\textrm{-}\mathsf{MHD}^{\NN\textrm{-}\mathrm{fil}} \hookrightarrow \Pp_\infty\textrm{-}\mathsf{MHD}^{\NN\textrm{-}\mathrm{fil}}. \] 
Recall the universal twisting morphism $\iota : \Pp^\antishriek \to \Pp_\infty$. The same proof as in Proposition \ref{mhd-barcobar} shows that given $A \in \Pp_\infty\textrm{-}\mathsf{MHD}^{\NN\textrm{-}\mathrm{fil}}$, the bar construction
\[ \mathrm{B}_\iota(A) := (\mathrm{B}_\iota (A_\Bbbk,W),\mathrm{B}_\iota (A_\CC,W,F),\mathrm{B}_\iota \varphi) \]
gives a well-defined $\NN$-filtered $\Pp^\antishriek\textrm{-}$mixed Hodge diagram. As in the setting of $\Pp$-algebras, the resulting functor
\[ \mathrm{B}_\iota(A) : \Pp_\infty\textrm{-}\mathsf{MHD}^{\NN\textrm{-}\mathrm{fil}} \to \Pp^\antishriek\textrm{-}\mathsf{MHD}^{\NN\textrm{-}\mathrm{fil}} \]
is fully-faithful.
\begin{prop}
  \label{mhd-inc-barcobar-prop}
  The functors $i$ and $\Omega_\kappa \mathrm{B}_\iota$ form an adjoint pair
  \begin{equation}
    \label{mhd-inc-barcobar-eq}
     \Omega_\kappa \mathrm{B}_\iota : \Pp_\infty\textrm{-}\mathsf{MHD}^{\NN\textrm{-}\mathrm{fil}} \leftrightharpoons \Pp\textrm{-}\mathsf{MHD}^{\NN\textrm{-}\mathrm{fil}} : i
  \end{equation}
  whose unit $\nu_A : A \to \Omega_\kappa \mathrm{B}_\iota(A)$ is an $\infty$-quasi-isomorphism for all $A \in \Pp_\infty\textrm{-}\mathsf{MHD}^{\NN\textrm{-}\mathrm{fil}}$.
\end{prop}
\begin{proof}
  By the analogous adjunction for filtered $\Pp$-algebras (\ref{inc-barcobar}), for all $0 \leq i \leq s$, there are natural isomorphisms
  \[ \Hom_{\Pp\textrm{-}\mathsf{Falg}}(\Omega_\kappa \mathrm{B}_\iota(A),A') \cong \Hom_{\Pp_\infty\textrm{-}\mathsf{Falg}}(A,i(A')). \]
  By naturality, these isomorphisms are compatible with the comparison morphisms and the stated adjunction follows. The unit $\nu_A$ is given component-wise by the units of the component-wise adjunctions, which are all filtered $\infty$-quasi-isomorphisms (and $\nu_{A_\CC}$ is a bifiltered $\infty$-quasi-isomorphism), by Proposition \ref{filunit-prop}. 
\end{proof}
Denote by 
\[ \Ho(\Pp\textrm{-}\mathsf{MHD}^{\NN\textrm{-}\mathrm{fil}}) \quad \text{and} \quad \Ho(\Pp_\infty\textrm{-}\mathsf{MHD}^{\NN\textrm{-}\mathrm{fil}}) \] 
the categories of $\NN$-filtered $\Pp$-mixed Hodge diagrams and $\Pp_\infty$-mixed Hodge diagrams localized at their respective notion of quasi-isomorphism. 
\begin{theo}
  \label{inftymhd-equiv}
  The inclusion $i : \Pp\textrm{-}\mathsf{MHD}^{\NN\textrm{-}\mathrm{fil}} \hookrightarrow  \Pp_\infty\textrm{-}\mathsf{MHD}^{\NN\textrm{-}\mathrm{fil}}$ induces an equivalence of categories
  \[ \Ho(\Pp\textrm{-}\mathsf{MHD}^{\NN\textrm{-}\mathrm{fil}}) \simeq \Ho(\Pp_\infty\textrm{-}\mathsf{MHD}^{\NN\textrm{-}\mathrm{fil}}). \]
\end{theo}
\begin{proof}
  By definition, the functor $i$ sends quasi-isomorphisms to $\infty$-quasi-isomorphisms. By Proposition \ref{filunit-prop}, the functor $\Omega_\kappa \mathrm{B}_\iota$ sends $\infty$-quasi-isomorphisms to quasi-isomorphisms. Hence, the adjunction (\ref{mhd-inc-barcobar-eq}) induces an adjunction between the homotopy categories. Note that for $A \in \Pp\textrm{-}\mathsf{MHD}^{\NN\textrm{-}\mathrm{fil}}$, we have that $\mathrm{B}_\iota(A) \cong \mathrm{B}_\kappa(A)$, so the counit $\Omega_\kappa \mathrm{B}_\iota(A) \to A$ is just the counit of the adjunction in Proposition \ref{mhd-barcobar}, which is a quasi-isomorphism. By Proposition \ref{mhd-inc-barcobar-prop}, the unit is an $\infty$-quasi-isomorphism and the result follows.
\end{proof}

\begin{rema}
  \label{incmhdlength-equiv-rema}
  Note that this is an equivalence between $\Pp$-mixed Hodge diagrams and $\Pp_\infty$-mixed Hodge diagrams of a fixed length $s$. For a given length $s$, there is an inclusion $\Pp\textrm{-}\mathsf{MHD}_s \hookrightarrow \Pp\textrm{-}\mathsf{MHD}_{s+1}$ given by inserting an identity as one of the comparison morphisms. As we prove in Proposition \ref{incmhdlength-equiv-prop} below, the study of isomorphism classes in $\Ho(\Pp\textrm{-}\mathsf{MHD}^{\NN\textrm{-}\mathrm{fil}}_{s+1})$ is reduced to the study of isomorphism classes in $\Ho(\Pp\textrm{-}\mathsf{MHD}^{\NN\textrm{-}\mathrm{fil}}_{s})$. In particular, we can always reduce to mixed Hodge diagrams of length $1$. 
\end{rema}

\subsection{Homotopy transfer}
We prove a version of the homotopy transfer theorem for mixed Hodge diagrams. To do so, we first  define a notion of morphism up to homotopy in $\Pp_\infty$-$\mathsf{MHD}$, analogous to the one of mixed Hodge diagrams.

\begin{defi}
  A \textit{ho-$\infty$-morphism} of $\Pp_\infty$-mixed Hodge diagrams $f : A \to B$ is 
  \begin{itemize}
    \item a filtered $\infty$-morphism $f_\Bbbk : (A_\Bbbk,W) \to (B_\Bbbk,W)$,
    \item a bifiltered $\infty$-morphism $f_\CC : (A_\CC,W,F) \to (B_\CC,W,F)$,
    \item for each $0 \leq i \leq s$, filtered $\infty$-morphisms $f_i : (A_i,W) \to (B_i,W)$ such that $f_0 = f_\Bbbk \otimes \CC$ and $f_s = f_\CC$ and 
    \item for each $u : i \to j$, a gauge $h_u$ (recall Definition \ref{gauge-defi} and equation (\ref{Linftymaps})) between the filtered $\infty$-morphisms $\varphi^B_u f_i$ and $f_j \varphi^A_u$, making the following square commute up to gauge equivalence:
    \begin{equation*}
      \begin{tikzcd}
        A_i \arrow[r,"f_i"] \arrow[d, "\varphi^A_u"'] \arrow[dr, "h_u",Rightarrow] & B_i \arrow[d, "\varphi^B_i"] \\
        A_j \arrow[r, "f_j"] & B_j.
      \end{tikzcd}
    \end{equation*} 
  \end{itemize}
\end{defi}

\begin{defi}
    A ho-$\infty$-morphism $f: A \to B$ is an \textit{isomorphism}, \textit{isotopy} or a \textit{quasi-isomorphism} if all components $f_i$ are $\infty$-isomorphisms, $\infty$-isotopies or $\infty$-quasi-isomorphisms, respectively, of (bi)filtered $\Pp_\infty$-algebras.
\end{defi}

In the case of $\Pp = Com$ and when the source is cofibrant, then ho-morphisms can be composed (see \cite{Cir-cofmodel}). But there is not, in general, a well-defined composition of ho-morphisms of mixed Hodge diagrams. There is, however, a well-defined composition for ho-$\infty$-morphisms, as we now prove.

\begin{prop}
    \label{comphoinfmorph-prop}
    Given two ho-$\infty$-morphisms of $\Pp_\infty$-mixed Hodge diagrams $f : A \to B$ and $g : B \to C$, there is a well defined composition $g \circ f$ which is a ho-$\infty$-morphism whose underlying $\infty$-morphisms $(g \circ f)_i$ are given by 
    \[ (g \circ f)_i = g_i \circledcirc f_i, \qquad \text{for } 0 \leq i \leq s.\]
\end{prop}
\begin{proof}
    It is sufficient to prove that there is a natural choice for a gauge $h_u$ such that
    \[ h_u \cdot (\varphi^C_u \circledcirc g_i \circledcirc f_i) = g_j \circledcirc f_j \circledcirc \varphi^A_u, \]
    for every arrow $u : i \to j$. Consider the following squares commuting up to gauge equivalence:
    \[ \begin{tikzcd}
        A_i \arrow[r,"f_i"] \arrow[d,"\varphi^A_u"'] \arrow[dr, "F_u",Rightarrow] & B_i \arrow[r,"g_i"] \arrow[d,"\varphi^B_u"] \arrow[dr, "G_u",Rightarrow] & C_i \arrow[d,"\varphi^C_u"] \\ A_j \arrow[r,"f_j"] & B_j \arrow[r,"g_j"] & C_j.
    \end{tikzcd}\]
    By Proposition \ref{comp-gauge-prop}, the element $G_u \circledcirc f_i$ is a gauge between 
    \[ \varphi^C_u \circledcirc g_i \circledcirc f_i \quad \text{and} \quad g_j \circledcirc \varphi_u^B \circledcirc f_i.\]
    Similarly, the element $g_j \circledcirc F_u$ is a gauge between
    \[ g_j \circledcirc \varphi_u^B \circledcirc f_i \quad \text{and} \quad g_j \circledcirc f_j \circledcirc \varphi_u^A. \]
    Recall the Deligne-Hinich space $\mathrm{MC}_\bullet(\Hom(\mathrm{B}_\iota(A_i), C_j))$. There is a smaller Kan-complex \[ \gamma_\bullet(\Hom(\mathrm{B}_\iota(A_i), C_j)) \subset \mathrm{MC}_\bullet(\Hom(\mathrm{B}_\iota(A_i), C_j)), \]
    due to Getzler \cite{GetzlerLinf} and whose inclusion is a weak-equivalence of Kan complexes. Moreover, its arrows are the gauges between the Maurer-Cartan elements. Given a horn 
    \[ \Lambda_k^n \to \gamma_\bullet(\Hom(\mathrm{B}_\iota(A_i), C_j)), \] 
    there is a canonical bijection between the set of horn fillers and the space of elements of degree $n$, $\Hom(\mathrm{B}_\iota(A_i), C_j)^n$ (see section $5$ of \cite{higherLie}). Choosing the element $0 \in L^n$, there is thus a canonical horn filler for any horn. In particular, for the horn
    \begin{equation*}
      \begin{tikzcd}[row sep = small]
            & g_j \circledcirc \varphi_u^B \circledcirc f_i \arrow[ddr,"g_j \circledcirc F_u"]  & \\
            &  & \\
            \varphi^C_u \circledcirc g_i \circledcirc f_i \arrow[uur,"G_u \circledcirc f_i"] & & g_j \circledcirc f_j \circledcirc \varphi_u^A,
      \end{tikzcd}
    \end{equation*} 
    there is a canonical choice of a filling which yields the desired gauge.
\end{proof}
We prove the homotopy transfer theorem, more generally, for a subcategory of $\Pp\textrm{-}\MHD^{\NN\textrm{-}\mathrm{fil}}$ that includes all $\Pp$-mixed Hodge diagrams and their bar-cobar resolutions.
\begin{defi}
  An $\NN$-filtered $\Pp$-mixed Hodge complex is \textit{strict} if the filtered complexes $(A_\Bbbk,L)$, $(A_i,L)$ $(A_\CC,L)$ are $d$-strict and the cohomologies $(H^*(A_\Bbbk),W)$, $(H^*(A_i),W)$, $(H^*(A_\CC),W,F)$, with their induced filtrations, are regular.
\end{defi}

\begin{lemm}
  \label{contrac-mhd}
  Let $A$ be a strict $\NN$-filtered mixed Hodge complex. Then, there are filtered contractions
\begin{equation}
  \label{contrac}
  \begin{tikzcd}
    (A_\Bbbk,W,L) \arrow[loop left, "h_\Bbbk"] \arrow[r,shift left,"p_\Bbbk"] \arrow[r,shift right, leftarrow, "s_\Bbbk"'] & (H^*(A_\Bbbk),W,L), \\  (A_i,W,L) \arrow[loop left, "h_i"] \arrow[r,shift left,"p_i"] \arrow[r,shift right, leftarrow, "s_i"'] & (H^*(A_i),W,L),
  \end{tikzcd}
\end{equation}
for all $1 \leq i \leq s-1$ and a bifiltered contraction
\[ \begin{tikzcd}
    (A_\CC,W,F,L) \arrow[loop left, "h_\CC"] \arrow[r,shift left,"p_\CC"] \arrow[r,shift right, leftarrow, "s_\CC"'] & (H^*(A_\CC),W,F,L).
  \end{tikzcd} \]
\end{lemm}
\begin{proof}
  By definition, $(A_\Bbbk,W,L)$ and $(A_i,W,L)$ are $\NN$-$d$-strict and $(A_\CC,W,F,L)$ is $\NN$-$d$-bistrict (recall Definition \ref{inddstrict-defi} and Definition \ref{inddbistrict-defi}). Thus, the result follows from Lemma \ref{fil-contrac-lemma} and Lemma \ref{bifil-contrac-lemma}. 
\end{proof}

Since the category of mixed Hodge structures has non-trivial extensions, given a complex of mixed Hodge structures $A \in \ch{\MHS}$, there is not, in general, a contraction where the morphisms preserve all the structure (see Remark $5.3$ of \cite{CiSo}). However, this is the case for split mixed Hodge structures (see Definition \ref{defi-split}).

\begin{lemm}
  \label{contrac-split-lemm}
  Let $A \in \ch{\MHS}$ be split as a graded mixed Hodge structure. Then, there exists a contraction
  \begin{equation*}
    \begin{tikzcd}
      A \arrow[loop left, "h"] \arrow[r,shift left,"p"] \arrow[r,shift right, leftarrow, "s"'] & H^*(A),
    \end{tikzcd}
  \end{equation*}
  where the morphisms are morphisms of mixed Hodge structures.
\end{lemm}
\begin{proof}
  Let $Z^n(A)$ denote the subspace of cycles of $A^n$ and $B^n(A)$ the subspace of boundaries. There are short exact sequences of mixed Hodge structures
  \[ 0 \to B^n(A) \to Z^n(A) \to H^n(A) \to 0 \]
  \[ 0 \to Z^n(A) \to A^n \xrightarrow{d} B^{n+1}(A) \to 0, \]
  for every degree $n$. Since $A$ is degree-wise split, these short exact sequences are split. The splittings yield an isomorphism of mixed Hodge structures
  \[ A^n \cong B^n(A) \oplus H^n(A) \oplus B^{n+1}(A). \]
  Endowing the graded mixed Hodge structure on the right-hand side with the differential $d(b,a,b') = (b',0,0)$ makes this into an isomorphism of complexes of mixed Hodge structures
  \[ A \cong B^*(A) \oplus H^*(A) \oplus B^{*+1}(A). \]
  Under this isomorphism, the maps $s(a) = (0,a,0)$, $p(b,a,b') = a$ and $h(b,a,b') = (0,0,b)$ give the desired contraction.
\end{proof}

\begin{theo}[Mixed Hodge Homotopy Transfer]
  \label{htt-mhd}
  Let $A = (A_\Bbbk,A_\CC,\varphi)$ be a strict $\NN$-filtered $\Pp$-mixed Hodge diagram. There exists an $\NN$-filtered $\Pp_\infty\textrm{-}\mathsf{MHD}$ structure $(H^*(A),m)$ whose underlying complex is $H^*(A)$ and ho-$\infty$-quasi-isomorphisms
  \[ A \xrightarrow{\sim} (H^*(A),m), \qquad (H^*(A),m) \xrightarrow{\sim} A. \]
  Furthermore, such a $\Pp_\infty$-mixed Hodge diagram is unique up to ho-$\infty$-isomorphism.
\end{theo}
\begin{proof}
  By Lemma \ref{contrac-mhd}, there exist filtered contractions
  \begin{equation*}
    \begin{tikzcd}
      (A_\Bbbk,W) \arrow[loop left, "h_\Bbbk"] \arrow[r,shift left,"p_\Bbbk"] \arrow[r,shift right, leftarrow, "s_\Bbbk"'] & (H^*(A_\Bbbk),W), \\ (A_i,W) \arrow[loop left, "h_i"] \arrow[r,shift left,"p_i"] \arrow[r,shift right, leftarrow, "s_i"'] & (H^*(A_i),W),
    \end{tikzcd}
  \end{equation*}
  for all $1 \leq i \leq s-1$ and a bifiltered contraction
  \[ \begin{tikzcd}
    (A_\CC,W,F) \arrow[loop left, "h_\CC"] \arrow[r,shift left,"p_\CC"] \arrow[r,shift right, leftarrow, "s_\CC"'] & (H^*(A_\CC),W,F),
  \end{tikzcd} \]
  Recall from Proposition \ref{filt-htt} that the homotopy transfer theorem yields filtered $\Pp_\infty$-algebra structures on $H^*(A_\Bbbk)$ and $H^*(A_i)$, together with filtered $\infty$-quasi-isomorphisms $P_\Bbbk$, $S_\Bbbk$, $P_i$, $S_i$ extending $p_\Bbbk$, $s_\Bbbk$, $p_i$ and $s_i$, respectively. Likewise, by Proposition \ref{bifil-htt}, there is a bifiltered $\Pp_\infty$-structure on $H^*(A_\CC)$ and bifiltered $\infty$-morphisms $P_\CC$ and $S_\CC$. The compositions $S_\Bbbk P_\Bbbk$, $S_i P_i$ and $S_\CC P_\CC$ are all gauge equivalent to the identity.

  For each $u : i \to j$, consider the composition 
  \[ \hat{\varphi}_u = P_j \varphi_u S_i. \] 
  Then, $\hat{\varphi}_1 = \varphi_u^*$, so $H^*(A)$, together with its level-wise $\Pp_\infty$-structures and $\hat{\varphi}$ as comparison morphisms, is an $\NN$-filtered $\Pp_\infty\textrm{-}\mathsf{MHD}$. Moreover, by Proposition \ref{comp-gauge-prop}, for each $u : i \to j$ there are gauges making the following squares commute up to gauge equivalence:
  \[ \begin{tikzcd}
      A_i \arrow[r,"P_i"] \arrow[d, "\varphi_u"] \arrow[dr, Rightarrow] & H^*(A_i) \arrow[d, "\hat{\varphi}_u"] \\
      A_j \arrow[r, "P_j"] & H^*(A_j),
    \end{tikzcd} \qquad
    \begin{tikzcd}
      H^*(A_i) \arrow[r,"S_i"] \arrow[d, "\hat{\varphi}_u"] \arrow[dr,Rightarrow] & A_i \arrow[d, "\varphi_u"] \\
      H^*(A_j) \arrow[r, "S_j"] & A_j.
    \end{tikzcd}\]
  This implies that there are ho-$\infty$-morphisms 
  \[ A \xrightarrow{P} (H^*(A_\Bbbk),H^*(A_\CC), \hat{\varphi}), \qquad (H^*(A_\Bbbk),H^*(A_\CC), \hat{\varphi}) \xrightarrow{S} A, \]
  whose level-wise components are $P_i$ and $S_i$. As these are filtered $\infty$-quasi-isomorphisms, it follows that $P$ and $S$ are ho-$\infty$-quasi-isomorphisms. Given another $\Pp_\infty$-structure $(H^*(A),m')$ and a ho-$\infty$-quasi-isomorphism
  \[ (H^*(A),m') \xrightarrow{S'} A, \]
  by Proposition \ref{comp-gauge-prop}, there exists a ho-$\infty$-morphism \[ (H^*(A),m) \to (H^*(A),m') \] whose level-wise components are $P_i \circ S_i'$. Such a morphism is a ho-$\infty$-isomorphism.
\end{proof}

\begin{defi}
  A \textit{minimal $\Pp_\infty$-model} of $A \in \Pp\textrm{-}\mathsf{MHD}^{\NN\textrm{-}\mathrm{fil}}$ is an $\NN$-filtered $\Pp_\infty$-mixed Hodge diagram whose components have trivial differential and such that there exists a ho-$\infty$-quasi-isomorphism to $A$.
\end{defi}
The following proposition, together with the homotopy transfer theorem, will be useful to give a characterization of formality of mixed Hodge diagrams.

\begin{prop}
  \label{zig-zag-equiv}
  Given $A$ and $A'$ strict $\NN$-filtered $\Pp$-mixed Hodge diagrams, the following are equivalent:
  \begin{enumerate}
    \item \label{zig-zag-equiv1} there exists a zig-zag of quasi-isomorphisms in $\Pp\textrm{-}\mathsf{MHD}^{\NN\textrm{-}\mathrm{fil}}$ between $A$ and $A'$,
    \item \label{zig-zag-equiv2} there exists a zig-zag of quasi-isomorphisms in $\Pp_\infty\textrm{-}\mathsf{MHD}^{\NN\textrm{-}\mathrm{fil}}$ between $A$ and $A'$,
    \item \label{zig-zag-equiv3} there exists a ho-$\infty$-quasi-isomorphism from $A$ to $A'$.
  \end{enumerate}
\end{prop}
\begin{proof}
  $(\ref{zig-zag-equiv1}) \Rightarrow (\ref{zig-zag-equiv2})$ is obvious. Let us assume $(\ref{zig-zag-equiv2})$. Given a zig-zag of quasi-isomorphisms of $\Pp_\infty$-mixed Hodge diagrams between $A$ and $A'$
  \[ A \xleftarrow{\sim} \cdots \xrightarrow{\sim} A', \]
  for each arrow in the zig-zag pointing left, say $f : B \to C$, we can produce a ho-quasi-isomorphism going the other way. First, choose (bi)filtered contractions like in Lemma \ref{contrac-mhd} for $B$ and $C$ and apply Theorem \ref{htt-mhd}. We can do this because $\NN$-filtered mixed Hodge diagrams quasi-isomorphic to strict ones are also strict. So every mixed Hodge diagram in the zig-zag is strict. Then, for each $u : i \to j$, consider the following diagram of ho-$\infty$-morphisms
  \[
    \begin{tikzcd}
      (H^*(B_i),m_i) \arrow[r,"S_i^B"] \arrow[d, "\hat{\varphi}^B_u"] \arrow[dr,Rightarrow] & B_i \arrow[d, "\varphi^B_u"] \arrow[r,"f_i"] &  C_i \arrow[d,"\varphi^C_u"] \arrow[r,"P_i^C"] \arrow[d, "\varphi^C_u"] \arrow[dr, Rightarrow] & (H^*(C_i),m_i) \arrow[d, "\hat{\varphi}^C_u"] \\
      (H^*(B_j),m_j) \arrow[r, "S^B_j"'] & B_j \arrow[r,"f_j"'] & C_j \arrow[r, "P_j^C"'] & (H^*(C_j),m_j)
    \end{tikzcd}
  \] 
  and let
  \[ g_i = P^C_i f_i S^B_i.\]
  By Proposition \ref{comphoinfmorph-prop}, the following square commutes up to gauge equivalence:
  \[ 
    \begin{tikzcd}
      (H^*(B_i),m_i) \arrow[r,"g_i"] \arrow[d,"\hat{\varphi}^B_u"] \arrow[dr,Rightarrow]& (H^*(C_i),m_i) \arrow[d,"\hat{\varphi}^C_u"]\\
      (H^*(B_j),m_j) \arrow[r, "g_j"] & (H^*(C_j),m_j).
    \end{tikzcd}
  \]
  Moreover, all morphism $g_i$ are a composition of filtered $\Pp_\infty$-quasi-isomorphisms and hence, are isomorphisms. One can thus invert this ho-$\infty$-isomorphism to get a ho-morphism
  \[ 
    \begin{tikzcd}
      (H^*(C_i),m_i) \arrow[r,"g_i^{-1}"] \arrow[d,"\hat{\varphi}^C_u"] \arrow[dr,Rightarrow]& (H^*(B_i),m_i) \arrow[d,"\hat{\varphi}^B_u"]\\
      (H^*(C_j),m_j) \arrow[r, "g_j^{-1}"] & (H^*(B_j),m_j).
    \end{tikzcd}
  \]
  Consider now, for each $u : i \to j$, the diagram
  \[ 
    \begin{tikzcd}
      C_i \arrow[r,"P_i^C"] \arrow[d,"\varphi^C_u"] \arrow[dr,Rightarrow] & (H^*(C_i),m_i) \arrow[r,"g_i"] \arrow[d,"\hat{\varphi}^C_u"] \arrow[dr,Rightarrow]& (H^*(B_i),m_i) \arrow[d,"\hat{\varphi}^B_u"] \arrow[r,"S_i^B"] \arrow[dr,Rightarrow] & B_i \arrow[d,"\varphi^B_u"]\\
      C_j \arrow[r,"P_j^C"'] & (H^*(C_j),m_j) \arrow[r, "g_j"'] & (H^*(B_j),m_j) \arrow[r,"S_j^B"'] & B_j.
    \end{tikzcd}
  \]
  The composition of these ho-$\infty$-morphisms is a ho-$\infty$-quasi-isomorphism from $C$ to $B$. Doing the same for all left pointing arrows of the zig-zag and composing all the resulting morphisms, one gets a ho-quasi-isomorphism from $A$ to $A'$. This proves $(\ref{zig-zag-equiv3})$.

  Now, assume $(\ref{zig-zag-equiv3})$. There exists a ho-$\infty$-quasi-isomorphism $f : A \to A'$. Apply the functor
  \[ \Omega_\kappa \mathrm{B}_\iota : \Pp_\infty\textrm{-}\mathsf{MHD}^{\NN\textrm{-}\mathrm{fil}} \to \Pp\textrm{-}\mathsf{MHD}^{\NN\textrm{-}\mathrm{fil}} \] 
  to $A$ and $A'$. For each $u : i \to j$, since $\varphi^{A'}_u f_i$ is gauge equivalent to $f_j \varphi^A_u$,  it follows by Lemma \ref{barcobar-hom}, that 
  \[ \Omega_\kappa \mathrm{B}_\iota (\varphi^{A'}_u)  \Omega_\kappa \mathrm{B}_\iota (f_i) \quad \text{is homotopic to} \quad \Omega_\kappa \mathrm{B}_\iota (f_j)  \Omega_\kappa \mathrm{B}_\iota (\varphi^A_u).\] 
  Hence, for each $u : i \to j$, there is a ho-quasi-isomorphism of $\Pp$-algebras
  \[ 
  \begin{tikzcd}
    \Omega_\kappa \mathrm{B}_\iota A_i \arrow[r,"\Omega_\kappa \mathrm{B}_\iota f_i"] \arrow[dr,Rightarrow] \arrow[d,"\Omega_\kappa \mathrm{B}_\iota \varphi^A_u"'] & \Omega_\kappa \mathrm{B}_\iota A'_i \arrow[d,"\Omega_\kappa \mathrm{B}_\iota \varphi^{A'}_u"]\\
    \Omega_\kappa \mathrm{B}_\iota A_j \arrow[r,"\Omega_\kappa \mathrm{B}_\iota f_j"'] & \Omega_\kappa \mathrm{B}_\iota A'_j.
  \end{tikzcd}
  \]
  By Proposition \ref{homorph-fact}, these ho-quasi-isomorphisms factor as zig-zags of strict squares. One can then compose with the counit on the left and on the right
  \[ 
    \begin{tikzcd}
      \Omega_\kappa \mathrm{B}_\iota A_i \arrow[r,"\sim"] \arrow[d,"\Omega_\kappa \mathrm{B}_\iota \varphi^A_u"]& A_i \arrow[d,"\varphi^A_u"] \\
      \Omega_\kappa \mathrm{B}_\iota A_j \arrow[r,"\sim"] & A_j,
    \end{tikzcd} \qquad
    \begin{tikzcd}
      \Omega_\kappa \mathrm{B}_\iota A'_i\arrow[r,"\sim"] \arrow[d,"\Omega_\kappa \mathrm{B}_\iota \varphi^{A'}_u"]& A'_i \arrow[d,"\varphi^{A'}_u"] \\
      \Omega_\kappa \mathrm{B}_\iota A'_j \arrow[r,"\sim"] & A'_j,
    \end{tikzcd}
  \]
  to get a zig-zag of quasi-isomorphisms in $\Pp\textrm{-}\mathsf{MHD}^{\NN\textrm{-}\mathrm{fil}}$ between $A$ and $A'$, showing $(\ref{zig-zag-equiv1})$.
\end{proof}

A consequence of this proposition and the homotopy transfer theorem is that any $\NN$-filtered $\Pp$-mixed Hodge diagram is quasi-isomorphic to an $\NN$-filtered $\Pp$-mixed Hodge algebra (see Definition \ref{Pmhalg-defi}):
\begin{prop}
  \label{Pmhalgmodels-prop}
  Given $A \in \Pp\textrm{-}\MHD^{\NN\textrm{-}\mathrm{fil}}$, there exists an $\NN$-filtered $\Pp$-mixed Hodge algebra which is quasi-isomorphic to $A$.
\end{prop}
\begin{proof}
  Theorem \ref{htt-mhd} implies that for any $A \in \Pp\textrm{-}\mathsf{MHD}^{\NN\textrm{-}\mathrm{fil}}$, there exists a minimal $\Pp_\infty$-model of $A$, let us denote it by $H \in \Pp_\infty\textrm{-}\mathsf{MHD}^{\NN\textrm{-}\mathrm{fil}}$. Since $H$ has trivial differential, its comparison morphisms are $\infty$-isomorphisms. This implies that the comparison morphisms of $\Omega_\kappa \mathrm{B}_\iota(H)$ are isomorphisms and thus that $\Omega_\kappa \mathrm{B}_\iota(H)$ is isomorphic to an $\NN$-filtered $\Pp$-mixed Hodge algebra. Since $A$ is ho-$\infty$-quasi-isomorphic to $H$ and $H$ is $\infty$-quasi-isomorphic to $\Omega_\kappa \mathrm{B}_\iota(H)$, it follows, by Proposition \ref{zig-zag-equiv}, that $A$ is quasi-isomorphic to an $\NN$-filtered $\Pp$-mixed Hodge algebra.
\end{proof}
\begin{rema}
  \label{cdgamodelMHD-rema}
  Note that in the previous proof, the model of $A$ in $\Pp$-mixed Hodge algebras is the bar-cobar resolution of $(H^*(A_\Bbbk),m_\Bbbk)$ with Hodge filtration given by the free extension of the Hodge filtration of $H^*(A_\CC)$ but twisted by an isomorphism of $\Pp$-algebras. Indeed, the model is of the form $\Omega_\kappa \mathrm{B}_\iota(H)$, where $H$ is a minimal $\Pp_\infty$-model of $A$. Its underlying $\Bbbk$-vector space is thus $\Omega_\kappa \mathrm{B}_\iota(H^*(A_\Bbbk))$ with weight filtration given by the one of $H^*(A)$ extended to the free $\Pp^\antishriek$-coalgebra and then to the free $\Pp$-algebra. On the other hand, its Hodge filtration is of the form $\Omega_\kappa \mathrm{B}_\iota(\varphi)^{^{-1}}(F)$, where $\varphi$ is the composition of the comparison morphisms of $H$ and $F$ is the free extension of the Hodge filtration of $H^*(A_\CC)$. 
\end{rema}
Given a length $s$, let $i : \Pp\textrm{-}\MHD_s \to \Pp\textrm{-}\MHD_{s+1}$ be the inclusion given by inserting an identity as one of the comparison morphisms. As referenced in Remark \ref{incmhdlength-equiv-rema}, we have a further consequence of Proposition \ref{zig-zag-equiv} and the homotopy transfer theorem:
\begin{prop}
  \label{incmhdlength-equiv-prop}
  The functor $i$ induces an essentially surjective functor on homotopy categories and a bijection between isomorphism classes of
  \[ \Ho(\Pp\textrm{-}\MHD_s^{\NN\textrm{-}\mathrm{fil}}) \quad \text{and} \quad \Ho(\Pp\textrm{-}\MHD_{s+1}^{\NN\textrm{-}\mathrm{fil}}). \]
\end{prop}
\begin{proof}
  The fact that $i$ induces an essentially surjective functor is a direct consequence of Proposition \ref{Pmhalgmodels-prop}. To see that $i$ induces the stated bijection, let $A$ and $A'$ have length $s$. Then, by Proposition \ref{zig-zag-equiv}, $A$ is quasi-isomorphic to $A'$ if and only if there exists a ho-$\infty$-quasi-isomorphism between them. Likewise, $i(A)$ is quasi-isomorphic to $i(A')$ if and only if there exists a ho-$\infty$-quasi-isomorphism between them. But, given a ho-$\infty$-quasi-isomorphism from $i(A)$ to $i(A')$, there exists such a morphism between $A$ and $A'$ obtained by contracting the homotopy commutative square whose comparison morphisms are identities.
\end{proof}

\section{Mixed Hodge Formality}
\label{sec-main}
In this section we start by defining formality of mixed Hodge diagrams and relate this notion to the splitting of mixed Hodge structures on homotopy groups. We then relate formality of a diagram and formality of its Koszul dual. We state and prove the main theorem of the paper, Theorem \ref{obst-theo}, giving an obstruction theory to formality. We then relate the first obstruction to the mixed Hodge structure on $\pi_3 \otimes \Bbbk$ (Proposition \ref{first-obs}). Afterwards, we study the case of $\alpha$-pure cohomology (section \ref{pure-sec}). We finish by giving a formula for the second obstruction to formality in the case where $\Pp = Ass$ and under certain conditions (Proposition \ref{rep-phi3-prop}).

\subsection{Formal mixed Hodge diagrams}
We define the notion of formality of mixed Hodge diagrams and prove that if the cohomology is split as a mixed Hodge structure, then the homotopy groups are also split. Fix $\Pp$ a Koszul operad and $\Cc$ a cooperad.

In the sequel, we will only encounter strict $\NN$-filtered $\Pp$-mixed Hodge diagrams and general $\NN$-filtered $\Cc$-mixed Hodge diagrams, as the bar $\mathrm{B}_\kappa(A)$ of a strict $\NN$-filtered mixed Hodge diagram $A$ may not, in general, be strict. To simplify notation, we henceforth refer to strict $\NN$-filtered $\Pp$-mixed Hodge diagrams by $\Pp$-mixed Hodge diagrams and $\NN$-filtered $\Cc$-mixed Hodge diagrams by $\Cc$-mixed Hodge diagrams.
\begin{defi}
  A $\Pp$-mixed Hodge diagram $A$ is said to be \textit{formal} if there is a zig-zag of quasi-isomorphisms of $\Pp$-mixed Hodge diagrams
  \[ A \xleftarrow{\sim} \cdots \xrightarrow{\sim} H^*A.\]
\end{defi}

\begin{defi}
  A $\Cc$-mixed Hodge diagram $C$ is said to be \textit{formal} if there is a zig-zag of weak-equivalences of $\Cc$-mixed Hodge diagrams
  \[ C \xleftarrow{\sim} \cdots \xrightarrow{\sim} H^*C.\]
\end{defi}

\begin{rema}
Note that if a mixed Hodge diagram $A$ is formal, then $(A_\Bbbk,W)$ is formal as a filtered algebra and $(A_\CC,W,F)$ is formal as a bifiltered algebra. The converse implication is not true. Indeed, if $H^*(A)$ is pure, then both $(A_\Bbbk,W)$ and $(A_\CC,W,F)$ are formal but $A$ may not be formal as a mixed Hodge diagram (see, for instance, Example \ref{exam-fieldext}).
\end{rema}
Let $(A,\mu)$ be a $\Pp$-mixed Hodge diagram, where $\mu$ denotes the $\Pp$-algebra structure. By Theorem \ref{htt-mhd}, there exists a minimal $\Pp_\infty$-model of $A$, which we denote by $(H^*(A),m)$. 
\begin{prop}
  \label{Pinfty-model}
  The mixed Hodge diagram $A$ is formal if and only if there exists a ho-$\infty$-isotopy
  \[ (H^*(A),m) \xrightarrow{\cong} (H^*(A),\mu^*). \]
\end{prop}
\begin{proof}
  By Theorem \ref{htt-mhd}, there is a ho-$\infty$-quasi-isomorphism \[ A \to (H^*(A),m). \] The result then follows by  Proposition \ref{zig-zag-equiv}.
\end{proof}

Let us now define the homotopy groups of a mixed Hodge diagram.

\begin{defi}
  The \textit{cohomotopy $\Pp^\antishriek$-coalgebra} $\pi^*(A)$ of $A \in \Pp\textrm{-}\mathsf{MHD}$ is defined as the cohomology of $\mathrm{B}_\kappa A$,
  \[ \pi^*(A) := H^*(\mathrm{B}_\kappa A). \]
\end{defi}

\begin{rema}
  \label{mhs-pi=morg}
  Let $\Pp = Com$ and $A$ be a $1$-connected commutative mixed Hodge diagram, that is $H^0(A_\Bbbk) = \Bbbk$ and $H^{<0}(A_\Bbbk) = 0 = H^{1}(A_\Bbbk)$. The mixed Hodge structures on $\pi^*(A)$ are isomorphic to the mixed Hodge structures on the homotopy groups of $A$ given in \cite{Cir-cofmodel} and \cite{morgan}. Indeed, following \cite{Cir-cofmodel}, the mixed Hodge structures on the homotopy groups of $A$ are given by the cohomology of $\mathrm{Ind}(M)$, the indecomposables of a cofibrant model $M$ of $A$. And note that the mixed Hodge algebra $\Omega_\kappa \mathrm{B}_\kappa A$ is a cofibrant model of $A$ and $\mathrm{Ind}(\Omega_\kappa \mathrm{B}_\kappa A) = \mathrm{B}_\kappa A$.   
\end{rema}

The following result shows that, for formal mixed Hodge diagrams, split mixed Hodge structures on cohomology (recall Definition \ref{defi-split})
give split mixed Hodge structures on homotopy.

\begin{prop}
  \label{prop-split}
  If $A \in \Pp\textrm{-}\mathsf{MHD}$ is formal and $H^*(A)$ is split as a graded mixed Hodge structure, then $\pi^*(A)$ is also split as a graded mixed Hodge structure.
\end{prop}
\begin{proof}
  If $A$ is formal, then $\mathrm{B}_\kappa A \simeq \mathrm{B}_\kappa H^*(A)$. Since $H^*(A)$ is split, it follows that $\mathrm{B}_\kappa H^*(A)$ is split, so its cohomology is also split.
\end{proof}

\begin{rema}
  The converse of Proposition \ref{prop-split} is not true. Indeed, in Example \ref{exam-obs2}, we give a mixed Hodge algebra which has pure Hodge structures on both cohomology and homotopy groups but which is not formal as a mixed Hodge diagram.
\end{rema}

\subsection{Formality and Koszul duality}
We next introduce the notions of Koszul dual (co)algebra and adapt to mixed Hodge diagrams the results of Berglund \cite{BergKoszul} about the relation of formality and coformality with Koszul duality.

A \textit{weight grading} on a $\Pp$-algebra $A$ is a decomposition
\[ A = A(1) \oplus A(2) \oplus \dots \]
such that the structure map $\mu : \Pp(A) \to A$ preserves the grading.
\begin{rema}
   This weight grading is not to be confused with the weight filtration of a mixed Hodge structure. We follow the terminology of \cite{BergKoszul} and only need to mention this weight grading in this section.
\end{rema}
The bar $\mathrm{B}_\kappa(A)$ is then bigraded by the induced weight grading $l_\omega$ and by arity $l_a$. Since $A$ is concentrated in positive weight grading, the bar construction is concentrated in the region $\{ l_\omega \geq l_a \}$. Let $\Dd = \{ l_\omega = l_a \}$ denote the diagonal. The coderivation $d_2$ of $\mathrm{B}_\kappa(A)$ induced by the product of $A$ preserves $l_\omega$. Let
\[ A^{\antishriek} := \Dd \cap \ker(d_2) \subset \mathrm{B}_\kappa(A). \]
Note that $A^\antishriek$ is a $\Pp^\antishriek$-coalgebra and the inclusion 
\[ A^\antishriek \hookrightarrow \mathrm{B}_\kappa(A) \]
is a map of coalgebras. A weight grading on $A$ is a \textit{Koszul grading} if this inclusion is a weak-equivalence of $\Pp^\antishriek$-coalgebras. A \textit{Koszul $\Pp$-algebra} is a $\Pp$-algebra that admits a Koszul grading. 

Dually, a weight grading on a $\Pp^\antishriek$-coalgebra $C$ is a decomposition
\[ C = C(1) \oplus C(2) \oplus \dots \]
such that the decomposition map $\Delta : C \to \Pp^\antishriek(C)$ preserves the grading. The cobar $\Omega_\kappa(C)$ is then bigraded by the weight grading $l_\omega$ and by arity $l_a$. Let $\Dd_c = \{ l_\omega = l_a \}$ denote the diagonal. The derivation $d_2$ of $\Omega_\kappa(C)$ induced by the decomposition of $C$ preserves $l_\omega$. Let
\[ C^\antishriek := \Dd_c / (\Dd_c \cap \Img(d_2)) \]
Then $C^\antishriek$ is a $\Pp$-algebra and the projection 
\[ \Omega_\kappa(C) \to C^\antishriek \]
is a map of algebras. A weight grading on $C$ is a \textit{Koszul grading} if this projection is a quasi-isomorphism of $\Pp$-algebras. A \textit{Koszul $\Pp^\antishriek$-coalgebra} is a $\Pp^\antishriek$-coalgebra that admits a Koszul grading. 

If $A$ is a mixed Hodge $\Pp$-algebra with a weight grading and the pieces $A(i)$ are $\Bbbk$-subspaces, then they are automatically sub-mixed Hodge structures. Thus, $A^\antishriek$ is a mixed Hodge $\Pp^\antishriek$-coalgebra and the inclusion $A^\antishriek \hookrightarrow \mathrm{B}_\kappa(A)$ is a map of mixed Hodge $\Pp^\antishriek$-coalgebras. Dually, If $C$ is a mixed Hodge $\Pp^\antishriek$-coalgebra and the pieces $C(i)$ are $\Bbbk$-subspaces, then they are automatically sub-mixed Hodge structures. Thus, $C^\antishriek$ is a mixed Hodge $\Pp$-algebra and the projection $\Omega_\kappa(C) \to C^\antishriek$ is a map of mixed Hodge $\Pp$-algebras. This observation leads to a direct adaptation of Theorem $18$ of \cite{BergKoszul} to the setting of mixed Hodge diagrams.

\begin{prop}
  \label{form-koszuldual-prop}
  Let $A \in \Pp\textrm{-}\mathsf{MHD}$. The following are equivalent
  \begin{enumerate}
    \item $A$ and $\mathrm{B}_\kappa(A)$ are formal,
    \item $A$ is formal and $H^*(A)$ is a Koszul $\Pp$-algebra,
    \item $\mathrm{B}_\kappa(A)$ is formal and $\pi^*(A)$ is a Koszul $\Pp^\antishriek$-coalgebra.
  \end{enumerate}
\end{prop}
\begin{proof}
  The implications $(1) \Rightarrow (2), (3)$ are a consequence of Theorem $18$ of \cite{BergKoszul}. We prove $(2) \Rightarrow (1)$. Since $A$ is formal, $A \simeq H^*(A)$, so $\mathrm{B}_\kappa(A) \simeq \mathrm{B}_\kappa(H^*(A))$. Since $H^*(A)$ is Koszul, $\mathrm{B}_\kappa(H^*(A)) \simeq H^*(\mathrm{B}_\kappa(H^*(A)))$. Hence, $\mathrm{B}_\kappa(A) \simeq \pi^*(A) \cong H^*(\mathrm{B}_\kappa(H^*(A)))$. The proof of $(3) \Rightarrow (1)$ is similar.
\end{proof}

\subsection{Operadic Cohomology}
In this section we define different notions of operadic cohomology suited for $\Pp$-mixed Hodge diagrams. The obstructions to formality of Theorem \ref{obst-theo} will live in one of these cohomologies.

Let $H$ be a graded $\Pp$-algebra with structure map $\gamma : \Pp(H) \to H$. The \textit{operadic cohomology} of $H$ is the cohomology of the complex
\[ \Hom^{\tau_\kappa}(\Pp^\antishriek(H),H) = (\Hom(\Pp^\antishriek(H),H),\delta = d + d_{\tau_\kappa}), \]
where $d(f) = f \circ d_{\mathrm{B}_\kappa A}$ is given by precomposing with the differential of $\mathrm{B}_\kappa(A)$ and 
\[ d_{\tau_\kappa} = \Pp^{\antishriek}(H) \xrightarrow{\Delta} \Pp^{\antishriek}(\Pp^{\antishriek}(H))\xrightarrow{\kappa \circ (\tau_\kappa;f)} \Pp(H) \xrightarrow{\gamma} H.\]
Here $\tau_\kappa$ is the universal twisting morphism $\tau_\kappa : \mathrm{B}_\kappa H \to H$ and the notation $\kappa \circ (\tau_\kappa;f)$ means that for each term $(\Pp^\antishriek(H))^{\otimes n}$, we apply $f$ in one of the tensors and apply $\tau_\kappa$ in all the others (see section 12.4 of \cite{LV} for more details). For the operads $Ass$, $Com$ and $Lie$, this is the Hochschild, Harrison and Chevaley-Heilenberg cohomology, respectively. We now define the corresponding notions of operadic cohomology in the setting of mixed Hodge algebras.

\begin{defi}
  \label{oper-cohom}
  Let $H$ be a graded $\Pp$-algebra in mixed Hodge structures and consider the free $\Pp$-algebra $\Pp(H)$ with the induced mixed Hodge structures. We define
  \begin{enumerate}
    \item The \textit{mixed Hodge $\Pp$-cohomology} of $H$, denoted by $\Pp H_{\mathrm{MHS}}^*(H)$, is the cohomology of the complex
    \begin{equation*}
      \Pp C^*_{\mathrm{MHS}} := \Hom_{\mathrm{MHS}}^{\tau_\kappa}(\Pp^\antishriek(H),H)
    \end{equation*}
    \item The \textit{Ext-$\Pp$-cohomology} of $H$, denoted by $\Pp H_{\mathrm{Ext}}^*(H)$, is the cohomology of the complex
          \begin{equation*}
            \Pp C^*_{\mathrm{Ext}} :=\mathrm{Ext}_{\mathrm{MHS}}^{1 \mathrm{\quad}\tau_\kappa}(\Pp^\antishriek(H),H)
          \end{equation*}
    \item \label{abs-p-cohom-defi} The \textit{Deligne-Beilinson $\Pp$-cohomology} of $H$, denoted by $\Pp H_{\mathrm{DB}}^*(H)$, is the cohomology of the cone
    \begin{align*}
      \Pp C^*_{\mathrm{DB}} := \mathrm{Cone}( W_0\Hom_\Bbbk^{\tau_\kappa}(\Pp^\antishriek(H),H) \oplus &F^0W_0 \Hom_\CC^{\tau_\kappa}(\Pp^\antishriek(H) \otimes \CC,H \otimes \CC) \\ &\xrightarrow{\iota_\Bbbk - \iota_F} W_0 \Hom_\CC^{\tau_\kappa}(\Pp^\antishriek(H) \otimes \CC,H \otimes \CC))
    \end{align*}
  \end{enumerate}
  In each case, we twist the Hom (or Ext) complex by the twisting morphism $\tau_\kappa: B_\kappa(H) \to H$. In the definition of Absolute Hodge $\Pp$-cohomlogy, the maps $\iota_\Bbbk$ and $\iota_F$ denote the inclusions of $W_0\Hom_\Bbbk^{\tau_\kappa}(\Pp(H),H)$ and $F^0W_0 \Hom_\CC^{\tau_\kappa}(\Pp(H) \otimes \CC,H \otimes \CC)$, respectively, into $W_0 \Hom_\CC^{\tau_\kappa}(\Pp(H) \otimes \CC,H \otimes \CC)$.
\end{defi}

\begin{rema}
  The definition of the cone that we employ is the following. Given a map of cochain complexes $f : A \to B$, we have
  \[ \mathrm{Cone}(f)^n = A^n \oplus B^{n-1}, \qquad d(a,b) = (d_A a, f(a) - db). \] 
  We shall often denote by $\delta$ the differential of any of the complexes in Definition \ref{oper-cohom} and also the differential of the classical operadic cohomology complex. It should be clear to which complex we are referring to from context.
\end{rema}

\begin{rema}
  The definition of the DB $\Pp$-cohomology groups $\Pp H_{\mathrm{DB}}^*(H)$ is motivated by the notion of absolute Hodge cohomology of a mixed Hodge complex, introduced by Beilinson in \cite{Beilinson}. In fact, the inner Hom complex $\underline{\Hom}(\Pp^\antishriek(H),H)$ is a complex of mixed Hodge structures and $\Pp H^*_{\mathrm{DB}}(H)$ is its absolute Hodge cohomology.
\end{rema}

Recall the vector space $\Pp(H)$ comes equipped with an extra grading
  \begin{equation*}
    (\Pp(H))_n = \Pp^\antishriek(n) \otimes_{\SS_n} H^{\otimes n},
  \end{equation*}
  which induces a grading on each of the defined cohomologies. Let us call this extra grading by arity. Denote by $\Pp H_{\mathrm{MHS}}^{k,m}(H)$ the classes whose representatives have arity $k$ and degree $k+m$. Do analogously for $\Pp H_{\mathrm{Ext}}^*(H)$ and denote by $\Pp H_{\mathrm{DB}}^{k,m}(H)$ the classes $[f_\Bbbk,f_\CC,h]$ such that $f_\Bbbk$ (and also $f_\CC$) have arity $k$ and degree $k+m$ and $h$ has arity $k-1$ and degree $k+m$.

\subsection{Obstructions to mixed Hodge formality}
\label{obsform-sec}
In the rest of this work, we assume that $\Pp = Ass, Com$ or $Lie$. In this section we construct a sequence of obstructions to formality of $\Pp$-mixed Hodge diagrams for these operads. 

Let $L$ be a complete $SL_\infty$-algebra. In section $1.1$ of \cite{higherLie}, the authors give a formula for gauge equivalences in terms of planar rooted trees. The first terms are given by the following:

\begin{lemm}
  \label{gauge-eq}
  Given $x \in L^0$ and $\lambda \in L^{-1}$, then we have that
  \begin{align*}
        \lambda \cdot x = x &+d\lambda + l_2(x,\lambda)  + \frac{1}{2}l_2(d\lambda,\lambda) \\ \nonumber &+ (\text{compositions of brackets } l_k \text{ with at least one } \lambda \text{ as input}).
    \end{align*}
\end{lemm}
Recall that for $\Pp = Ass, Com$ or $Lie$, a $\Pp_\infty$-algebra structure on a complex $H$ is given by a collection of maps
\[ m = (m_1,m_2,\dots),\]
whose components $m_k$ are maps of the form $m_k : H^{\otimes k} \to H$, where $m_k$ has degree $2-k$, for $k \geq 2$. Likewise, an $\infty$-morphism $f : A \to B$ is given by a collection of maps
\[ f = (f_1,f_2,\dots),\]
whose components are maps $f_k$ of the form $f_k : A^{\otimes k} \to B$, where $f_k$ has degree $1-k$. Recall also that $\infty$-morphisms are the Maurer-Cartan elements of an $SL_\infty$-algebra with underlying complex 
\[ \Hom(\mathrm{B}_\iota(A),B),\]
(see Proposition \ref{Pinfmorph=MC-prop}). We will use the following technical lemmas in the proof of Theorem \ref{obst-theo}.

\begin{lemm}
    \label{obs-lemm}
    Let $(H,W,m)$ and $(H,W,m')$ be filtered $\Pp_\infty$-algebras whose underlying complex $H$ has trivial differential. Given $n \geq 2$, let $f : (H,W,m) \to (H,W,m')$ be a filtered $\infty$-isotopy such that
    \begin{align*}
        &f_i = 0 \qquad \text{for} \quad 2 \leq i \leq n-1,\\
        &f_n = \delta(h),
    \end{align*}
    for some $h \in \Pp C^{n-1,1-n}(H)$. Then, $f$ is gauge equivalent to a filtered $\infty$-isotopy \[ f' : (H,W,m) \to (H,W,m'), \] such that 
    \[ f_i' = 0 \qquad \text{for} \quad 2 \leq i \leq n.\]
\end{lemm}
\begin{proof}
    Define $\lambda \in W_0 \Hom(\mathrm{B}_\iota(H,m),(H,m'))$ as 
    \[ \lambda = (0,\dots,0,h,0,\dots),\]
    where $-h$ is placed in arity $n-1$. Then, $\lambda$ is an element of degree $-1$. We prove that
    \[ f' = \lambda \cdot f\]
    has the desired properties. By Lemma \ref{gauge-eq} and since the differential of $H$ is trivial, the gauge action is given by
    \begin{align}
        \label{gaugeact-eq}
        \lambda \cdot f = f &+ \lambda \circ d_{\mathrm{B}_\iota(H,m)} + l_2(f,\lambda) \\ \nonumber &+ (\text{compositions of brackets } l_k \text{ with at least one } \lambda \text{ as input}).
    \end{align}
    Recall the formulas for the brackets of $ W_0 \Hom(\mathrm{B}_\iota(H,m),(H,m'))$:
    \[ l_n(f_1,...,f_n) = \sum_{\sigma \in \SS_n} (-1)^{\mathrm{sgn}(\sigma,f_1,...,f_n)} \gamma_{A'} \circ (\iota \otimes f_{\sigma(1)} \otimes \cdots \otimes f_{\sigma(n)}) \circ \Delta_n^A,\]
    (see (\ref{Linftymaps}). Since $\lambda_i = 0$ for $1 \leq i \leq n-2$, the non-written terms of the sum (\ref{gaugeact-eq}) all have components of arity $\geq n+1$. Moreover, since $f_1 = \mathrm{id}$ and $f_i = 0$ for $2 \leq i \leq n-1$, it follows that 
    \[f_i' = (\lambda \cdot f)_i = 0 \quad \text{for} \quad 2 \leq i \leq n-1. \]
    and 
    \[ f'_n = f_n + (\lambda \circ d_{\mathrm{B}_\iota(H,m)} + l_2(\mathrm{id}, \lambda))_n =  f_n + (h \star m_2 + m_2 \star h) = f_n - \delta(h) = 0.\qedhere\]
\end{proof}
Let $f : A \to A'$ and $g : A' \to A''$ be ho-$\infty$-isotopies of $\Pp_\infty$-mixed Hodge diagrams of length $1$. Thus $f = (f_\Bbbk,f_\CC,F)$ is given by $\infty$-morphisms
\[ f_\Bbbk : A_\Bbbk \to A'_\Bbbk, \quad f_\CC : A_\CC \to A_\CC'\]
such that $(f_\Bbbk)_1 = \mathrm{id}$ and $(f_\CC)_1  = \mathrm{id}$ and a gauge $F$
\[ F \cdot (\varphi^{A'}\circledcirc f_\Bbbk \otimes \CC) = f_\CC \circledcirc \varphi^A. \]
Note that the gauge $F$ is given by a collection of maps
\[ F = (F_1,F_2,\dots)\]
with $F_k : (A_\Bbbk \otimes \CC)^{\otimes k} \to A_\CC'$ of degree $-k$. The same follows for $g = (g_\Bbbk,g_\CC,G)$. Recall that, by Proposition \ref{comphoinfmorph-prop}, the composition $g \circ f$ is well-defined and has components $g_\Bbbk \circledcirc f_\Bbbk$ and $g_\CC \circledcirc f_\CC$. Denote its gauge by $K$.
\begin{lemm}
    \label{comphoinfmorfarity-lemm}
    Given $n \geq 2$, suppose that 
    \[  (g_\Bbbk)_i  = 0, \quad (g_\CC)_i = 0, \quad G_{i-1} = 0, \]
    for $2 \leq i \leq n$. Then, we have that
    \[ (g_\Bbbk \circledcirc f_\Bbbk)_i = (f_\Bbbk)_i, \quad (g_\CC \circledcirc f_\CC)_i = (f_\CC)_i, \quad K_{i-1} = F_{i-1},\]
    for $2 \leq i \leq n$.
\end{lemm}
\begin{proof}
    The fact that
    \[ (g_\Bbbk \circledcirc f_\Bbbk)_i = (f_\Bbbk)_i \quad \text{and} \quad (g_\CC \circledcirc f_\CC)_i = (f_\CC)_i\] 
    for $2 \leq i \leq n$ is a direct consequence of the definition of the composition of $\infty$-morphisms. We show that $K_{i-1} = F_{i-1}$ in the same range. Recall, from the proof of Proposition \ref{comphoinfmorph-prop}, that $K$ is obtained in two steps. Firstly, we consider the gauge $x_{01} := G \circledcirc (f_\Bbbk \otimes \CC)$ between 
    \[ x_0 := \varphi^{A''} \circledcirc (g_\Bbbk \otimes \CC) \circledcirc (f_\Bbbk \otimes \CC) \quad \text{and} \quad x_1 :=  g_\CC \circledcirc \varphi^{A'} \circledcirc (f_\Bbbk \otimes \CC) \]
    and the gauge $x_{12} := g_\CC \circledcirc F$ between 
    \[ x_1 = g_\CC \circledcirc \varphi^{A'} \circledcirc (f_\Bbbk \otimes \CC) \quad \text{and} \quad x_2 := g_\CC \circledcirc f_\CC \circledcirc \varphi^{A}. \]
    These gauges form a horn $\Lambda^2_{1}$ 
    \begin{equation*}
      \begin{tikzcd}[row sep = small]
            & x_1 \arrow[ddr,"x_{12}"]  & \\
            &  & \\
            x_0 \arrow[uur,"x_{01}"] & & x_2
      \end{tikzcd}
    \end{equation*} 
    in the Kan complex $\gamma_{\bullet}(\Hom(\mathrm{B}_\iota(A_\Bbbk \otimes \CC), A''))$. Then $K$ is given by the arrow filling this horn corresponding to the choice of $0$ (see section $5$ of \cite{higherLie}). There are explicit formulas for $K$ (see section $5$ of \cite{higherLie}), whose lower terms are given by
    \begin{align*}
        K = x_{01} + x_{12} + (\text{compositions of} &\text{ brackets } l_k \text{ with at least} \\ &\text{one bracket having inputs } x_{01} \text{ and } x_{12}). 
    \end{align*}  
    Note that 
    \[ (G \circledcirc (f_\Bbbk \otimes \CC))_i = 0 \quad \text{and} \quad (g_\CC \circledcirc F)_i = F_i, \]
    for $1 \leq i \leq n-1$. Therefore, 
    \[ K_i = F_i, \quad \text{for } 1 \leq i \leq n-1. \]
\end{proof}
In the following theorem, we construct "successively defined obstructions". This means that each obstruction $\theta_k$ is only defined if all previous obstructions are defined and trivial.

\begin{theo}
  \label{obst-theo}
  Given $A \in \Pp\textrm{-}\mathsf{MHD}$, there exist successively defined classes 
  \begin{equation*}
    \theta_k \in \Pp H_{\mathrm{DB}}^{k,2-k}(H^*(A)), \mathrm{\qquad} k \geq 3
  \end{equation*}
  such that if all classes are trivial, then $A$ is formal.
\end{theo}
\begin{proof}
  Recall that any mixed Hodge diagram $A$ of length $s$ is quasi-isomorphic to a mixed Hodge diagram $A'$ of length $1$ and that formality of $A \in \Pp\textrm{-}\mathsf{MHD}_s$ is equivalent to formality of $A' \in \Pp\textrm{-}\mathsf{MHD}_1$ (see Proposition \ref{incmhdlength-equiv-prop}). Therefore, we may assume that $A$ has length $1$. By Theorem \ref{htt-mhd} and Proposition \ref{Pinfty-model}, there exists a minimal $\Pp_\infty$-model $(H^*(A),m)$ of $A$ and $A$ is formal if and only if $(H^*(A),m)$ is ho-$\infty$-isotopic to $H^*(A)$ seen as a mixed Hodge $\Pp$-algebra with the induced product on cohomology. Denote by 
  \[ m_2 : \Pp(2) \otimes_{\SS_2} H^*(A)^{\otimes 2} \to H^*(A) \]
  the induced product. To simplify notation, let also $H = H^*(A)$. Here, $m$ denotes a tuple $ m = (m_\Bbbk,m_\CC,\varphi)$, where $m_\Bbbk$ is a filtered $\Pp_\infty$-structure on $H$, $m_\CC$ is a bifiltered $\Pp_\infty$-structure on $H \otimes \CC$ and $\varphi$ an $\infty$-morphism between them. Suppose that
  \begin{align*}
      &m_\Bbbk = (0,m_2,0,\dots,0,(m_\Bbbk)_n,\dots), \\
      &m_\CC = (0,m_2,0,\dots,0,(m_\CC)_n,\dots), \\
      &\varphi = (\mathrm{id},0,\dots,\varphi_{n-1},\dots).
  \end{align*}
  for some $n \geq 3$, the case $n = 3$ being trivial. Then, by the formulas for $\Pp_\infty$-algebras, we have that
  \[ \delta((m_\Bbbk)_n) = 0, \quad  \delta((m_\CC)_n) = 0.\]
  Moreover, by Proposition $3.3$ (a) of \cite{Bashar}, we have that
  \[ (m_\Bbbk)_n - (m_\CC)_{n} = \delta(\varphi_{n-1})\]
  This implies that the triple 
  \[ ((m_\Bbbk)_n, (m_\CC)_n, \varphi_{n-1})\]
  defines a cycle in $\Pp C_{\mathrm{DB}}^{n,2-n}(H^*(A))$. The associated class is the $n$-th obstruction to formality
  \[ \theta_n = [(m_\Bbbk)_n, (m_\CC)_n, \varphi_{n-1}]  \in \Pp H_{\mathrm{DB}}^{n,2-n}(H^*(A)). \]
  If this class is trivial, there exists a triple $(a_\Bbbk,a_\CC,\psi) \in \Pp C_{\mathrm{DB}}^{n-1,2-n}(H^*(A))$ such that
  \[ ((m_\Bbbk)_n, (m_\CC)_n, \varphi_{n-1}) = \delta (a_\Bbbk,a_\CC,\psi) = (\delta a_\Bbbk, \delta a_\CC, a_\Bbbk - a_\CC - \delta \psi).\]
  By Proposition $3.3$ (b) of \cite{Bashar}, there exists a filtered $\Pp_\infty$-algebra $(H,m_\Bbbk')$ and a bifiltered $\Pp_\infty$-algebra $(H\otimes \CC,m_\CC')$ such that the collections
  \begin{align*}
      &g_\Bbbk = (\mathrm{id},0,\dots,0,a_\Bbbk,0,\dots) \\
      &g_\CC = (\mathrm{id},0,\dots,0,a_\CC,0,\dots)
  \end{align*}
  define filtered and bifiltered $\infty$-morphisms 
  \begin{align*}
      &g_\Bbbk : (H,m_\Bbbk) \to (H,m'_\Bbbk) \\
      &g_\CC : (H \otimes \CC,m_\CC) \to (H \otimes \CC,m'_\CC).
  \end{align*} 
  The composition of $\infty$-morphisms
  \[ g_\CC \circledcirc \varphi \circledcirc (g_\Bbbk \otimes \CC)^{-1}\]
  satisfies 
  \begin{align*}
      &(g_\CC \circledcirc \varphi \circledcirc (g_\Bbbk \otimes \CC)^{-1})_{i} = 0, \quad \text{ for } i < n-1, \\
      &(g_\CC \circledcirc \varphi \circledcirc (g_\Bbbk \otimes \CC)^{-1})_{n-1} = \delta(- \psi).
  \end{align*}   
  Then, by Lemma \ref{obs-lemm}, it follows that this composition, seen as a Maurer-Cartan element of 
  \[  \Hom(\mathrm{B}_\iota(H,m'_\Bbbk \otimes \CC), (H,m'_\CC))\]
  is gauge equivalent to an $\infty$-morphism $\varphi'$ that satisfies
  \begin{align}
      \varphi_i = 0, \quad \text{ for } i \leq n-1.
  \end{align}
  The resulting $\Pp_\infty$-mixed Hodge diagram 
  \[ (H,m_\Bbbk',m_\CC',\varphi')\]
  satisfies 
  \begin{align*}
      m'_\Bbbk &= (0,m_2,0,\dots,0,(m_\Bbbk)_{n+1},\dots), \\
      m'_\CC &= (0,m_2,0,\dots,0,(m_\CC)_{n+1},\dots), \\
      \varphi' &= (\mathrm{id},0,\dots,0,\varphi_{n},\dots).
  \end{align*}
  This allows to define, in the same way, the $n+1$-st obstruction. If all obstructions are trivial, then there is a sequence of $\Pp_\infty$-mixed Hodge diagrams 
  \[ (H,m_\Bbbk,m_\CC,\varphi) \xrightarrow[]{f^{(1)}} (H,m_\Bbbk',m_\CC',\varphi') \xrightarrow[]{f^{(2)}} \cdots\]
  connected by ho-$\infty$-isotopies $f^{(i)}$ and that converges to the $\Pp$-mixed Hodge diagram $(H,m_2)$. Moreover, each ho-$\infty$-isotopy $f^{(n)} = (f^{(n)}_\Bbbk, f^{(n)}_\CC, F^{(n)})$ in this sequence has trivial components $(f^{(n)}_\Bbbk)_i$, $(f^{(n)}_\CC)_i$ and $(F^{(n)})_{i-1}$ for $2 \leq i \leq n$.
  This implies that the composition of all of them is well-defined and gives a ho-$\infty$-isotopy
  \[ (H,m) \xrightarrow[]{f} (H,m_2), \]
  which implies that $A$ is formal. We prove the claim that the infinite composition of the ho-$\infty$-isotopies is well-defined. By Lemma \ref{comphoinfmorfarity-lemm}, given $n \geq 2$, we have that
  \[ f_i = (f^{(1)} \circledcirc \cdots \circledcirc f^{(n-1)})_i, \quad \text{for } i \leq n,\]
  where we denote $f_i = ((f_\Bbbk)_i,(f_\CC)_i,F_{i-1})$ and similarly for the composition on the right-hand side. Hence, the infinite composition $f$ has well-defined components for each arity $n$. Moreover, the equations of ho-$\infty$-isotopies for $f$ at arity $n$ only involve the components $f_i$ for $i \leq n$ and, by Proposition \ref{comphoinfmorph-prop}, the composition of ho-$\infty$-morphisms $f^{(1)} \circledcirc \cdots \circledcirc f^{(n-1)}$ is well-defined and so it satisfies these equations.
\end{proof}
Note that each obstruction $\theta_k$ depends on the choices of boundaries that witness the vanishing of the previous obstructions. Thus, the non-vanishing of $\theta_k$ does not immediately imply non-formality, as there could be another different sequence of obstructions that vanish. We now prove that the first obstruction is independent of choices and, in certain cases, the second obstruction as well. 

\begin{prop}
    \label{firstobswelldef-prop}
    Let $A \in \MHD$. The first obstruction $\theta_3$ to formality of $A$ does not depend on choices. Furthermore, suppose that one of the following conditions holds:
    \begin{enumerate}
        \item $H^*(A)$ is concentrated in even degrees or
        \item $A$ is a real mixed Hodge diagram and $H^n(A)$ has pure Hodge structures of weight $n$ for each degree $n \in \ZZ$.
    \end{enumerate}
    Then, $\theta_3$ is trivial so the second obstruction $\theta_4$ is defined and this obstruction does not depend on choices either.      
\end{prop}
\begin{proof}
    Suppose that $A$ has length $1$ and let $(H^*(A),m_\Bbbk,m_\CC,\varphi)$ be a minimal $\Pp_\infty$-model of $A$. Identify $H^*(A_\Bbbk) \otimes \CC = H^*(A_\CC)$ and suppose that $\varphi_1 = \mathrm{id}$. The first obstruction is the class 
    \[ \theta_3 = [(m_\Bbbk)_3,(m_\CC)_3,\varphi_2] \in \Pp H_{\mathrm{DB}}^{3,-1}(H^*(A)). \]
    Thus the class depends, \textit{a priori}, on the minimal $\Pp_\infty$-model. However, if 
    \[ (H^*(A),m'_\Bbbk, m_\CC', \varphi') \]
    is another such model of $A$, then there is a ho-$\infty$-isotopy 
    \[ (H^*(A),m_\Bbbk,m_\CC,\varphi) \xrightarrow[]{f} (H^*(A),m_\Bbbk',m_\CC',\varphi').\]
    In particular, there are $\infty$-isotopies
    \begin{align*}
        (H^*(A),W,m_\Bbbk) &\xrightarrow[]{f_\Bbbk} (H^*(A),W,m_\Bbbk'), \\
        (H^*(A),W,F,m_\CC) &\xrightarrow[]{f_\CC} (H^*(A),W,,F,m_\CC').
    \end{align*}
    By Proposition $3.3$ (a) of \cite{Bashar}, it follows that
    \begin{align*}
        (m_\Bbbk)_3 - (m_\Bbbk')_3 &= \delta((f_\Bbbk)_2), \\
        (m_\CC)_3 - (m_\CC')_3 &= \delta((f_\CC)_2).
    \end{align*}
    There is, moreover, a gauge $h$ between $\varphi' \circledcirc (f_\Bbbk \otimes \CC)$ and $f_\CC \circledcirc \varphi$. 
    By Lemma \ref{gauge-eq}, it follows that
    \[  f_\CC \circledcirc \varphi= \varphi' \circledcirc (f_\Bbbk \otimes \CC) + h \circ d_{\mathrm{B}_\iota(H^*(A),m)} + l_2((\varphi' \circledcirc (f_\Bbbk \otimes \CC),h) +  \cdots\]
    Note that
    \begin{align*}
        &(\varphi' \circledcirc (f_\Bbbk \otimes \CC))_2 = \varphi'_2 + (f_\Bbbk)_2 \otimes \CC, \\
        &(f_\CC \circledcirc \varphi)_2 = \varphi_2 + (f_\CC)_2.
    \end{align*}
    Moreover, we have that
    \begin{align*}
        &(h \circ d_{\mathrm{B}_\iota(H^*(A),m)})_2 = h_1 \star m_2, \\
        &(l_2((\varphi' \circledcirc (f_\Bbbk \otimes \CC),h))_2 = m_2 \star h_1
    \end{align*}
    and the other summands have trivial arity $2$ component. It then follows that
    \[ \varphi_2 = \varphi'_2 + (f_\Bbbk)_2 \otimes \CC - (f_\CC)_2 - \delta(h_1).\]
    Hence,
    \[ ((m_\Bbbk)_3,(m_\CC)_3,\varphi_2) = ((m_\Bbbk')_3,(m_\CC')_3,\varphi_2') + \delta((f_\Bbbk)_2,(f_\CC)_2, h_1).\]
    This proves the first claim. We assume now that $H^*(A)$ is concentrated in even degrees and prove that $\theta_3$ is trivial and that $\theta_4$ does not depend on the choice of a model of $A$. Note that all the morphisms in the previous equation are of degree $-1$. Therefore, if $H^*(A)$ is concentrated in even degrees, then these morphisms are trivial. In this case, the second obstruction is defined and given by
    \[ \theta_4 = [(m_\Bbbk)_3,(m_\CC)_3,\varphi_2] \in \Pp H_{\mathrm{DB}}^{4,-2}(H^*(A)).\]
    The existence of a ho-$\infty$-isotopy 
     \[ (H^*(A),m_\Bbbk,m_\CC,\varphi) \xrightarrow[]{f} (H^*(A),m_\Bbbk',m_\CC',\varphi')\]
     then implies that 
     \begin{align*}
        (m_\Bbbk)_4 - (m_\Bbbk')_4 &= \delta((f_\Bbbk)_3), \\
        (m_\CC)_4 - (m_\CC')_4 &= \delta((f_\CC)_3).
    \end{align*}
    This follows again by Proposition $3.3$ (a) of \cite{Bashar} and the fact that $(f_\Bbbk)_2$ and $(f_\CC)_2$ are trivial. The gauge equation
    \[ h \cdot (\varphi' \circledcirc (f_\Bbbk \otimes \CC)) =  \varphi \circledcirc f_\CC\]
    in arity $3$ translates to
    \[ \varphi_3 = \varphi'_3 + (f_\Bbbk)_3 \otimes \CC - (f_\CC)_3 - \delta(h_2).\]
    This is because the other terms in arity $3$ involve compositions of $(f_\Bbbk)_2$, $(f_\CC)_2$ and $h_1$, which are trivial. It then follows that
    \[ ((m_\Bbbk)_4,(m_\CC)_4,\varphi_3) = ((m_\Bbbk')_4,(m_\CC')_4,\varphi_3') + \delta((f_\Bbbk)_3,(f_\CC)_3, h_2)\]
    and the second obstruction $\theta_4$ is independent of the minimal $\Pp_\infty$-model of $A$. Lastly, we suppose that $A$ is a real mixed Hodge diagram with pure Hodge structures on cohomology of weight equal to cohomological degree and show that $\theta_3$ is trivial and $\theta_4$ does not depend on the choice of model of $A$. Let $(H^*(A),m_\Bbbk,m_\CC,\varphi)$ be a minimal $\Pp_\infty$-model of $A$ and recall that $\theta_3$ is given by
    \[ \theta_3 = [(m_\RR)_3,(m_\CC)_3,\varphi_2]. \]
    Note that $\varphi_2$ is a morphism of degree $-1$ of the form
    \[ \varphi_2 : H^*(A_\CC)^{\otimes 2} \to H^*(A_\CC).\]
    Since $H^*(A)$ has pure Hodge structures of weight equal to degree, then same is true for $H^*(A)^{\otimes 2}$ and for the inner Hom
    \[ \underline{\Hom}(H^*(A)^{\otimes 2},H^*(A)).\]
    In particular, the space $\underline{\Hom}^{-1}(H^*(A)^{\otimes 2},H^*(A))$ of morphisms of degree $-1$ has a pure Hodge structure of weight $-1$. Decompose $\varphi_2$ into its $(p,q)$ components
    \[ \varphi_2 = \sum_{p+q = -1} \varphi_2^{p,q}.\]
    Note that $\varphi_2^{p,q}$ preserves the Hodge filtration of $H^*(A)$ for $p \geq 0$. Moreover, for $p < 0$, the complex conjugate $\overline{\varphi_2^{p,q}}$
    also preserves the Hodge filtration. Denote 
    \[ (\varphi_2)_\CC := \sum_{p \geq 0} \varphi_2^{p,q} - \sum_{p<0} \overline{\varphi_2^{p,q}}.\]
    Note that the morphism
    \[ (\varphi_2)_\RR = \sum_{p<0} (\varphi_2^{p,q} + \overline{\varphi_2^{p,q}})\]
    preserves the real structure of $H^*(A)$. It then follows that
    \[ \varphi_2 = (\varphi_2)_\RR + (\varphi_2)_\CC \]
    is the sum of a real morphism with one preserving the Hodge filtration. By Proposition $3.3$ (a) of \cite{Bashar}, we have that
    \[ \delta(\varphi_2) = (m_\RR)_3 - (m_\CC)_3. \]
    It then follows that
    \[ \delta(\varphi_2)_\RR - (m_\RR)_3 = - (\delta(\varphi_2)_\CC + (m_\CC)_3). \]
    The left-hand side is a morphism of real vector spaces, the right-hand side preserves the Hodge filtration. This equality thus implies that both sides are morphisms of mixed Hodge structures. Since there are no non-trivial morphisms of mixed Hodge structures between pure Hodge structures of different weights, it follows that
    \[ \delta(\varphi_2)_\RR - (m_\RR)_3 = - (\delta(\varphi_2)_\CC + (m_\CC)_3) = 0.\]
    This implies that the triple 
    \[ ((\varphi_2)_\RR,(\varphi_2)_\CC,0) \]
    is a boundary for the cycle representing $\theta_3$, so this obstruction is trivial. As in the proof of Theorem \ref{obst-theo}, such a boundary defines a $\Pp_\infty$-mixed Hodge diagram $(H^*(A),m_\RR^{(2)}, m_\CC^{(2)},\varphi^{(2)})$ and a ho-$\infty$-isotopy 
    \[ (H^*(A),m_\RR, m_\CC,\varphi) \xrightarrow[]{} (H^*(A),m_\RR^{(2)}, m_\CC^{(2)},\varphi^{(2)}).\]
    The second obstruction $\theta_4$ is then given by
    \[ \theta_4 = [(m_\RR^{(2)})_4,(m_\CC^{(2)})_4, \varphi^{(2)}_3]. \]
    Suppose now that $((H^*(A),m_\RR',m_\CC',\varphi')$ is another minimal $\Pp_\infty$-model of $A$ such that 
    \[ (m_\RR')_3 = 0, \quad (m_\CC')_3 = 0 \quad \text{and} \quad (\varphi'_2) = 0.\]
    The second obstruction is, in this case, given by
    \[ \theta_4' = [(m_\RR')_4,(m_\CC')_4, \varphi'_3].\]
    We finish the proof by showing that $\theta_4 = \theta'_4$. Being both models of $A$, there exists a ho-$\infty$-isotopy
    \[ (H^*(A),m_\RR^{(2)}, m_\CC^{(2)},\varphi^{(2)}) \xrightarrow[]{f} (H^*(A),m'_\RR, m'_\CC,\varphi'). \]
    In particular, there are $\infty$-isotopies
    \begin{align*}
        &(H^*(A),m_\RR^{(2)}) \xrightarrow[]{f_\RR} (H^*(A),m'_\RR), \\
        &(H^*(A),m_\CC^{(2)}) \xrightarrow[]{f_\CC} (H^*(A),m'_\CC)
    \end{align*}
    and a gauge $h$ between $\varphi' \circledcirc (f_\RR \otimes \CC)$ and $f_\CC \circledcirc \varphi^{(2)}$. Since $\varphi^{(2)}_2 = \varphi'_2 = 0$, the gauge equations yield the relation
    \[ (f_\RR)_2 \otimes \CC = (f_\CC)_2 - \delta(h_1).\]
    By the same reasoning as before, $h_1$ (being a morphism of degree $-1$ between pure Hodge structures of weight equal to degree) is of the form
    \[ h_1 = (h_1)_\RR + (h_1)_\CC, \]
    where $(h_1)_\RR$ is a morphism of real vector spaces and $(h_1)_\CC$ preserves the Hodge filtration. This implies the equality
    \[ (f_\RR)_2 \otimes \CC - \delta(h_1)_\RR = (f_\CC)_2 + \delta(h_1)_\CC. \]
    Again, by the same reasoning as before, it follows that 
    \[ (f_\RR)_2 \otimes \CC - \delta(h_1)_\RR = 0 = (f_\CC)_2 + \delta(h_1)_\CC.\]
    Then, Lemma \ref{obs-lemm} implies that $f_\RR$ is gauge equivalent to an $\infty$-isotopy $f_\RR'$ such that $(f_\RR')_2 = 0$ and similarly for $f_\CC$. By transitivity of the gauge relation, there exists a ho-$\infty$-isotopy
    \[ (H^*(A),m_\RR^{(2)}, m_\CC^{(2)},\varphi^{(2)}) \xrightarrow[]{f'} (H^*(A),m'_\RR, m'_\CC,\varphi'). \]
    such that $(f_\RR')_2 = 0 = (f_\CC')_2$. Denote by $h$ the gauge between $\varphi' \circledcirc (f_\RR' \otimes \CC)$ and $f_\CC' \circledcirc \varphi^{(2)}$. For arity $3$, the equations for the gauge relation imply that
    \[ \varphi^{(2)}_3 = \varphi'_3 + (f_\Bbbk)_3 \otimes \CC - (f_\CC)_3 - \delta(h_2), \]
    as the other possible terms involve a composition of $h_1$ with $(f_\RR)_2$ or $(f_\CC)_2$, which are trivial. This equation then implies that the cycles representing $\theta_4$ and $\theta_4'$ differ by a boundary:
     \[ ((m_\Bbbk^{(2)})_4,(m_\CC^{(2)})_4,\varphi^{(2)}_3) = (m_\Bbbk')_4,(m_\CC')_4,\varphi_3') + \delta((f_\Bbbk')_3,(f_\CC')_3, h_2).\]
\end{proof}
As a consequence, we have:
\begin{coro}
    If the first obstruction $\theta_3$ to formality of $A \in \Pp\textrm{-}\MHD$ is non-trivial, then $A$ is not formal. Similarly, if $A$ satisfies one of the conditions of the previous proposition and $\theta_4$ is non-trivial, then it follows that $A$ is not formal.
\end{coro}
\begin{proof}
    Let $(A,\mu)$ be a mixed Hodge diagram with $\Pp$-algebra structure denoted by $\mu$. The corollary follows from the fact that if $A$ is formal, then $(H^*(A),\mu^*,\mu^*,\mathrm{id})$ is a minimal $\Pp_\infty$-model of $A$. The higher components of this model are all trivial so it yields trivial first and second obstructions. By Proposition \ref{firstobswelldef-prop}, these do not depend on the choice of model, so they must be trivial.
\end{proof}

\subsection{The first obstruction I}
\label{firsobsI-sec}
In this section we relate part of the first obstruction of Theorem \ref{obst-theo} to the splitting of mixed Hodge structures on homotopy groups. We also give an example illustrating that formality of mixed Hodge diagrams does not satisfy descent of formality.

Let $A$ be a $1$-connected $\Pp$-mixed Hodge diagram. That is, $H^0(A) \cong k$, $H^{<0}(A) = 0$ and $H^1(A)=0$. Assume also that $A$ has length $1$ (see Remark \ref{incmhdlength-equiv-rema}). Recall that, by Theorem \ref{htt-mhd}, there exists a minimal $\Pp_\infty$-model of $A$:
\[ H = (H^*(A_\Bbbk),H^*(A_\CC),\hat{\varphi}). \] 
Then, $\mathrm{B}_\iota H$ (see Theorem \ref{mhd-barcobar}) is a mixed Hodge structure with Hodge filtration given by $\mathrm{B}_{\iota}\hat{\varphi}^{-1}(F)$, where $F$ is the free extension of the Hodge filtration of $H^*(A_\CC)$ (see Remark \ref{cdgamodelMHD-rema}). Consider the filtration of $\mathrm{B}_\iota H$ given by arity
\[ T_k \mathrm{B}_\iota H := \bigoplus_{i \leq k} \Pp(i) \otimes_{\SS_i} H^{\otimes i}. \]
Here, for simplicity, we identify $H^*(A_\Bbbk) \otimes \CC \cong H^*(A_\CC)$ and denote them by $H$. The first terms of the $E_1$ page of the spectral sequence associated to $(\mathrm{B}_\iota H,T)$ are given by 
\[ \begin{tikzcd}[row sep = tiny]
      (\Pp(3) \otimes_{\SS_3} H^{\otimes 3})^4 \arrow[r,"\delta"] & (\Pp(2) \otimes_{\SS_2} H^{\otimes 2})^5 \arrow[r,"\delta"] & H^6 \\
      0 & (\Pp(2) \otimes_{\SS_2} H^{\otimes 2})^4 \arrow[r,"\delta"] & H^5 \\
       0 & (\Pp(2) \otimes_{\SS_2} H^{\otimes 2})^3 \arrow[r,"\delta"] & H^4 \\
       0 & 0 & H^3 \\
       0 & 0 & H^2
  \end{tikzcd} \]
where $\delta$ is the differential of operadic cohomology. There is an induced filtration $T$ on $\pi^*(A) = H^*(\mathrm{B}_\iota(H))$. Recall that $\Pp = Ass, Com$ or $Lie$ and so the comparison morphism $\hat{\varphi}$ is an $\infty$-morphism with components
\[ \hat{\varphi}_k : (\Pp(k) \otimes_{\SS_k} H^{\otimes k},W) \to (H,W). \]
\begin{lemm}
  \label{obs-ses-lemm}
  Given an integer $k \geq 2$, suppose that $\hat{\varphi}_i = 0$ for $1 < i < k$. Then, $\hat{\varphi}_k$ seen as an element
    \[
       \hat{\varphi}_k \in \mathrm{Ext}_{\mathsf{MHS}}^1(\Pp(k) \otimes_{\SS_k} H^{\otimes k}, T_{k-1} \mathrm{B}_\iota(H))
    \]
  is the obstruction to the splitting of the short exact sequence of mixed Hodge structures
    \begin{equation}
    \label{ses-k-eq}
    0 \to T_{k-1} \mathrm{B}_\iota(H) \to T_{k} \mathrm{B}_\iota(H) \to Gr_k^T \mathrm{B}_\iota(H) \to 0
    \end{equation} 
\end{lemm}
\begin{proof}
  The obstruction to the splitting of (\ref{ses-k-eq}) is given by the composition of a left inverse $l$ to
  \[ T_{k-1} \mathrm{B}_\iota(H) \to T_{k} \mathrm{B}_\iota(H) \]
  which is a morphism of complexes over $\Bbbk$ together with a right inverse $r$ of
  \[ T_{k} \mathrm{B}_\iota(H) \otimes \CC \to Gr_k^T \mathrm{B}_\iota(H) \otimes \CC \]
  which preserves the Hodge filtration (see \cite{mhsext-carlson}). We can take $l$ to be the projection and $r = (\mathrm{B}_\iota \hat{\varphi})^{-1}$. Since $\hat{\varphi}_i = 0$ for $1 < i < k$, we have that
  \[ (\mathrm{B}_\iota \hat{\varphi})^{-1}(x) = x - \hat{\varphi}_k(x), \quad \text{for } x \in \Pp(k) \otimes_{\SS_k} H^{\otimes k}. \]
  Hence, $l r = - \hat{\varphi}_k$.
\end{proof}
Let us denote the product in cohomology of degree $2$ by 
\[ \mu : \Pp(2) \otimes_{\SS_2} (H^2(A))^{\otimes 2} \to H^4(A). \] 
Since $A$ is $1$-connected, there is a well-defined map
\begin{align}
  \label{map-R}
  R : \Pp H_{\mathrm{DB}}^{3,-1}(H^*(A)) &\to \mathrm{Ext}_{\MHS}^1(\Ker(\mu),H^3(A)) \\
  [m_\Bbbk,m_\CC,\varphi] &\mapsto \varphi|_{\Ker{\mu}}. \nonumber
\end{align}
For a class $[m_\Bbbk,m_\CC,\varphi] \in \Pp H_{\mathrm{DB}}^{3,-1}(H^*(A))$, the morphism 
\[ \varphi :\Pp(2) \otimes_{\SS_2} H^*(A)^{\otimes 2} \to H^*(A) \] 
is of degree $-1$. Moreover, if $[m_\Bbbk,m_\CC,\varphi] = 0$, there exists a tuple $(g_\Bbbk,g_\CC,h)$, where 
\[ h : \Pp(2) \otimes_{\SS_2} H^*(A)^{\otimes 2} \to H^*(A) \]
is also of degree $-1$ and $\varphi = \delta h + g_\CC - g_\Bbbk$ (recall Definition \ref{oper-cohom}). If $A$ is $1$-connected, then 
\[ \delta h(a \otimes b) = 0 \quad \textrm{for } a \otimes b \in \Ker(\mu) \subset \Pp(2) \otimes_{\SS_2} (H^2(A))^{\otimes 2}, \] 
so the map $R$ is indeed well-defined. By Theorem \ref{obst-theo}, there is a first obstruction 
\[ \theta_3 \in \Pp H_{\mathrm{DB}}^{2,-1}(H^*(A)) \] 
to formality of $A$.

\begin{prop}
  \label{first-obs}
  The element $R(\theta_3)$ is the obstruction to the splitting of the following short exact sequence of mixed Hodge structures:
  \begin{equation}
    \label{mhs-seq1}
    0 \to H^3(A) \to \pi^3(A) \to \Ker(\mu) \to 0.
  \end{equation}
\end{prop}
\begin{proof}
  Let $(H,\hat{\varphi})$ be a minimal $\Pp_\infty$-model of $A$. By Lemma \ref{obs-ses-lemm} for $k=2$, we have that $\hat{\varphi}_2$ is the obstruction to the splitting of
  \begin{equation}
    \label{ses-obs-proof}
     0 \to H \to T_2 \mathrm{B}_\iota(H) \to \Pp(2) \otimes_{\SS_2} H^{\otimes 2} \to 0.
  \end{equation}
  The $E_2$ page of the spectral sequence of $(\mathrm{B}_\iota H, T)$ in low degrees is given by
  \[ \begin{tikzcd}[row sep = tiny]
      \Ker(\mu) & \Coker(\mu) \\
        0 & H^3 \\
        0 & H^2
  \end{tikzcd} \]
  By degree reasons, $E_\infty = E_2$ for these degrees. Restricting the short exact sequence (\ref{ses-obs-proof}) to degree $3$ of $\mathrm{B}_\iota(H)$ and taking cohomology, we get the short exact sequence of the statement. The obstruction to the splitting of this sequence is precisely 
  \[ R(\theta_3) = \hat{\varphi}_2|_{\Ker(\mu)}.\qedhere \]
\end{proof}

\begin{exam}
  \label{exam-fieldext}
  We end this section with an example of a commutative mixed Hodge algebra $A$ with coefficients in $\QQ$, which is formal as a rational cdga but not formal as a mixed Hodge diagram. Let us define the rational cdga
  \begin{align*}
      &A = \Lambda (x_1,x_2,y_1,y_2,y_{12})/(\mathrm{deg} \geq 5),\\
      &|x_1| = |x_2| = 2, \quad |y_1| = |y_2| = |y_{12}| = 3, 
  \end{align*} 
  with the only non-trivial differential on generators being $d y_{12} = x_1 x_2$. Define the weight filtration on the generators of $A$ by
  \[W_2 = \langle x_1, x_2 \rangle \subset W_3 = \langle x_1,x_2,y_1,y_2 \rangle \subset W_4 = A. \]
  Fix a number $\lambda \in \RR \backslash \QQ$. Writing $y = y_1 + i y_2$, define the Hodge filtration on $A \otimes \CC$ by
  \[ F^1 = A \otimes \CC \supset F^2 = \big\langle \overline{y}, y_{12} + \lambda y \big\rangle \supset F^3 = 0.\]
  These filtrations endow $A$ with a mixed Hodge algebra structure whose cohomology is pure. Therefore, $A$ is formal as a rational cdga. However, a primitive of $x_1 x_2$ which lies in $F^2$ is given by $y_{12} + \lambda y$. This implies that $\pi^3(A)$ is non split. Indeed, as a vector space, $\pi^3(A)$ is given by
  \[ \pi^3(A) \cong \langle y_1, y_2, y_{12} \rangle. \]
  Moreover, its filtrations are
  \begin{align*}
      &W_2 = 0 \subset W_3 = \langle y_1, y_2 \rangle \subset W_4 = \pi^3(A), \\
      &F^1 = \pi^3(A) \supset F^2 = \langle \overline{y}, y_{12} + \lambda y \rangle \supset F^3 = 0.
  \end{align*} 
  Computing the obstruction to the splitting of (\ref{mhs-seq1}) yields
  \[ R(\theta_3)( [x_1] \otimes [x_2] ) = \lambda [\overline{y}], \]
  which implies that $R(\theta_3) \neq 0 \in \mathrm{Ext}^1_{\MHS}(\Ker(\mu),H^3(A))$. Hence, by Proposition \ref{first-obs} and Proposition \ref{firstobswelldef-prop}, $A$ is not formal. Tensoring by $\RR$, this obstruction vanishes and actually the real mixed Hodge cdga $A \otimes \RR$ is formal as a mixed Hodge diagram. There is a direct quasi-isomorphism of mixed Hodge cdga's from $A \otimes \RR$ to $H^*(A \otimes \RR)$:
  \begin{align*}
    A \otimes \RR &\xrightarrow{f} H^*(A \otimes \RR) \\
    x_i &\mapsto [x_i] \\
    y_1 &\mapsto [y_1] \\
    y_2 &\mapsto  [y_2] \\
    y_{12} &\mapsto - 2 \lambda [y_1].
  \end{align*}
  Replacing $\QQ$ by another field $\QQ \subset \Bbbk \subset \CC$ and $\RR$ by a field containing $\Bbbk$, the same example shows that mixed Hodge formality over a field extension does not imply mixed Hodge formality over the original field. This is in sharp contrast with classical formality (see \cite{Sullivan}).
\end{exam}

\subsection{Purity and formality}
\label{pure-sec}
We now focus on mixed Hodge diagrams whose cohomology is $\alpha$-pure. This notion was introduced by Cirici and Horel in \cite{CiHo} as a generalization of the notion of purity achieved by smooth complex projective varieties. There are many interesting complex algebraic varieties which have $\alpha$-pure cohomology and the results of the previous sections get simplified in this case.

\begin{defi}
  \label{alpha-pure}
  Given a positive rational number $\alpha \in \QQ$, a graded mixed Hodge structure $H^* \in Gr(\MHS)$ is said to be \textit{$\alpha$-pure} if $H^n$ has a pure Hodge structure of weight $\alpha n$, for all $n$. When $\alpha n$ is not an integer, then $H^n$ is the trivial vector space.
\end{defi}
For $A \in \Pp\text{-}\mathsf{MHD}$ with $\alpha$-pure cohomology, we can rewrite the obstructions to formality.
\begin{theo}
  \label{obst-col}
  Let $A \in \Pp\textrm{-}\mathsf{MHD}$ be such that $H^*(A)$ is $\alpha$-pure. Then there exist successively defined classes
  \begin{equation*}
    \phi_k \in \Pp H_{\mathrm{Ext}}^{k,1-k}(H^*(A)), \mathrm{\qquad} k \geq 2,
  \end{equation*}
  such that if all are trivial, then $A$ is formal.
\end{theo}
\begin{proof}
  There is a short exact sequence
  \begin{align*}
    \Pp H_{\MHS}^{k,m}(H^*(A)) &\to \Pp H_{\mathrm{DB}}^{k,m}(H^*(A)) \to \Pp H_{\mathrm{Ext}}^{k-1,m}(H^*(A)). \\
    [f] &\mapsto [f,f,0]
  \end{align*}
  Since $H^*(A)$ is $\alpha$-pure, there are no non-trivial morphisms in 
  \[\Hom^{m}_{\MHS}(\Pp^\antishriek(k) \otimes_{\SS^k}(H^*(A))^{\otimes k},H) \]
  for $m \neq 0$. Hence, $\Pp H_{\MHS}^{k,m}(H^*(A)) = 0$ for $m \neq 0$. This implies that the natural morphism
  \begin{align}
    \label{mapPHDBPHExt-eq}
    \Pp H_{\mathrm{DB}}^{k,m}(H^*(A)) &\xrightarrow{\pi} \Pp H_{\mathrm{Ext}}^{k-1,m}(H^*(A)) \\ \nonumber
    [f_\Bbbk,f_\CC,h] &\mapsto [h]
  \end{align}
  is injective for $m \neq 0$. The injectivity of this morphism together with Theorem \ref{obst-theo} prove the result.
\end{proof}

\begin{rema}
  \label{weighpi-alpha-rema}
  If $A \in \Pp\text{-}\mathsf{MHD}$ has $\alpha$-pure cohomology then the weight filtration of $\pi^n(A)$ is related to the filtration induced by arity of $\mathrm{B}_\kappa(H^*(A))$ in the following way:
  \[ T_k \pi^n(A) = W_{\alpha(n+k-1)} \pi^n(A). \]
  In particular, $W_m \pi^n(A)$ is trivial for $m < \alpha n$. This follows from the fact that given a minimal $\Pp_\infty$-model $(H^*(A),m)$ of $A$, then
  \[ (\Pp(k) \otimes_{\SS_k} H^*(A)^{\otimes k})^n = \Pp(k) \otimes_{\SS_k} \big(H^*(A)^{\otimes k}\big)^{n+k-1} \]
  is concentrated in weight $\alpha(n+k-1)$. So \[ T_k \mathrm{B}_\iota(H^*(A))^n = W_{\alpha(n+k-1)} (\mathrm{B}_\iota (H^*(A)))^n \] and the same happens for $\pi^n(A) = H^n(\mathrm{B}_\iota(H^*(A)))$.
\end{rema}

The obstructions in Theorem \ref{obst-col} can be interpreted as a measure of the highest level of the weight filtration for which a model of $A \in \Pp\textrm{-}\mathsf{MHD}$ is split up to that weight level. More precisely,

\begin{prop}
  \label{inter-obs}
  Let $A \in \Pp\textrm{-}\mathsf{MHD}$ be such that $H^*(A)$ is $\alpha$-pure. Then, 
  \begin{enumerate}
    \item \label{inter-obs1} given $k \geq 2$, there is a sequence of trivial obstructions $\varphi_i$ for $i \leq k$ if and only if $A$ is quasi-isomorphic to a mixed Hodge $\Pp$-algebra $B$ such that 
    \[ W_{\alpha(n + k -1)}B^n \]
    is split as a mixed Hodge structure for every $n$.
    \item \label{inter-obs2} The $\Pp$-mixed Hodge diagram $A$ is formal if and only if it is quasi-isomorphic to a split mixed Hodge $\Pp$-algebra.
  \end{enumerate}
\end{prop}
\begin{proof}
  The second statement is the limit case of the first. We prove (\ref{inter-obs1}). If there is a sequence of trivial obstructions up to $\varphi_k$, then there exists a $\Pp_\infty$-model of $A$ of the form
  \[ H = (H^*(A),m_\Bbbk,m_\CC,\hat{\varphi}), \]
  such that $\hat{\varphi}_i = 0$ for $2 \leq i \leq k$. Applying the bar-cobar functor of Proposition \ref{mhd-barcobar} to $H$, one gets a mixed Hodge $\Pp$-algebra $\Omega_\kappa \mathrm{B}_\iota H$ quasi-isomorphic to $A$. Moreover, its Hodge filtration is given by $\Omega_\kappa \mathrm{B}_\iota \hat{\varphi}^{-1}(F)$, where $F$ is the filtration on $\Omega_\kappa \mathrm{B}_\iota H^*(A) \otimes \CC$ induced by the Hodge filtration of $H^*(A) \otimes \CC$. Since $\hat{\varphi}_i = 0$ for $2 \leq i \leq k$, the restriction of $\mathrm{B}_\iota \hat{\varphi}$ to 
  \[ W_{\alpha(n+k-1)} (\mathrm{B}_\iota H \otimes \CC)^n = \bigoplus_{i \leq k} \Pp^{\antishriek}(i) \otimes_{\SS_i} \big(H^*(A)^{\otimes i}\big)^{n+i-1} \] 
  is just the identity (see Remark \ref{weighpi-alpha-rema}). Therefore, the restriction of $\Omega_\kappa \mathrm{B}_\iota \varphi$ to 
  \[ W_{\alpha(n+k-1)} (\Omega_\kappa \mathrm{B}_\iota H)^n \]
  is also the identity. This means that the graded sub-mixed Hodge structure 
  \[ W_{\alpha(* + k -1)} (\Omega_\kappa \mathrm{B}_\iota H)^* \] 
  is just the one induced by the pure Hodge structure of $H^*(A)$ by taking tensor products and direct sums. Hence, it is split. Therefore, $\Omega_\kappa \mathrm{B}_\iota  H$ is a mixed Hodge algebra quasi-isomorphic to $A$ satisfying the property of statement (\ref{inter-obs1}).

  On the other hand, if $A$ is quasi-isomorphic to some mixed Hodge algebra $B$ such that 
  \[ W_{\alpha(n + k -1)} B^n \]
  is split for all $n$, then one can choose contractions like in Lemma \ref{contrac-mhd} which are equal when restricted to 
  \[ W_{\alpha(*+k-1)} B^* \quad \text{and} \quad W_{\alpha(*+k-1)} H^*(B). \]
  The minimal $\Pp_\infty$-model $(H^*(A),m_\Bbbk,m_\CC,\hat{\varphi})$ constructed with these contractions (see Theorem \ref{htt-mhd}) has $\hat{\varphi}_i = 0$ for $2 \leq i \leq k$. Thus, there is a sequence of obstructions which are trivial up to $\phi_k$.
\end{proof}
\begin{rema}
  Note that the existence of a split model of $A \in \Pp\text{-}\mathsf{MHD}$ is stronger than asking for the mixed Hodge structures on $\pi^*(A)$ to be split. Indeed, by Proposition \ref{prop-split}, formality implies the splitting of such but, as we will later see in Example \ref{exam-obs2}, there exist mixed Hodge diagrams which have pure Hodge structures on cohomology and on $\Bbbk$-homotopy groups but which are not formal.
\end{rema}
The following is a simple consequence of Proposition \ref{inter-obs}.

\begin{lemm}
  \label{form-coform-lemm}
  Let $A \in \Pp\textrm{-}\mathsf{MHD}$. If $\mathrm{B}_\kappa A \in \Pp^\antishriek\textrm{-}\mathsf{MHD}$ is formal, $\pi^*(A)$ is split and $H^*(A)$ is $\alpha$-pure, then $A$ is formal.
\end{lemm}
\begin{proof}
  Note that $A \simeq \Omega_\kappa \mathrm{B}_\kappa A$. Since $\mathrm{B}_\kappa A \simeq H^*(\mathrm{B}_\kappa A)$ and \[ \pi^*(A) = H^*(\mathrm{B}_\kappa A) \] is split, it follows that $A \simeq \Omega_\kappa H^*(\mathrm{B}_\kappa A)$ is quasi-isomorphic to a split mixed Hodge $\Pp$-algebra. The result then follows from point (\ref{inter-obs2}) of Proposition \ref{inter-obs}.
\end{proof}
We now give a criterion for mixed Hodge diagrams to have pure Hodge structures on homotopy groups.
\begin{prop}
  \label{koszul-alpha-prop}
  Let $A \in \Pp\textrm{-}\mathsf{MHD}$ be such that $H^*(A)$ is $\alpha$-pure, Koszul and generated as an algebra by elements of a fixed degree $r \geq 2$. Then, $s^{-1}\pi^*(A)$ is $\alpha r / (r-1)$-pure.
\end{prop}
\begin{proof} 
  Let $H = H^*(A)$ and denote by $(H,m)$ a minimal $\Pp_\infty$-model of $A$. Since $H^*(A)$ is $\alpha$-pure, the filtered coalgebra $(\mathrm{B}_\iota(H,m),W)$ is isomorphic to $(\mathrm{B}_\kappa(H,m_2),W)$, where $m_2$ denotes the product of $H$. Since $H$ is Koszul, the inclusion
  \[ (H^\antishriek,W) \hookrightarrow (\mathrm{B}_\kappa(H),W) \]
  is a filtered quasi-isomorphism. Denote by $V$ the set of generators of $H$. Then,
  \[ H^\antishriek = \bigoplus_{k \geq 1} \Pp^\antishriek(k) \otimes_{\SS_k} V^{\otimes k} \]
  Since $V$ is concentrated in degree $r$, the $k$-th summand of $H^\antishriek$ is concentrated in degree $rk + 1-k$. Moreover, since $H$ is $\alpha$-pure, the $k$-th summand has weight $\alpha k r$. Therefore, for each $k \geq 1$, there is a non-trivial homotopy group in degree $rk + 1 -k$,
  \[ \pi^{rk+1-k}(A) \]
  and weight $\alpha kr$. Thus, the shifted homotopy group 
  \[ (s^{-1}\pi^*(A))^{rk - k} \]
  has weight $\alpha k r$ and so $s^{-1} \pi^*(A)$ is $\alpha r / (r-1)$-pure.
\end{proof}
The previous results imply a simple relation between formality of $A$ and formality of $\mathrm{B}_\kappa(A)$ in the conditions of Proposition \ref{koszul-alpha-prop}.
\begin{prop}
  \label{form-coform-prop}
  Let $A \in \Pp\textrm{-}\mathsf{MHD}$ be such that $H^*(A)$ is $\alpha$-pure, Koszul and generated as an algebra by elements of a fixed degree $r \geq 2$. Then $A$ is formal if and only if $\mathrm{B}_\kappa(A)$ is formal.
\end{prop}
\begin{proof}
  If $A$ is formal, then by Proposition \ref{form-koszuldual-prop}, $\mathrm{B}_\kappa(A)$ is formal. On the other hand, if $\mathrm{B}_\kappa(A)$ is formal, then Proposition \ref{koszul-alpha-prop} implies that the conditions of Lemma \ref{form-coform-lemm} are satisfied and so $A$ is formal.
\end{proof}
We can now use Proposition \ref{koszul-alpha-prop} to give an example of a non-formal commutative mixed Hodge algebra which has pure Hodge structures on cohomology and homotopy groups.
\begin{exam}
  \label{exam-obs2}
  Let $H$ be the commutative mixed Hodge algebra with trivial differential and underlying algebra given by
  \[ H = \Lambda(x_1, x_2, y,z) / J, \]
  where $J$ is the ideal generated by
  \[ J = \langle (y-x_1)^2, (y-x_2)^2, yx_1, yx_2, x_1x_2, z(z-y) \rangle. \]
  All generators have degree $2$ and $H$ is given degree-wise with its $1$-pure Hodge structure by
  \[ \begin{tikzcd}[row sep = 0.5ex]
    H^{3,3} & z^3 \\
    H^5 & 0 \\
    H^{2,2} & x_1 z, x_2 z, y^2, z^2 \\
    H^3 & 0 \\
    H^{1,1} & x_1, x_2, y, z \\
    H^1 & 0 \\
    H^0 & \Bbbk
  \end{tikzcd}
  \]
  It is straightforward to check that $H$ is Koszul and so it satisfies the conditions of Proposition \ref{koszul-alpha-prop}. It thus follows that any mixed Hodge diagram with cohomology $H$ has pure Hodge structures on its homotopy groups. We now construct a non-trivial $\infty$-automorphism $\varphi$ of $H$ (its comparison morphism) and apply the functor $\Omega_\kappa \mathrm{B}_\iota$ to get a non-formal mixed Hodge diagram. We define $\varphi$ to be trivial everywhere, except on 
  \[ \varphi_3 : (H^2 \otimes \CC)^{\otimes 3} \to (H \otimes \CC)^4.\]
  And on this vector space, it is given by
  \begin{align*}
    &\varphi_3(x_j,u,x_k) = \begin{cases}
    i z^2 & j \neq k, \, u = y,z \\
    i x_j z & j = k, \, u = y,z,
    \end{cases} \\
    &\varphi_3(x_j,x_k,x_l) = \begin{cases}
    - i z^2 & j = k = l \\
    0 & \text{otherwise},
    \end{cases}  \\
    &\varphi_3(u,x_j,x_k) = \varphi_3(x_j,x_k,u) =\begin{cases}
    - i x_j z & j = k, \, u = y,z \\
    0 & j \neq k, u = y,z,
    \end{cases}  \\
    &\varphi_3(u,x_j,v) = - i (y-x_k) z, \quad k \neq j, \, u,v = y,z \\
    &\varphi_3(x_j,u,v) = \varphi_3(u,v,x_j) = i z^2, \quad u,v = y,z \\
    &\varphi_3(u,v,w) = i(x_1 + x_2)z, \quad u,v,w = y,z
  \end{align*}
  By degree reasons, for $\varphi$ to be an $\infty$-morphism, it must only satisfy the equation $\delta(\varphi_3) = 0$, where $\delta$ is the differential of the Hochschild complex. This is a straightforward check. In particular, $\varphi_3$ is a Hochschild cycle and gives an element
  \[ [\varphi_3] \in HH_{\mathrm{Ext}}^{3,-2}(H), \]
  where the cohomology group is the $\mathrm{Ext}$-operadic cohomology for the associative operad (see Definition \ref{oper-cohom}). We now show that this class is non-trivial. If it were, then there would exist maps
  \[ f_3 \in \Hom_\Bbbk((H^2)^{\otimes 3}, H^4), \quad g_3 \in F^0 \Hom_\CC((H^2 \otimes \CC)^{\otimes 3},(H \otimes \CC)^4), \]
  \[ h_2 \in \Hom_\CC((H^2 \otimes \CC)^{\otimes 2}, H^4), \]
  such that
  \begin{equation}
    \label{phi3triv-eq}
    \varphi_3 = f_3 - g_3 + \delta(h_2).
  \end{equation} 
  By degree reasons, $g_3$ has to be trivial. Hence, equation (\ref{phi3triv-eq}) would imply
  \[ \varphi_3(x_1,y,x_2) = iz^2 = f_3(x_1,y,x_2) + x_1 h_2(y,x_2) + h_2(x_1,y) x_2 \]
  By isolating the imaginary part of this equation, we may assume that $f_3(x_1,y,x_2)$ is trivial. However, the resulting equation 
  \[ i z^2 = x_1 h_2(y,x_2) + h_2(x_1,y) x_2 \]
  has no solution, since by multiplying both sides by $y$ would yield
  \[ i z^3 = i z^2 y = 0, \]
  which is not true. Applying then $\Omega_\kappa \mathrm{B}_\iota$ yields a mixed Hodge diagram $\Omega_\kappa \mathrm{B}_\iota(H)$ with cohomology $H$ and with a non-trivial obstruction to formality 
  \[ [\varphi_3] \in HH_{\mathrm{Ext}}^{3,-2}(H). \]
  Since $H$ is concentrated in even degrees, it follows, by Proposition \ref{firstobswelldef-prop}, that $\Omega_\kappa \mathrm{B}_\iota(H)$ is not formal. 
\end{exam}

\subsection{Formulas in the associative case}
\label{AssformMHD-sec}
In the case of associative mixed Hodge diagrams, there are explicit formulas for the infinity algebras, morphisms and gauge involved in the homotopy transfer theorem, Theorem \ref{htt-mhd}. We use these formulas to compute the second obstruction of Theorem \ref{obst-col}. We first recall the formulas of the homotopy transfer theorem for an associative algebra $A$ (see, for instance, section $3$ of \cite{CiSo}). See also \cite{Markl-trans-ass} for the general formulas where $A$ is an $A_\infty$-algebra. We adapt the formulas to the sign conventions we use.

\begin{defi}
  \label{sidecond-defi}
  Let $A \in \ch{\Bbbk}$. We say that a contraction 
  \begin{equation*}
    \begin{tikzcd}
      A \arrow[loop left, "h"] \arrow[r,shift left,"p"] \arrow[r,shift right, leftarrow, "s"'] & H^*(A), 
    \end{tikzcd}
  \end{equation*}
  satisfies the \textit{side conditions} if the following relations are satisfied:
  \[ ph = 0, \quad hs = 0, \quad h^2 = 0. \]
\end{defi}
\begin{rema}
  If $A$ is a $d$-strict filtered complex or a $d$-bistrict bifiltered complex, then the contractions constructed in the proofs of Lemma \ref{fil-contrac-lemma} and Lemma \ref{bifil-contrac-lemma} satisfy the side conditions.
\end{rema}

\begin{defi}
  Let $(A,\mu)$ be an associative algebra over $\Bbbk$ and 
  \begin{equation*}
    \begin{tikzcd}
      A \arrow[loop left, "h"] \arrow[r,shift left,"p"] \arrow[r,shift right, leftarrow, "s"'] & H^*(A), 
    \end{tikzcd}
  \end{equation*}
  a contraction. Define the \textit{$\mathfrak{p}$-kernels} by setting $\mathfrak{p}_2 = \mu$ and 
  \begin{align*}
    \mathfrak{p}_n = \sum_{\substack{k + l = n \\ k,l \geq 1}} (-1)^{k+1} \mu((h \circ \mathfrak{p}_k) \otimes (h \circ \mathfrak{p}_l))
  \end{align*}
  with the formal convention that $h \mathfrak{p}_1 = id$. Define also the \textit{$\mathfrak{q}$-kernels} by setting 
  \[ \mathfrak{q}_n = (-1)^n\mu((h \circ \mathfrak{q}_{n-1}) \otimes id) + \sum_{j=1}^{n-1} (-1)^{n} \mu( (sp)_j \otimes (h \circ \mathfrak{q}_{n-j})), \]
  with the convention that $\mathfrak{q}_1 = id$. Here, 
  \[ (sp)_m = (-1)^{m-1}sp \circ \mathfrak{q}_m + \sum_{B(m)} (-1)^{\frac{k(k+1)}{2}} (h \circ \mathfrak{q}_k)((sp \circ \mathfrak{q}_{r_1}) \otimes \cdots \otimes (sp \circ \mathfrak{q}_{r_k})), \]
  where 
  \[ B(m) = \{(k,r_1,...,r_k) \, | \, 2 \leq k \leq m, \, r_1, \ldots, r_k \geq 1, \, r_1 + \cdots r_k = m \}. \]
\end{defi}

\begin{rema}
  Note that if $A$ is filtered (or bifiltered) and the contraction is a filtered (or bifiltered) contraction, then the $\mathfrak{p}$-kernels and $\mathfrak{q}$-kernels are filtered (or bifiltered) maps. This leads to the following theorem.
\end{rema}

\begin{theo}
  \label{htt-assmhd-theo}
  Let $A = (A_\Bbbk,A_\CC,\varphi)$ be an associative mixed Hodge diagram and let 
  \begin{equation*}
    \begin{tikzcd}
      (A_\Bbbk,W) \arrow[loop left, "h_\Bbbk"] \arrow[r,shift left,"p_\Bbbk"] \arrow[r,shift right, leftarrow, "s_\Bbbk"'] & (H^*(A_\Bbbk),W), \\  (A_i,W) \arrow[loop left, "h_i"] \arrow[r,shift left,"p_i"] \arrow[r,shift right, leftarrow, "s_i"'] & (H^*(A_i),W), \\ (A_\CC,W,F) \arrow[loop left, "h_\CC"] \arrow[r,shift left,"p_\CC"] \arrow[r,shift right, leftarrow, "s_\CC"'] & (H^*(A_\CC),W,F)
    \end{tikzcd}
  \end{equation*}
  be filtered and bifiltered contractions satisfying the side conditions. For each $n \geq 1$, define
  \begin{alignat*}{2}
      &(m_\Bbbk)_n = p_\Bbbk \circ \mathfrak{p}_n \circ s_\Bbbk^{\otimes n}, \quad &&(P_\Bbbk)_n = p_\CC \circ \mathfrak{q}_n, \\
      &(S_\Bbbk)_n = h_\Bbbk \circ \mathfrak{p}_n \circ s_\Bbbk^{\otimes n}, \quad &&(H_\Bbbk)_n = h_\Bbbk \circ \mathfrak{q}_n,
  \end{alignat*}
  where $\mathfrak{p}_n$ and $\mathfrak{q}_n$ are defined with respect to the contraction $(s_\Bbbk,p_\Bbbk,h_\Bbbk)$. Do similarly for the intermediate filtered contractions (resulting in morphisms $m_i$, $P_i$, $S_i$ and $H_i$ for each $1 \leq i \leq s-1$) and the bifiltered contraction (resulting in morphisms $m_\CC$, $P_\CC$, $S_\CC$ and $H_\CC$). These maps together with the morphisms  
  \[ \hat{\varphi}_u = P_j \circledcirc \varphi_u \circledcirc S_i, \]
  defined for each arrow $u : i \to j$, make $(H^*(A),m_\Bbbk,m_\CC,\hat{\varphi})$ an $A_\infty\textrm{-}\mathsf{MHD}$. Moreover, the tuples $(P_\Bbbk, P_i, P_\CC, P_u)$ and $(S_\Bbbk,S_i, S_\CC)$, where 
  \[ P_u = P_j \circledcirc \varphi_u \circledcirc H_i \quad \text{and} \quad S_u = H_j \circledcirc \varphi_u \circledcirc S_i \]
  are ho-$\infty$-quasi isomorphisms:
  \[
    \begin{tikzcd}
      A_i \arrow[r,"P_i"] \arrow[d, "\varphi_u"] \arrow[dr, Rightarrow] & H^*(A_i) \arrow[d, "\hat{\varphi}_u"] \\
      A_j \arrow[r, "P_j"] & H^*(A_j),
    \end{tikzcd} \qquad
    \begin{tikzcd}
      H^*(A_i) \arrow[r,"S_i"] \arrow[d, "\hat{\varphi}_u"] \arrow[dr,Rightarrow] & A_i\arrow[d, "\varphi_u"] \\
      H^*(A_j) \arrow[r, "S_j"] & A_j.
    \end{tikzcd}
  \] 
\end{theo}
\begin{proof}
  Theorem $5$ of \cite{Markl-trans-ass} implies that $(H^*(A_\Bbbk),m_\Bbbk)$ defines an $A_\infty$-structure, 
  \[ S_\Bbbk : (H^*(A_\Bbbk),m_\Bbbk) \to A_\Bbbk \quad \text{and} \quad P_\Bbbk : A_\Bbbk \to (H^*(A_\Bbbk),m_\Bbbk)  \]
  define $\infty$-quasi-isomorphisms and $H_\Bbbk$ defines a gauge between $S_\Bbbk \circledcirc P_\Bbbk$ and $id_A$. The same follows for $(m_i,P_i,S_i,H_i)$ and $(m_\CC,P_\CC,S_\CC,H_\CC)$. Then, each $\hat{\varphi}_u = P_j \circledcirc \varphi_u \circledcirc S_i$ is a composition of $\infty$-morphisms and so it is an $\infty$-morphism as well. This shows that $(H^*(A),m_\Bbbk,m_\CC,\hat{\varphi})$ is an $A_\infty\textrm{-}\mathsf{MHD}$. Since $H_i$ is a gauge between $S_i P_i$ and $id_{A_i}$, it follows by Proposition \ref{comp-gauge-prop} that $P_u$ is a gauge between $\hat{\varphi}_u \circledcirc P_i$ and $P_j \circledcirc \varphi_u$. And similarly for $S_u$. Moreover, by the formulas, we have that 
  \[ (P_\Bbbk)_1 = p_\Bbbk \quad \text{and} \quad (S_\Bbbk)_1 = s_\Bbbk. \] 
  So $P_\Bbbk$ and $S_\Bbbk$ are $\infty$-quasi-isomorphisms. The same follows for $P_i, S_i, P_\CC$ and $S_\CC$. This proves the second claim of the theorem.
\end{proof}

Using this theorem, we can compute explicit formulas for the first obstructions to formality. To simplify the computations, let us further assume extra conditions on $A$ and the contractions. We assume also that $A$ is a mixed Hodge diagram of length $1$. That is, it is given by 
\[ (A_\Bbbk,W), \quad (A_\CC,W,F) \quad \text{and} \quad \varphi : (A_\Bbbk,W) \xrightarrow{\sim} (A_\CC,W). \]
These conditions are satisfied in the geometric examples we consider in the last section. In the following, we abuse notation and write $s_\Bbbk$ for $s_\Bbbk \otimes \CC$ in equations involving morphisms of $\CC$-modules and do the same for other morphisms defined over $\Bbbk$.
\begin{prop}
  \label{rep-phi3-prop}
  Let $A$ be an associative mixed Hodge diagram of length $1$ with $\alpha$-pure cohomology. Suppose that $(A_\Bbbk \otimes \CC) = A_\CC$ and $\varphi = id$. Denote by $\mu$ the product of $A_\CC$. Suppose further that there exist contractions
   \begin{equation*}
    \begin{tikzcd}
      (A_\Bbbk,W) \arrow[loop left, "h_\Bbbk"] \arrow[r,shift left,"p_\Bbbk"] \arrow[r,shift right, leftarrow, "s_\Bbbk"'] & (H^*(A_\Bbbk),W), \\ (A_\CC,W,F) \arrow[loop left, "h_\CC"] \arrow[r,shift left,"p_\CC"] \arrow[r,shift right, leftarrow, "s_\CC"'] & (H^*(A_\CC),W,F),
    \end{tikzcd}
  \end{equation*}
  satisfying the side conditions and such that $(s_\Bbbk \otimes \CC) = s_\CC$ and $(p_\Bbbk \otimes \CC) = p_\CC$. Denote $s = (s_\Bbbk \otimes \CC) = s_\CC$ and $p = (p_\Bbbk \otimes \CC) = p_\CC$. Then, $\hat{\varphi}_2 = 0$, $(m_\Bbbk)_3 = 0$ and $(m_\CC)_3 = 0$. Furthermore,
  \[ \hat{\varphi}_3 = p[\mu( id \otimes h_\CC h_\Bbbk \mu) - \mu(h_\CC h_\Bbbk \mu \otimes id) ]s^{ \otimes 3}.\]
\end{prop}
\begin{proof}
  Observe that the side conditions, together with the fact that $(s_\Bbbk \otimes \CC) = s_\CC = s$ and $(p_\Bbbk \otimes \CC) = p_\CC = p$ imply that $(P_\CC)_n s^{\otimes n} = 0$ and $p (S_\Bbbk)_n = 0$ for $n \geq 2$. Now note that
  \[ \hat{\varphi}_2 = (P_\CC S_\Bbbk)_2 = p (S_\Bbbk)_2 + (P_\CC)_2 s^{\otimes 2} = 0\] 
  Since $\delta \hat{\varphi}_2 = (m_\Bbbk)_3 - (m_\CC)_3$, we have that \[ (m_\Bbbk)_3 = (m_\CC)_3 : H^*(A)^{\otimes 3} \to H^*(A) \] are maps of mixed Hodge structures. Since $H^*(A)$ is $\alpha$-pure and the maps $(m_\Bbbk)_3$ and $(m_\CC)_3$ are of degree $-1$, it follows that they are maps of pure Hodge structures of different weights and are therefore zero. The map $\hat{\varphi}_3$ is given by
  \[ \hat{\varphi}_3 = (P_\CC S_\Bbbk)_3 = p (S_\Bbbk)_3 + (P_\CC)_2( s \otimes (S_\Bbbk)_2) - (P_\CC)_2((S_\Bbbk)_2,s) + (P_\CC)_3 s^{\otimes 3}. \]
  The first and last terms are trivial by our initial considerations. To compute the middle terms, observe that
  \begin{align*}
    &(P_\CC)_2 = p \circ \mathfrak{q}_2 = p(\mu( h_\CC \otimes id) + \mu(sp \otimes h_\CC)) \\
    &(S_\Bbbk)_2 = h_\Bbbk \mu s^{\otimes 2}.
  \end{align*}
  Thus, we have that
  \begin{align*}
    (P_\CC)_2 (s \otimes (S_\Bbbk)_2) = p(\mu(id \otimes h_\CC h_\Bbbk \mu) s^{\otimes 3}) \\
    (P_\CC)_2 ((S_\Bbbk)_2 \otimes s) = p(\mu( h_\CC h_\Bbbk \mu \otimes id) s^{\otimes 3})
  \end{align*} 
  and the formula for $\hat{\varphi}_3$ follows.
\end{proof}
Note that, in the conditions of this proposition, $\hat{\varphi}_3$ is a representative of the second obstruction $\phi_3$ of Theorem \ref{obst-col}.

\section[]{Applications to geometry}
\label{sec-appgeo}
In this last section, we apply the results of the previous sections regarding formality of mixed Hodge diagrams to complex algebraic varieties. First, we define mixed Hodge formality and mixed Hodge coformality of complex algebraic varieties and give examples where these two notions are equivalent. We show that the varieties of \cite{CCM} are not mixed Hodge formal and give examples of mixed Hodge formal varieties. We end by considering compact Kähler manifolds and relate the second obstruction to mixed Hodge formality two the ABC-Massey products defined by Angella and Tomassini \cite{DaniToma} for complex manifolds.

\subsection[]{Rational Homotopy of Algebraic Varieties}
By Theorem $9.3$ of \cite{navarro},
there is a functor
\[ \Aa : \mathrm{Var}_\CC \to \mathsf{MHD}[\mathrm{Qiso}^{-1}] \]
from complex algebraic varieties to commutative mixed Hodge diagrams over $\QQ$
such that, for $X \in \mathrm{Var}_\CC$, the cdga $\Aa_\QQ(X)$ is quasi-isomorphic to Sullivan's cdga of piecewise linear forms $\Aa_{pl}(X)$ of the underlying complex analytic space. Moreover, the composition of this functor with cohomology recovers Deligne's mixed Hodge theory \cite{DeHIII}. As in the additive case, the construction of this functor is based on Hironaka's theory on resolution of singularities and cohomological descent, and may be thought of a multiplicative enhancement of Deligne's original functor to mixed Hodge complexes.

\begin{rema}
  The same construction also applies to compact Kähler manifolds. There is a functor
\[ \Aa : \mathsf{Compact} \textrm{\hspace{1ex}}\mathsf{Kahler} \to \mathsf{MHD}, \]
such that $\Aa(X)_\QQ \simeq A_{pl}(X)$. Moreover, the complex filtered cdga $(\Aa(X)_\CC,F)$ is given by $\Aa(X)_\CC = \Aa_{dR}(X) \otimes \CC$ the algebra of complex de Rham forms on $X$ and $F$ is the Hodge filtration given by forms of $(p,q)$ type.
\end{rema}
\begin{defi}
  A complex algebraic variety $X$ is said to be \textit{mixed Hodge formal} if $\Aa(X)$ is formal as a commutative mixed Hodge diagram.
\end{defi}
Dually, we have
\begin{lemm}
  Let $X$ be a simply connected complex algebraic variety. Then, there is a $Lie$-mixed Hodge diagram $\Ll(X)$ whose underlying rational Lie algebra is Quillen's Lie model of $X$. Furthermore, the mixed Hodge structures on $H^*(\Ll(X))$ are isomorphic to the mixed Hodge structures on the homotopy groups $\pi_{*}(X)$.
\end{lemm}
\begin{proof}
  Recall that $\Aa(X)$ is quasi-isomorphic to a commutative mixed Hodge algebra (see Proposition \ref{Pmhalgmodels-prop}) whose underlying vector space is $\Omega_\kappa \mathrm{B}_\kappa(H^*(X))$. Let us denote such a model by $\Mm(X)$. Now note that any complex algebraic variety has cohomology of finite type. The fact that $X$ is simply connected implies that $\Mm(X)$ is of finite type and one can associate to $X$ a $Lie$-mixed Hodge diagram $\Ll(X) = s^{-1} \mathrm{B}_\kappa(\Mm(X))^{\vee}$. Since the underlying rational cdga of $\Mm(X)$ is quasi-isomorphic to $\Aa_{pl}(X)$, it follows that the underlying rational $Lie$-algebra of $\Ll(X)$ is quasi-isomorphic to Quillen's model of $X$ (see, for instance, Theorem $26.5$ of \cite{RatHomTheo}). As observed in Remark \ref{mhs-pi=morg}, $H^*(\Ll(X))$ and $\pi_*(X)$ are isomorphic as graded mixed Hodge structures.
\end{proof}
\begin{defi}
  A simply-connected complex algebraic variety $X$ is said to be \textit{mixed Hodge coformal} if $\Ll(X)$ is formal as a $Lie$-mixed Hodge diagram.
\end{defi}
Although every smooth complex projective variety is formal, there are smooth complex projective varieties which are not coformal in the classical sense. This follows, for instance, from Theorem $2$ of \cite{BergKoszul}, that says that a space $X$ is both formal and coformal if and only if it is formal and $H^*(X)$ is a Koszul algebra. For example, $\CC \PP^n$ is formal, but its cohomology is not Koszul for $n \geq 2$, hence it is not coformal. 

If the cohomology of a complex algebraic variety is Koszul and generated in some fixed degree, then we have the following direct application of Proposition \ref{form-coform-prop}.
\begin{theo}
  \label{form-coform-theo2}
  Let $X$ be a simply connected complex algebraic variety whose cohomology is $\alpha$-pure, Koszul and generated as an algebra by elements in a fixed degree $r \geq 2$. Then $X$ is mixed Hodge formal if and only if it is mixed Hodge coformal.
\end{theo}
Examples of complex algebraic varieties satisfying the conditions of Theorem \ref{form-coform-theo2} include the compactified moduli spaces $\overline{\Mm}_{0,n}$ of smooth genus $0$ curves with $n$ marked points (as proven by Dotsenko in \cite{M0nKoszul}), the configuration spaces $F_k(\CC^n)$ of $k$ points in $\CC^n$ and simply connected complex surfaces (see \cite{BergKoszul}).

\subsection[]{The example of Carlson-Clemens-Morgan}

In \cite{CCM}, Carlson, Clemens and Morgan give families of complex projective varieties $\{X_P\}$ which are all diffeomorphic and have the same pure Hodge structure on rational cohomology but have distinct mixed Hodge structures on $\pi_3 \otimes \QQ$. In particular, they cannot be mixed Hodge formal.
Let us briefly describe their method.

Given a simply connected complex projective variety (or more generally, compact Kähler manifold) $X$, the authors construct a short exact sequence of mixed Hodge structures of the form
\begin{equation*}
  0 \to H^3(X;\QQ) \to (\pi_3(X) \otimes \QQ)^* \to \Ker(\mu) \to 0, 
\end{equation*}
where $\mu : H^2(X) \otimes H^2(X) \to H^4(X)$ is the cup product. This short exact sequence might not split, let us denote by 
\[ u^*(\pi_3) \in \mathrm{Ext}^1_{\MHS}(\Ker(\mu), H^3(X;\QQ)) \] 
the obstruction to its splitting.
Applying this to $\{X_P\}$, they show that the restriction of $u^*(\pi_3)$ to $\mathrm{Ext}^1_{\MHS}(K^{alg},H^3(X_P))$, where 
\[ K^{alg} = \Ker(\mu) \cap (H^2(X_P;\QQ) \cap H^{1,1}(X_P))^{\otimes 2}, \] 
differs between several elements of the family $\{X_P\}$.

\begin{rema}
  In \cite{CCM}, the authors work with mixed Hodge structures defined over torsion-free abelian groups, but the same constructions work for rational vector spaces as well.
\end{rema}

Recall $\theta_3$ the first obstruction to formality of Theorem \ref{obst-theo} and the map $R$ defined in (\ref{map-R}). Observe that $u^*(\pi_3)$ is precisely the obstruction to the splitting of (\ref{mhs-seq1}). Hence, we have the following corollary of Proposition \ref{first-obs}:

\begin{coro}
  \label{ccmobs-coro}
  Let $X$ be a simply connected compact Kähler manifold. Then, 
  \begin{equation*}
    u^*(\pi_3)  = R(\theta_3) \in \mathrm{Ext}^1_{\MHS}(\Ker(\mu), H^3(X;\QQ)).
  \end{equation*}
\end{coro}
It follows that, in the examples of \cite{CCM}, the manifolds $X_P$ for which $u^*(\pi_3) \neq 0$ are not mixed Hodge formal. 

\subsection{Examples of mixed Hodge formal varieties}
Let us now see some examples of mixed Hodge formal complex varieties.
\begin{exam}
  \label{ex-CPn}
  The cohomology of $\CC \PP^n$ is $H^*(\CC \PP^n;\QQ) = \Lambda(x)/(x^{n+1})$, where $|x| = (1,1)$. A model for $\Aa(\CC \PP^n)$ is given by the mixed Hodge dg-algebra $\Mm = \Lambda(x,y)$, where $|x| = 2$, $|y| = 2n + 1$, $dy = x^{n+1}$ and $x$ has weight $2$ while $y$ has weight $2n+2$. In $\Mm_\CC = \Mm \otimes \CC$, $x$ has (highest) Hodge filtration $1$ while $y$ has Hodge filtration $n+1$.
  The morphism $ \Mm \to H^*(\CC \PP^n;\QQ) $ given by $x\mapsto x$ and $y\mapsto 0$
  is a quasi-isomorphism of mixed Hodge algebras, so $\CC \PP^n$ is mixed Hodge formal. For $n \geq 2$, $\CC \PP^n$ is not mixed Hodge coformal since it is not even coformal (this follows from Theorem $2$ of \cite{BergKoszul} and the fact that $H^*(\CC \PP^n)$ is not Koszul for $n \geq 2$).
\end{exam}

\begin{exam}
  Any complex variety with free cohomology algebra is mixed Hodge formal. This includes, for instance, $\CC^n \backslash \{0\} \simeq S^{2n-1}$ and complex tori so, in particular, complex abelian varieties. To see this, let $X$ be a complex algebraic variety such that $H^*(X)$ is free as an algebra. Then, there exists a rational map $f_\QQ : H^*(X) \to \Aa(X)_\QQ$ obtained by choosing representatives for each generator of $H^*(X)$. Similarly, there exists a map $f_\CC : H^*(X) \otimes \CC \to \Aa(X)_\CC$ preserving the Hodge filtration. On generators, $f_\QQ$ and $f_\CC$ differ by a boundary, so one can also construct a homotopy $H^*(X) \otimes \CC \to \Aa(X)_\CC \otimes I$  from $\varphi_{\Aa(X)} f_\QQ$ to $f_\CC$. Thus, there exists a ho-quasi-isomorphism (recall Definition \ref{ho-morph-defi}) $f : H^*(X) \to \Aa(X)$ and by Proposition \ref{homorph-fact}, $\Aa(X)$ is formal.
\end{exam}

\begin{exam}
Given complex algebraic varieties $X,Y$, the product of models $\Aa(X) \otimes \Aa(Y)$ is a model for the product $X \times Y$. Hence, the product of mixed Hodge formal complex varieties is mixed Hodge formal.
\end{exam}

The previous examples were shown to be mixed Hodge formal by computing explicit models. Although the obstructions to formality are hard to compute in general, one can use them in some cases to show mixed Hodge formality. The proof of the following theorem is inspired in \cite{intformConfspace}, where Salvatore proves intrinsic formality of the configuration spaces $F_k(\RR^n)$, with coefficients in a ring, for $k \leq n$.
\begin{prop}
  \label{confspaceform-prop}
  The space $F_{k}(\CC^n)$ of configurations of $k$ points in $\CC^n$ is mixed Hodge formal and mixed Hodge coformal for $k < 2n$.
\end{prop}
\begin{proof}
  First note that $F_k(\CC^n)$ is a Koszul space in the sense of Berglund (see section 5 of \cite{BergKoszul}). It has cohomology algebra generated in degree $2n-1$ so, by Theorem \ref{form-coform-theo2}, $F_k(\CC^n)$ is mixed Hodge formal if and only if it is mixed Hodge coformal. To check formality, we show that all the obstructions to formality of Theorem \ref{obst-col} are trivial. Recall the cochain complex computing $\Pp H^*_{\mathrm{Ext}}(H^*(F_k(\CC^n)))$ (where the obstructions live)
  \begin{equation*}
      \mathrm{Ext}_{\MHS}^1(A,M) \xrightarrow[]{\delta} \mathrm{Ext}_{\MHS}^1(\Pp^\antishriek(2) \otimes_{\SS_2} A^{\otimes 2},M) \xrightarrow[]{\delta} Ext_{\MHS}^1(\Pp^\antishriek(3) \otimes_{\SS_3} A^{\otimes 3},M) \to \cdots.
  \end{equation*}
  This cochain complex comes from a cosimplicial set with $m$-simplices
  \begin{equation*}
    \mathrm{Ext}_{\MHS}^1(\Pp^\antishriek(m) \otimes_{\SS_m} H^*(F_k(\CC^n))^{\otimes m},H^*(F_k(\CC^n)))
  \end{equation*}
  by taking the alternating sum of the coface maps. The cohomology of this complex is thus isomorphic to the cohomology of the complex obtained by taking first the kernel of the codegeneracy maps. But the codegeneracy maps of this cosimplicial set are given by precomposing with 
  \begin{equation*}
    s^m_i : (H^*(F_k(\CC^n)))^{\otimes m-1} \to (H^*(F_k(\CC^n)))^{\otimes i-1} \otimes H^0(F_k(\CC^n)) \otimes (H^*(F_k(\CC^n)))^{\otimes m-i} 
  \end{equation*}
  that inserts a unit in the $i$-th position. Hence, the representatives of the obstruction to formality $[\psi_t]$ can be taken to be morphisms of the form $f : H^*(F_k(\CC^n))^{\otimes t} \to H^*(F_k(\CC^n))$ of degree $1-t$ (see Proposition \ref{mh-ext}) and which are zero on summands with at least one term $H^0(F_k(\CC^n))$ in the tensor product. We shall see that for $k < 2n$, every such morphism is trivial. Now, the cohomology of $F_k(\CC^n)$ is generated by classes of degree $2n-1$ and has top degree $(k-1)(2n-1)$.
  Thus, $H^{>0}(F_k(\CC^n))^{\otimes s}$ has classes with minimum degree $s(2n-1)$, so a morphism of degree $1-s$ has image in $H^{\geq s(2n-1)+1-s}(F_k(\CC^n))$. For $k < 2n$, the first multiple of $n-1$ greater or equal to $s(2n-1)+1-s$ is greater than $(k-1)(2n-1)$, the top cohomological degree. Hence, there cannot exist nontrivial morphisms $f : H^{>0}(F_k(\CC^n))^{\otimes s} \to H^*(F_k(\CC^n))$ of degree $1-s$.
\end{proof}

Another family of examples comes from varieties whose cohomology is of complete intersection type. A sequence of polynomials $r_1,\dots,r_m \in \Bbbk[x_1,\dots,x_n]$ is said to be \textit{regular} if for each $1 \leq i \leq m$, the class of $r_i$ in the quotient
\[ \Bbbk[x_1,\dots,x_n] / (r_1,\dots,r_{i-1}) \]
is not a zero-divisor.
\begin{defi}
  A graded commutative algebra $H$ is \textit{of complete intersection type} if it is given by
  \[ H = \Bbbk[x_1,\dots,x_n] /(r_1,\dots,r_m), \]
  with $x_i$ in positive even degrees and $r_1,\dots,r_m \in \Bbbk[x_1,\dots,x_n]$ a regular sequence.
\end{defi}

\begin{lemm}
  Any commutative mixed Hodge diagram $A$ with $1$-pure cohomology of complete intersection type is formal.
\end{lemm}
\begin{proof}
  Identify $H^*(A_\Bbbk) \otimes \CC = H^*(A_i) = H^*(A_\CC)$  through the comparison morphisms and denote these algebras by $H^*(A_\CC)$. For each arrow $u : i \to j$, we may assume that $\varphi_u^* = id$. Denote the product of $H^*(A)$ by $\mu^*$. By Theorem \ref{htt-mhd}, there is $H = (H^*(A),m_\Bbbk,m_\CC,\hat{\varphi})$ a minimal $C_\infty$-model of $A$. Since $H^*(A)$ is $1$-pure, there exist filtered and bifiltered $\infty$-isotopies
  \begin{align*}
    &f_\Bbbk : (H^*(A_\Bbbk),m_\Bbbk,W) \to (H^*(A_\Bbbk),\mu^*,W), \\
    &f_i : (H^*(A_i),m_i,W) \to (H^*(A_i),\mu^*,W), \\
    &f_\CC : (H^*(A_\CC),m_\CC,W,F) \to (H^*(A_\CC),\mu^*,W,F).
  \end{align*} 
  Composing the transferred structure with such morphisms, it follows that $H$ is $\infty$-isomorphic to a $C_\infty$-mixed Hodge diagram whose underlying $C_\infty$-algebras are $(H^*(A_\Bbbk),\mu^*)$ and $(H^*(A_\CC),\mu^*)$. Then, $A$ is formal if there exists an $\infty$-isotopy 
  \[ (H^*(A),\mu^*,\mu^*,\hat{\varphi}) \to (H^*(A),\mu^*,\mu^*,id). \]
  Denote by $[\hat{\varphi}]$ the automorphism of $(H^*(A_\CC),\mu^*)$ in the homotopy category of filtered cdga's given by the zig-zag of comparison morphisms. Then, $A$ is formal if $[\hat{\varphi}]$ is the identity. We suppose, for simplicity, that $A$ has length $1$. The comparison morphism
  \[ \hat{\varphi} : (H^*(A),W) \to (H^*(A),W) \]
  is composed of morphisms $\hat{\varphi}_n$ of degree $1-n$ for $n \geq 1$. Since $H^*(A)$ is $1$-pure, every such morphism automatically preserves the weight filtration. Hence, the set of filtered $\infty$-automorphisms of $H^*(A)$ up to gauge equivalence coincides with the set of $\infty$-automorphisms of $H^*(A)$ (forgetting the weight filtration) up to gauge equivalence. The latter is in bijection with the set of automorphisms of $(H^*(A),\mu^*)$ in the homotopy category of cdga's. Since 
  \[ H^*(A_\CC) = \CC[x_1,\dots,x_n] / (r_1,\dots,r_m) \]
  is of complete intersection type, it follows that a Sullivan minimal model of $H^*(A_\CC)$ is given by
  \[ M = \Lambda(x_1,\dots,x_n) \otimes \Lambda(y_1,\dots,y_m), \]
  with $d(x_i) = 0$ and $d(y_i) = r_i$. This follows from standard results on regular sequences and their Koszul complexes (see \cite{Eisenbud}) The only automorphism of $M$ that induces the identity on $H^*(A_\CC)$ is the identity of $M$. Therefore, the comparison morphism $\hat{\varphi}$ is gauge equivalent to the identity and $A$ is formal.
\end{proof}

Every complex Grassmannian $Gr_k(\CC^n)$ and, more generally, complex flag manifolds, have cohomology of complete intersection type. In \cite{StelGioZoller}, the authors prove the analogous result for the notion of strong formality of bigraded bidifferential algebras and complex manifolds. Following their argument, we have
\begin{prop}
  \label{homogKahler-prop}
  Any homogeneous compact Kähler manifold is mixed Hodge formal.
\end{prop}
\begin{proof}
  Any homogeneous compact Kähler manifold $X$ is a product $T \times F$, where $T$ is a complex torus and $F$ is a flag manifold (see \cite{BorelRein}). Complex tori are mixed Hodge formal, flag manifolds have cohomology of complete intersection type (see \cite{Borel} \cite{BorelHirz}) and so are also mixed Hodge formal. Finally, the product of mixed Hodge formal varieties is mixed Hodge formal and the result follows.
\end{proof}

\subsection{Mixed Hodge diagrams of Kähler manifolds}
\label{cmpxman-sec}

We now review basic definitions of complex geometry and classical Hodge theory and give a functor from compact Kähler manifolds to real mixed Hodge diagrams. The complexified de Rham algebra $\Aa = \Aa_{\mathrm{dR}}(X) \otimes \CC$ of every complex manifold $X$ admits a decomposition into $(p,q)$-forms
\[ \Aa^n = \bigoplus_{p+q = n} \Aa^{p,q}, \qquad \overline{\Aa^{p,q}} = \Aa^{q,p}. \]
Its differential splits into $d = \del + \delb$ of bidegrees $(1,0)$ and $(0,1)$ respectively.
A particular class of complex manifolds for which mixed Hodge structures arise naturally in homotopy is that of compact Kähler manifolds. For these manifolds, the complexified de Rham algebra satisfies the so-called $\del\delb$-condition:

\begin{defi} A double complex $(\Aa,\del,\delb)$ is said to satisfy the \textit{$\del\delb$-condition} if
 \[\Ker(\delb)\cap \Img(\del)=\Ker(\del)\cap \Img(\delb)=\Img(\del\delb).\]
\end{defi}
There are several equivalent conditions that characterize the $\del\delb$-condition (see \cite{DGMS}). Most importantly for our applications, the $\del\delb$-condition is equivalent to asking that:
\begin{enumerate}[(i)]
 \item both spectral sequences, associated to the row and column filtrations of the double complex, degenerate at $E_1$, and
 \item the two induced filtrations on the degree-$n$ cohomology induce a pure Hodge structure of weight $n$.
\end{enumerate}
In particular, by considering the Real de Rham algebra $ \Aa_{\mathrm{dR}}(X)$ together with the canonical filtration,
as well as the Hodge filtration
  \[ F^p \Aa_\CC^n(X)= \bigoplus_{i \geq p} \Aa^{i, n-i}\]
  on the complexified de Rham algebra $\Aa_\CC(X) = \Aa_{\mathrm{dR}}(X) \otimes \CC$,
we obtain a functor
\[\Aa_\RR : \mathsf{compact} \; \mathsf{K}\ddot{\mathsf{a}}\mathsf{hler} \to \mathsf{MHD}_\RR\]
from compact Kähler manifolds to real mixed Hodge diagrams, where the comparison morphisms are identities.
\begin{defi}
  A compact Kähler manifold $X$ is said to be \textit{mixed Hodge formal over $\RR$} if $\Aa_\RR(X)$ is a formal mixed Hodge diagram.
\end{defi}

\begin{rema}
In the above construction, we could additionally consider the data of $\Aa_{pl}(X)$ and the quasi-isomorphism $\Aa_{pl}(X)\otimes\RR\simeq \Aa_{\mathrm{dR}}(X)$, to obtain a functor with values in the category of mixed Hodge diagrams over $\QQ$. Note that, in this case, mixed Hodge formality over $\QQ$ implies mixed Hodge formality over $\RR$. As shown by the algebraic Example \ref{exam-fieldext}, the converse is not true.
\end{rema}

In \cite{DaniToma}, Angella and Tomassini introduced a notion of triple Massey product that is sensitive to the complex analytic structure of a complex manifold. This uses \textit{Bott-Chern} an \textit{Aeppli} cohomologies, defined respectively by
\[ H_{BC}(X) := \frac{\Ker(\del) \cap \Ker(\delb)}{\Img(\del \delb)}, \qquad H_A(X) := \frac{\Ker(\del \delb)}{\Img(\del) + \Img(\delb)}. \]
The inclusion induces natural maps
\[
  \begin{tikzcd}
  & H_{BC}(X) \arrow[ld] \arrow[d] \arrow[rd]& \\
  H_\del(X) \arrow[dr] & H_{dR}(X;\CC) \arrow[d] & H_{\delb}(X) \arrow[ld] \\
  & H_A(X) & 
\end{tikzcd} 
\]
connecting Bott-Chern and Aeppli cohomologies with Dolbeault and anti-Dolbeault cohomologies (defined as the cohomologies with respect to $\delb$ and $\del$ respectively).
\begin{rema}
The $\del\delb$-condition is equivalent to asking that the map $H_{BC}(X) \to H_\delb(X)$ is an isomorphism (which in turn, is equivalent to asking that all of the above maps are isomorphisms).
\end{rema}
Note that the wedge product of differential forms makes $H_{BC}(X)$ into a cdga and $H_A(X)$ an $H_{BC}(X)$-module.
\begin{defi}[\cite{DaniToma}]
  Let $X$ be a complex manifold and \[ [a] \in H^{k_1}_{BC}(X), \quad [b] \in H^{k_2}_{BC}(X) \quad \text{and} \quad [c] \in H^{k_3}_{BC}(X) \] such that $[a][b] = [b][c] = 0$. The \textit{triple ABC Massey product} of these classes is the class
  \[ \langle [a],[b],[c] \rangle_{ABC} = [a y - x c],\]
  seen as an element in the quotient
  \[ H^{k_1+k_2+k_3-2}_{A}(X) / ([a]H^*_A(X) + H^*_A(X)[c]), \]
  where $ab = i \del \delb x$ and $bc = i \del \delb y$ for some $x \in \Aa_{dR}^{k_1+ k_2 - 2}(X) \otimes \CC$ and $y \in \Aa_{dR}^{k_2 + k_3 -2}(X) \otimes \CC$.
\end{defi}

\begin{rema}
    Triple ABC Massey products are obstructions to another notion of formality, defined for complex manifolds, which is called strong formality (see \cite{StelMili}). As observed in \cite{Stel-pluri}, if a compact Kähler manifold is strongly formal, then it is also mixed Hodge formal over $\RR$. Examples of such include, for instance, compact Kähler manifolds of dimension $n \geq 2$ with the Hodge diamond of a complete intersection \cite{Stel-pluri} and smooth proper toric varieties \cite{StelGioZoller}.
\end{rema}

We now show that when $X$ is a compact Kähler manifold, the second obstruction to mixed Hodge formality of Theorem \ref{obst-col} computes ABC-Massey products. To do this, we use canonical filtered and bifiltered contractions, introduced in \cite{CiHo2}, for compact Kähler manifolds.

Assume that $X$ is compact Kähler. The Kähler metric induces a Hodge-star operator
\[ \star : \Aa^{p,q} \to \Aa^{m-p,m-q} \textrm{\quad defined by \hspace{1ex}} \alpha \wedge \star \overline{\beta} = \langle \alpha , \beta \rangle \mathrm{vol},\]
where vol is the volume form determined by the metric and $m$ is the complex dimension of $X$. Denote by $\delta$ either of the operators $\del$ or $\delb$. Then, the operator $\delta^* = - \star \delta \star$ is the $\Ll_2$-adjoint of $\delta$ and $d^* = \del^* + \delb^*$. Defining $\Delta_\delta = \delta^* \delta + \delta \delta^*$, Hodge theory gives orthogonal decompositions
\[ \Aa^{p,q} = \Hh_\delta^{p,q} \oplus \Delta_\delta(\Aa^{p,q}),\]
where $\Hh_\delta^{p,q} = \Ker(\Delta_\delta) \cap \Aa^{p,q}$ is the space of harmonic forms of type $(p,q)$. There are also isomorphisms 
\[ \Hh_\delta^{p,q} \cong \frac{\ker (\delta)}{\Img (\delta)} |(p,q). \]
The Laplacian identities $\Delta_d = 2 \Delta_\del = 2 \Delta_\delb$ identify all spaces of harmonic forms 
\[ \Hh_d^{p,q} = \Hh_\del^{p,q} = \Hh_\delb^{p,q}, \] 
which we shall denote by $\Hh^{p,q}$. These induce the $1$-pure Hodge structure on $H^*(X)$
\[ H^n(X) \cong \bigoplus_{p+q=n} \Hh^{p,q}. \]
Denote by $\pi : \Aa^{p,q} \to \Hh^{p,q}(X)$ the projection, by $\iota : \Hh^{p,q}(X) \hookrightarrow \Aa^{p,q}$ the inclusion and by $G_\delta : \Aa^{p,q} \to \Aa^{p,q}$ the Green operator. The latter is defined by $0$ on $\delta$-harmonic forms and by the inverse of $\Delta_\delta$ on the orthogonal complement. The following is proved in Lemma $4.1$ of \cite{CiHo2}.
\begin{lemm}
  Define the maps $h_\RR = d^* G_d$ and $h_\CC = \delb^* G_\delb$. Then, the diagrams 
  \[ 
  \begin{tikzcd}
    (\Aa_{dR}(X),W) \arrow[loop left, "h_\RR"] \arrow[r,shift left,"\pi"] \arrow[r,shift right, leftarrow, "\iota"'] & (H^*(X;\RR),W), \\ (\Aa,W,F) \arrow[loop left, "h_\CC"] \arrow[r,shift left,"\pi"] \arrow[r,shift right, leftarrow, "\iota"'] & (H^*(X;\CC),W,F),
  \end{tikzcd}
  \]
  define real filtered and complex bifiltered contractions satisfying the side conditions.
\end{lemm}
Note that the same maps $\iota$ and $\pi$ appear in both contractions. The conditions of Proposition \ref{rep-phi3-prop} are thus satisfied and so the resulting $A_\infty$-transferred structure 
\[ (H^*(X),m_\RR,m_\CC,\hat{\varphi}) \] 
(see Theorem \ref{htt-assmhd-theo}) satisfies $((m_\RR)_3,(m_\CC)_3,\hat{\varphi}_2) = (0,0,0)$ and 
\[ \hat{\varphi}_3([a],[b],[c]) = \pi( a \wedge h_\CC h_\RR (b \wedge c) - h_\CC h_\RR(a\wedge b) \wedge c),\]
for $[a],[b],[c] \in H^*(X)$. In the formula above, we denote by $a,b,c$ the harmonic representatives of the classes $[a],[b],[c] \in H^*(X)$.
\begin{rema}
  The vanishing of the second obstruction is also a consequence of the general fact that $\Aa(X)$ is a mixed Hodge diagram with real coefficients and $H^*(X)$ is $1$-pure (see Proposition \ref{firstobswelldef-prop}).
\end{rema}
We now compare the obstruction class $\phi_3$ (represented by $\hat{\varphi}_3$) with ABC-Massey products. Given $H$ a real pure Hodge structure of weight $n$, the $p$-th intermediate Jacobian of $H$ for $2p \geq n$ is given by:
\[ J^p H = \frac{H \otimes \CC}{ H + F^p H \otimes \CC}. \]
Consider also the space of Hodge $p$-classes
\[ \mathrm{Hdg}(H,p) := F^p (H \otimes \CC)^{2p} \cap H. \]
Given
\[ [a] \in \mathrm{Hdg}^{2p_1}(X,p_1), \quad [b] \in \mathrm{Hdg}^{2p_2}(X,p_2), \quad [c] \in \mathrm{Hdg}^{2p_3}(X,p_3) \]
such that $[ab] = 0 = [bc]$, let $p = p_1 + p_2 + p_3$ and denote by $\pi^p$ the composition
\begin{align*}
    \frac{H^{2p-2}_{A}(X)}{[a]H^*(X) + H^*(X)[c]} \xrightarrow{\pi} &\frac{H^{2p-2}(X;\CC)}{[a]H^*(X) + H^*(X)[c]}  \\ &\hspace{4ex}\to\frac{J^p H^{2p-2}(X)}{[a]J^{p_2+p_3} H^*_A(X) + J^{p_1 + p_2} H^*_A(X)[c]} 
\end{align*} 
\begin{rema}
  Note that the first map in the composition above is well-defined since $\pi : \Aa^{p,q} \to \Hh^{p,q}$ induces a map of $H_{BC}$-modules
  \begin{align*}
    \pi : H_A^{p,q}(X) &\to H^{p,q}(X) \\
    [\omega] &\mapsto [\pi(\omega)].
  \end{align*}
  This map is the inverse to the natural map $H^*(X) \to H_A^*(X)$. 
\end{rema}

\begin{rema}
  \label{pureTatepip-rema}
  Note that when $H^*(X)$ has pure Tate Hodge structures, meaning 
  \[ H^{2p}(X;\CC) = H^{p,p} \quad \text{and} \quad H^{2p-1}(X;\CC) = 0, \] 
  then $F^p H^{2p-2}(X) = 0$. This implies that 
  \[ J^{p} H^{2p-2}(X) = \frac{H^{2p-2}(X;\CC)}{H^{2p-2}(X;\RR)} \cong i \cdot H^{2p-2}(X;\RR). \]
  In this case, $\pi^p$ is the projection onto the imaginary part.
\end{rema}

Denote by $\mathrm{HH}_\mathrm{Ext}^*(H^*(X))$ the $\mathrm{Ext}\textrm{-}\Pp$-cohomology of $H^*(X)$ (see Definition \ref{oper-cohom}) for $\Pp = Ass$.
\begin{theo}
  \label{obst3=abc-theo}
  There is a well-defined map 
  \begin{align*}
    \mathrm{HH}_{\mathrm{Ext}}^{3}(H^*(X)) &\xrightarrow{\mathrm{ev}} \frac{J^p H^{k_1+k_2+k_3-2}(X)}{[a]J^{p_2+p_3} H^*(X) + J^{p_1 + p_2} H^*(X)[c]} \\
    [f] &\mapsto f([a] \otimes [b] \otimes [c]).
  \end{align*}
  Moreover, the second obstruction $\phi_3$ of Theorem \ref{obst-col} is mapped to
  \[ -2 \cdot \mathrm{ev}(\phi_3) = \pi^p (i \cdot \langle [a], [b], [c] \rangle_{ABC}). \]
\end{theo}
\begin{proof}
  To prove the first claim, we show that for \[ [f] = 0 \in \mathrm{HH}_{\mathrm{Ext}}^{3}(H^*(X)), \] it follows that $\mathrm{ev}([f]) = 0$. By Proposition \ref{mh-ext}, a class \[ [f] \in \mathrm{HH}_{\mathrm{Ext}}^{3}(H^*(X)) \] is represented by a morphism
  \[ f : H^*(X;\CC)^{\otimes 3} \to H^*(X;\CC) \]
  and if the class is trivial, there exist a real morphism \[ g_\RR : H^*(X)^{\otimes 3} \to H^*(X), \] a complex filtered morphism \[ g_\CC : (H^*(X;\CC)^{\otimes 3},F) \to (H^*(X;\CC),F)\] and a complex morphism \[ h_2 : H^*(X;\CC)^{\otimes 2} \to H^*(X;\CC) \] such that
  \[ f = g_\RR \otimes \CC + g_\CC + \delta h_2. \]
  Here, $\delta$ is the Hochschild differential. Then,
  \begin{align*}
      &g_\RR([a] \otimes [b] \otimes [c]) \in H^*(X;\RR), \\
      &g_\CC([a] \otimes [b] \otimes [c]) \in F^p H^*(X;\CC), \\
      &\delta h_2 ([a] \otimes [b] \otimes [c]) = (-1)^{|a||h_2|} [a] h_2([b] \otimes [c]) - h_2([a] \otimes [b]) [c].
  \end{align*}  
  This implies that $\mathrm{ev}([f]) = 0$ and so $\mathrm{ev}$ is well-defined. The obstruction $\phi_3$ is represented by the map
  \[ \hat{\varphi}_3([a],[b],[c]) = \pi( a \wedge h_\CC h_\RR (b \wedge c)- h_\CC h_\RR(a\wedge b) \wedge c), \]
  for $[a],[b],[c] \in H^*(X)$,where $a,b,c$ denote the harmonic representatives. Thus, to prove the second claim, it is sufficient to check that 
  \begin{align*}
    &2 \del \delb h_\CC h_\RR(a b) = ab,  \\
    &2 \del \delb h_\CC h_\RR (b c) = bc,
  \end{align*} 
  Any operator that commutes with $\Delta_\delta$ also commutes with $G_\delta$. Also, by the Laplacian identities, one has $2 G_d = G_\del = G_\delb$. Together with the identities $[\del,\delb^*] = [\delb,\del^*] = 0$, we obtain that 
  \begin{align*}
    &[\delb, h_\CC] = id - \iota \pi, \textrm{\qquad \hspace{1.1ex}} [\del,h_\CC] = 0, \\
    &[\delb, h_\RR] = \frac{1}{2}(id - \iota \pi), \textrm{\quad} [\del,h_\RR] = \frac{1}{2}(id - \iota \pi).
  \end{align*}
  It then follows that 
  \begin{align*}
    2 \del \delb h_\CC h_\RR &= - 2 \delb \del h_\CC h_\RR = 2 \delb h_\CC \del h_\RR \\ &= 2 (id - \iota \pi - h_\CC \delb)(\frac{1}{2}(id - \iota \pi) - h_\RR \del)  \\
    &= id - \iota \pi - h_\CC \delb (id - \iota \pi) + (...) \del \\
    &= id - \iota \pi - (...) \delb + (...)\del.
  \end{align*}
  The two last terms applied to $a b$ are trivial. Also, as $[a][b] = 0$, it follows that $\iota \pi (ab) = 0$. And the same goes for $bc$.
\end{proof}
The previous theorem allows us to borrow examples from \cite{Steletall-nonformal}, in which the authors compute non-trivial ABC-Massey products, to obtain examples of non-mixed Hodge formal manifolds.

\begin{coro}
  \label{nonformblowup-coro}
  For any compact Kähler manifold $Y$ of dimension at least $4$ and with pure Tate Hodge structures on cohomology, there is a finite sequence of blow-ups of $Y$ at points and lines such that the resulting manifold $\tilde{Y}$ is not mixed Hodge formal. 
\end{coro}
\begin{proof}
  In \cite{Steletall-nonformal}, it is proved that for any compact Kähler manifold $Y$ of dimension at least $4$, there such a sequence of blow ups $\tilde{Y}$ with a non-trivial ABC Massey product. Note that blowing up at points and lines preserves the property of having pure Tate Hodge structures on cohomology. Note also that, as observed in Remark \ref{pureTatepip-rema}, if $H^*(\tilde{Y})$ is pure Tate, then $\pi_p$ is the identity when applied to purely imaginary classes. It then follows that, if $[a], [b]$ and $[c]$ are Hodge classes such that $[a][b] = 0 = [b][c]$, then $\langle [a], [b], [c] \rangle_{ABC}$ is a real class and so Theorem \ref{obst3=abc-theo} implies that
  \[ -2 \cdot \mathrm{ev}(\phi_3) = \pi^p (i \cdot \langle [a], [b], [c] \rangle_{ABC}) = i \cdot \langle [a], [b], [c] \rangle_{ABC}. \]
  Hence, if $Y$ has pure Tate Hodge structures on cohomology, it follows that $\mathrm{ev}(\phi_3)$ is not trivial. By Proposition \ref{firstobswelldef-prop}, $\tilde{Y}$ is not mixed Hodge formal.
\end{proof}

\bibliographystyle{alpha}
\bibliography{bibliography}

\end{document}